\documentclass[11pt]{amsart}
\usepackage[dvipsnames,usenames]{color}
\usepackage{hyperref}
\usepackage{graphicx}
\usepackage{epsfig}
\usepackage[latin1]{inputenc}
\usepackage{amsmath}
\usepackage{amsfonts}
\usepackage{amssymb}
\usepackage{amsthm}
\usepackage{amscd}
\usepackage{verbatim}
\usepackage{caption}
\usepackage{subcaption}
\usepackage{pinlabel}
\usepackage{stmaryrd}
\usepackage{enumerate, enumitem}
\usepackage{todonotes}
\usepackage{bm}
\usepackage{thmtools}
\usepackage{thm-restate}

\usepackage[foot]{amsaddr}

\usepackage{tikz}
\usetikzlibrary{arrows}
\usetikzlibrary{decorations.pathreplacing}
\usepackage{verbatim}
\usetikzlibrary{cd}

\newtheorem{theorem}{Theorem}[section]

\newtheorem{lemma}[theorem]{Lemma}
\newtheorem{proposition}[theorem]{Proposition}

\newtheorem{corollary}[theorem]{Corollary}

\theoremstyle{definition}
\newtheorem{definition}[theorem]{Definition}

\newtheorem{question}[theorem]{Question}
\theoremstyle{remark}
\newtheorem{remark}[theorem]{Remark}
\newtheorem{example}[theorem]{Example}

\def\F{\mathbb{F}}
\def\N{\mathbb{N}}

\def\Z{\mathbb{Z}}

\def\K{\mathbb{K}}

\def\cC{\mathcal{C}}

\def\cR{\mathcal R}
\def\cS{\mathcal S}
\def\cB{\mathcal{B}}

\def\Im{\operatorname{im}}

\def \gr {\operatorname{gr}}

\def\d{\partial}
\def\varep{\varepsilon}
\def\co{\colon}

\def\id{\textup{id}}

\def\im{\operatorname{im}}

\def\CF {\mathit{CF}}

\def\HF {\mathit{HF}}
\def\HFK {\mathit{HFK}}
\newcommand\HFhat{\widehat{\HF}}
\newcommand\CFhat{\widehat{\CF}}

\newcommand \CFm {\CF^-}
\newcommand \HFm {\HF^-}

\def\CFK{\mathit{CFK}}

\def\Hom{\mathrm{Hom}}

\def\Chatz{\smash{\widehat{\mathcal{C}}_\mathbb{Z}}}
\def\cCz{\smash{\mathcal{C}_{\mathbb{Z}}}}
\def\cC{\mathcal{C}}

\newcommand{\shift}{\mathrm{sh}} 

\def\lebang{<^!}
\def\gebang{>^!}
\def\leqbang{\leq^!}
\def\geqbang{\geq^!}

\def\KL{\mathfrak{K}} 
\def\KLs{\mathfrak{K}_{\mathrm{sym}}}
\def\sgn{\operatorname{sgn}}

\newcommand{\X}{\mathbb{X}}

\def\CFKUV{\CFK_\cR} 
\def\CFKX{\CFK_\X} 
\def\bfx{\mathbf{x}}

\def\Span{\operatorname{Span}}

\newcommand{\fm}{\mathfrak{m}}

\newcommand{\loc}{\mathcal{L}}
\newcommand{\R}{\mathcal{R}}
\newcommand{\ru}{\mathcal{R}_U} 
\newcommand{\rv}{\mathcal{R}_V} 

\newcommand{\stair}{\mathcal{S}}

\author[I. Dai]{Irving Dai}
\address {Department of Mathematics, The University of Texas at Austin, Austin, TX 78712, USA}
\email{irving.dai@math.utexas.edu}

\author[J. Hom]{Jennifer Hom}
\address {School of Mathematics, Georgia Institute of Technology, Atlanta, GA 30332, USA}
\email{hom@math.gatech.edu}

\author[M. Stoffregen]{Matthew Stoffregen}
\address {Department of Mathematics, Michigan State University, East Lansing, MI 48824, USA}
\email{stoffre1@msu.edu}

\author[L. Truong]{Linh Truong}
\address {Department of Mathematics, University of Michigan, Ann Arbor, MI 48103, USA}
\email{tlinh@umich.edu}

\numberwithin{equation}{section}

\title[Homology concordance]{Homology concordance and knot Floer homology}

\begin{document}

\begin{abstract}
We study the homology concordance group of knots in integer homology three-spheres which bound integer homology four-balls. Using knot Floer homology, we construct an infinite number of $\Z$-valued, linearly independent homology concordance homomorphisms which vanish for knots coming from $S^3$. This shows that the homology concordance group modulo knots coming from $S^3$ contains an infinite-rank summand. The techniques used here generalize the classification program established in previous papers regarding the local equivalence group of knot Floer complexes over $\F[U, V]/(UV)$. Our results extend this approach to complexes defined over a broader class of rings.
\end{abstract}

\maketitle

\section{Introduction}\label{sec:intro} 
Beginning with the $\tau$-invariant \cite{OS4ball}, the knot Floer homology package of \linebreak Ozsv\'ath-Szab\'o \cite{OSknots} and independently J.\ Rasmussen \cite{RasmussenThesis} has had numerous applications to the study of smooth knot concordance, especially of knots in $S^3$. See \cite{Homsurvey} for a survey of such applications. Less investigation has been done so far on knot concordance in general manifolds, although there have been many recent results with bearing here, including relative adjunction inequalities constraining the genera of smoothly embedded surfaces in a four-manifold \cite{Manolescu-Marengon-Piccirillo, Hedden-Raoux}, invariants of almost concordance classes of knots in lens spaces \cite{Celoria}, and the existence of knots which are not homology concordant to any knot in $S^3$ \cite{Levine-PL}.  

In this paper, we study the homology concordance group $\Chatz$ generated by pairs $(Y, K)$, where $Y$ is an integer homology sphere bounding an acyclic smooth $4$-manifold and $K$ is a knot in $Y$. Two classes $(Y_1,K_1)$ and $(Y_2, K_2)$ are equivalent in $\Chatz$ if there is a homology cobordism from $Y_1$ to $Y_2$ in which $-K_1 \sqcup K_2$ bounds a smoothly embedded annulus. Let $\cCz$ denote the set of knots in $S^3$ modulo homology concordance; this admits an inclusion into $\Chatz$ whose image consists of all classes $(Y, K)$ for which $K$ is homology concordant to a knot in $S^3$. Note that $\cCz$ is a quotient of the usual smooth concordance group $\cC$.\footnote{It is not known if $\cC \to \cCz$ is an isomorphism.}

The starting point for this paper is the work of Levine, Lidman, and the second author \cite{HomLidmanLevine}, which established that $\Chatz/\cCz$ is infinitely generated and contains a $\mathbb{Z}$-subgroup.  This was later extended by Zhou \cite{Zhou} to show that $\Chatz/\cCz$ contains a $\Z^\infty$-subgroup. The present paper gives an infinite family of homomorphisms from $\Chatz/\cCz$ to $\mathbb{Z}$:

\begin{theorem}\label{thm:main} 
For each $(i,j) \in (\mathbb{Z}\times \mathbb{Z}^{\geq 0}) - (\mathbb{Z}^{\leq 0} \times \{0\})$, there is a homomorphism 
\[ \varphi_{i,j} \co \Chatz \to \Z. \]
For classes in $\cCz \subset \Chatz$, all homomorphisms of the form $\varphi_{i,j}$ with $j \neq 0$ vanish. Hence these descend to homomorphisms 
\[ \varphi_{i,j} \co \Chatz/\cCz \to \Z. \]
Moreover, for knots in $S^3$, the remaining homomorphisms $\varphi_{i,0}$ agree with the homomorphisms $\varphi_{i}$ defined in \cite{DHSTmoreconcord}. 
In addition, 
\[ \bigoplus_{n > 1}  \varphi_{n,n-1} \co \Chatz/\cCz \to \bigoplus_{n > 1}  \Z \]
is surjective.
\end{theorem}

\noindent
In fact, the homomorphisms $\varphi_{i,j}$ are well-defined concordance homomorphisms for knots in any integer homology sphere, not just $Y$ bounding an acyclic $4$-manifold.

To see that the $\varphi_{n,n-1}$ are nonvanishing, we use the knots considered by Zhou in \cite{Zhou} (see Section \ref{sec:computations}); in particular, a consequence of the proof of Theorem~\ref{thm:main} is that the subgroup $\mathbb{Z}^\infty\subset \Chatz/\cCz$ constructed in \cite{Zhou} is a summand.

\begin{corollary}
The group $\Chatz/\cCz$ has a $\Z^\infty$-summand.
\end{corollary}

The invariants $\varphi_{i,j}$ factor through the local equivalence group of knot Floer complexes (\cite[Theorem 1.5]{Zemkeconnsuminv}, forgetting the involutive part) and are related to stable equivalence from \cite[Theorem 1]{Homsurvey}; equivalently $\nu^+$-equivalence of \cite{KimPark}. To understand this, recall that following \cite[Section 3]{Zemkeconnsuminv}, the knot Floer complex can be viewed as a module over $\F[U, V]$ where $\F$ is the field of two elements. Local equivalence is then an equivalence relation between certain such complexes. One can also consider an analogous notion of local equivalence for complexes defined over other rings; see for example \cite{DHSThomcobord} and \cite{DHSTmoreconcord}.

In certain situations, it is possible to obtain a complete classification of knotlike complexes up to local equivalence. This is (roughly speaking) the approach underlying \cite{DHSThomcobord} and \cite{DHSTmoreconcord}. Here, we carry out this program for a new ring $\X$, which turns out to be useful for studying homology concordance. In fact, we provide a more general class of rings for which the our algebraic classification results hold. This class includes $\X$ together with the ring $\R = \F[U, V]/(UV)$ considered in \cite{DHSTmoreconcord}. Indeed, the bulk of this paper consists of identifying the appropriate structural features of $\R$ that will allow us to generalize the arguments of \cite{DHSTmoreconcord} to the present setting.

Explicitly, $\X$ is defined to be the commutative $\F$-algebra 
\[
\frac{
\F[U_B, \{W_{B,i}\}_{i\in \mathbb{Z}}, V_T, \{W_{T,i}\}_{i\in \mathbb{Z}}]} { (U_BV_{T}, \ U_BW_{B,i}-W_{B,i+1}, \ V_TW_{T,i}-W_{T,i+1})}.
\]
We give this a bigrading $\gr=(\gr_1,\gr_2)\in \mathbb{Z}\times \mathbb{Z}$ by setting
\[
\gr(U_B) = (-2, 0) \qquad \text{and} \qquad \gr(W_{B,i}) = (-2i, -2)
\]   
and
\[
\gr(V_T) = (0, -2) \qquad \text{and} \qquad \gr(W_{T, i}) = (-2, -2i).
\]
The definition of $\X$ is meant to take into account $\R=\F[U,V]/(UV)$-local equivalence from \cite[Definition 1]{Homconcordance} and \cite{DHSTmoreconcord}, while involving the Floer homology of the ambient three-manifold. 

Understanding $\X$ is rather involved; an extended discussion can be found in Section~\ref{sec:preliminary}. For now, we make two remarks. The first is that there is a morphism of unital bigraded $\F$-algebras $\F[U,V]\to \X$ given by
\[U\mapsto U_B+W_{T,0} \quad \text{ and } \quad V\mapsto V_T+W_{B,0}.\] 
Given a knot $K$ in a homology sphere $Y$, we produce an $\X$-complex simply by taking the usual knot Floer complex over $\F[U, V]$ and performing the substitution above: 
\[
\CFK_\X(Y, K) = \CFK(Y, K) \otimes_{\F[U,V]} \X.
\]
The second is that the elements of $\X$ can be put in correspondence with certain lattice points in $\Z \times \Z$. (See Definition~\ref{def:rureformulation} for a precise discussion.) In Section~\ref{sec:standards}, we construct a preferred family of $\X$-complexes, each of which is parameterized by a finite sequence of lattice points, thought of as elements of $\X$. Our main structural theorem is that every knotlike $\X$-complex is locally equivalent to a unique complex in this family. Moreover, the total order introduced by the second author in \cite{Hominfiniterank} extends to induce a total order for complexes over the ring $\X$. We show that this total order can be understood by defining a certain total order $\leqbang$ on $\Z \times \Z - \{(0,0)\}$ and considering the lexicographic order on the set of sequences induced by $\leqbang$. See Definition~\ref{def:latticetotalorder} for a definition of $\leqbang$.

\begin{theorem}\label{thm:localequivlex}
Every knot Floer complex $\CFKX(Y,K)$ for a knot $K$ in an integer homology sphere $Y$ is locally equivalent (up to grading shift) to a standard complex as defined in Definition~\ref{def:standard}. Each standard complex is uniquely determined by a finite sequence of nonzero pairs of integers $(i_k,j_k)_{k=1}^{2n}$ which obeys the symmetry condition $(i_k, j_k) = -(i_{2n+1-k}, j_{2n+1-k})$. Moreover, local equivalence classes over $\X$ are ordered according to the lexicographic order on sequences $(i_k,j_k)_{k=1}^{2n}$ induced by the order $\leqbang$ of Definition~\ref{def:latticetotalorder}.
\end{theorem}

Note that if we let $D = (\mathbb{Z}\times \mathbb{Z}_{\geq 0}) - (\mathbb{Z}_{\leq 0} \times \{0\})$ be the region of definition for the indices of the homomorphisms $\varphi_{i,j}$, then $\Z \times \Z - \{(0,0)\} = D \sqcup -D$. Roughly speaking, $\varphi_{i,j}(K)$ will be the signed count of how many parameters $(i_k, j_k)$ in the standard complex representative of $\CFKX(Y, K)$ are equal to $\pm (i, j)$. See Section~\ref{sec:preliminary} for an example.


Our approach to Theorems \ref{thm:main} and \ref{thm:localequivlex} is to work for a certain type of rings, called \emph{grid rings}, and to establish classification results for all rings of this type.  The theorems for knot Floer homology are established as a consequence of these general algebraic results, specialized to the ring $\X$. See Section~\ref{sec:knotlike}.


\subsection{Topological applications}

We now list some connections between our homomorphisms $\varphi_{i,j}$ and other invariants, together with some topological applications.

\begin{restatable}{proposition}{proptau}\label{prop:tau}
	Let $K$ be a knot in an integer homology sphere $Y$. Then we have the following equality relating the Ozsv\'ath-Szab\'o $\tau$-invariant with $\varphi_{i,j}$:
	\[ \tau(Y, K) = \sum_{(i,j) } (i-j)\varphi_{i,j}(K). \]\end{restatable}

Similarly, $\varep(Y,K)$ can be determined from the standard complex representative of $\CFKX(Y,K)$ appearing in Theorem \ref{thm:localequivlex}. This leads to a re-proof of the following:

\begin{corollary}{\cite[Corollary 1.9]{HomLidmanLevine}}\label{cor:1.9}
Let $K$ be a knot in an integer homology sphere $Y$.  If $\varep(Y,K)=0$ and $\tau(Y,K)\neq 0$, then $K\subset Y$ is not homology concordant to any knot in $S^3$.
\end{corollary}

In fact, Corollary 1.5 can be extended to obstruct homology concordances to any knot in an integer homology sphere $L$-space.

We also are also able to show that if $Y$ is a Seifert fiber space, then the $\varphi_{i,j}$ are constrained. Our conventions on the sign of Seifert fibered spaces follow \cite{Karakurt-Lidman-Tweedy}; in particular \emph{positive} Seifert fiber spaces $Y$ bound negative-definite plumbings and have $\mathit{HF}^-_{\mathrm{red}}(Y)$ concentrated in even degrees, and negative Seifert fiber spaces $Y$ have $\mathit{HF}^-_{\mathrm{red}}(Y)$ concentrated in odd degrees. 

\begin{restatable}{proposition}{propseifert}\label{prop:seifert}
Let $K\subset Y$ be a knot in an integer homology sphere. If there is some $(i,j)$ with $j>0$ for which $\varphi_{i,j}(K)>0$, then $K$ is not homology concordant to a knot in any negative Seifert fiber space.  If there is some $(i,j)$ with $j>0$ for which $\varphi_{i,j}(K)<0$, then $K$ is not homology concordant to a knot in any positive Seifert fiber space. 
\end{restatable}	

The following corollary is immediate from Theorem \ref{thm:main} and Proposition \ref{prop:seifert}:

\begin{corollary}
There exist pairs $(Y, K)$ which are not homology concordant to any knot in a Seifert space.
\end{corollary}
\begin{proof}
According to Theorem~\ref{thm:main} we may select a class on which the homomorphisms $\varphi_{n, n-1}$ take any desired set of values.
\end{proof}

We can also bound various genera in line with our results on concordance unknotting number and concordance genus (of knots in $S^3$) in \cite{DHSTmoreconcord}. Define the homology-concordance genus 
\[
g_{H,c}(Y,K)=\min_{(Y',K') \mbox{ homology concordant to } (Y,K)} g_3(Y',K'),
\]
and the homology-concordance unknotting number:
\[
u_{H,c}(Y,K)=\min_{(Y',K') \mbox{ homology concordant to } (Y,K)} u(Y',K').
\]
Note that the homology-concordance unknotting number may be infinite.  Let 
\begin{equation*}
N(Y, K)=\sup_{(i,j)\mid \varphi_{i,j}(K)\neq 0} |i-j|,
\end{equation*}

\begin{proposition}\label{prop:genera-bounds}
The homology concordance genus and homology concordance unknotting numbers satisfy:
\begin{enumerate}
\item $g_{H,c}(K)\geq N(K)/2$,
\item $u_{H,c}(K)\geq N(K)$.
\end{enumerate}
\end{proposition}
\noindent
Proposition~\ref{prop:genera-bounds} should be compared with \cite[Theorem 1.14]{DHSTmoreconcord}.


\begin{remark}
Following the appearance of this paper on the arXiv, Zhou \cite{Zhou2bridge} used the invariants $\varphi_{i,j}$, blow-downs of two-bridge links, and cables to show that $\Chatz/\cCz$ admits a $\Z^\infty$-summand generated by a family of knots in a single manifold.
\end{remark}

\subsection*{Organization}
In Section \ref{sec:background}, we briefly recall the requisite properties of knot Floer homology. We give a brief primer on $\X$-complexes and their applications (delaying proofs) in Section \ref{sec:preliminary}. In Section \ref{sec:knotlike}, we introduce the notion of a \textit{grid ring} and a \textit{knotlike complex} over a grid ring. This will allow us to define the local equivalence group $\KL$ of knotlike complexes, and show that it is totally ordered. In Section \ref{sec:standards}, we define the \emph{standard complexes} in analogy to the construction over $\R=\F[U,V]/(UV)$ from \cite{DHSTmoreconcord}. In Sections~\ref{sec:ai} and \ref{sec:char}, we dive into the technical heart of the paper and show that every knotlike complex is locally equivalent to a standard complex. We then establish our homomorphisms in Section \ref{sec:homs}. The applications to knot Floer theory are given in Sections~\ref{sec:algebra-for-knot-floer} and \ref{sec:vanishing}, with calculations of the relevant knot Floer complexes in Section~\ref{sec:computations}. \\ 
\\
\noindent
\textbf{Acknowledgements.} Irving Dai was supported by National Science Foundation grant DGE-1148900. Jennifer Hom was supported by National Science Foundation grants DMS-1552285 and DMS-2104144. Matthew Stoffregen was supported by National Science Foundation grant DMS-1952762. Linh Truong was supported by National Science Foundation Grants DMS-1606451, DMS-2005539, and DMS-2104309. The authors would like to thank the Park City Mathematics Institute for hosting us during the July 2019 Research Program: \emph{Quantum Field Theory and Manifold Invariants}, where part of this work was completed. We also thank the MATRIX Institute for hosting three of us during the January 2019 workshop: \emph{Topology of manifolds: Interactions between high and low dimensions}.  We are grateful to Artem Kotelskiy and Hugo Zhou for helpful conversations. Lastly, we thank the referee for many helpful comments.

\section{Background on knot Floer homology}\label{sec:background}

In this section, we give a brief overview of knot Floer homology, primarily to establish notation. We assume that the reader is familiar with knot Floer homology as in \cite{OSknots} and \cite{RasmussenThesis}; see \cite{ManolescuKnotIntro}, \cite{Homsurvey}, and \cite{HomPCMI} for survey articles on this subject. Our conventions mostly follow those in \cite{Zemkeabsgr}; see, in particular, Section 1.5 of \cite{Zemkeabsgr}, see also Section 2 of \cite{DHSTmoreconcord}.

We consider the bigraded ring $\F[U,V]$, with bigrading $\gr=(\gr_U,\gr_V)$, where we call $\gr_U$ the \emph{$U$-grading} and $\gr_V$ the \emph{$V$-grading}, and $\gr(U)=(-2,0)$ and $\gr(V)=(0,-2)$. When other bigraded rings and maps are present, we will often write $\gr_1 = \gr_U$ and $\gr_2 = \gr_V$.

Let $K$ be a nullhomologous knot in a closed, oriented three-manifold $Y$. Then we may associate to $(Y, K)$ a chain complex $\CFK_{\F[U,V]}(Y,K)=\CFK(Y,K)$, called the \emph{knot Floer complex} of the pair $(Y,K)$.  This complex is obtained by considering the free module over $\F[U,V]$ of intersection points of two Lagrangians in a symmetric product of a Riemann surface, coming from a Heegaard diagram along with some analytic input, and where the differential counts isolated holomorphic disks, weighted by their intersection with two basepoints. 

The complex $\CFK(Y,K)$ is naturally bigraded by two gradings, also called $\gr_U$ and $\gr_V$, compatible with the action of $\F[U,V]$ on $\CFK(Y,K)$. The differential $\d$ of $\CFK(Y,K)$ has bidegree $(-1,-1)$.  The Alexander grading of a homogeneous $x\in \CFK(Y,K)$ is defined by $A(x)= (\gr_U(x)-\gr_V(x))/2$. The complex $\CFK(Y,K)$ depends on the data involved in its construction but is an invariant of $(Y,K)$ up to $\F[U,V]$-equivariant homotopy equivalence.  

We now recall some facts from \cite{OSknots}. Firstly, we have the following symmetry property. Let $\overline{\CFK}(Y,K)$ denote the complex obtained by interchanging the roles of $U$ and $V$. (Note that we thus also interchange the values of $\gr_U$ and $\gr_V$.) Then
\[ \CFK(Y,K) \simeq \overline{\CFK}(Y,K). \]
The knot Floer complex behaves nicely with respect to connected sums. Indeed, we have that
\[ \CFK((Y_1,K_1) \# (Y_2,K_2)) \simeq \CFK(Y_1, K_1) \otimes_{\F[U,V]} \CFK(Y_2 K_2). \]
We also have that
\[ \CFK(-(Y,K)) \simeq \CFK(Y, K)^\vee, \]
where $\CFK(Y, K)^\vee = \Hom_{\F[U,V]}(\CFK(Y, K), \F[U,V])$.

In the literature, there is a more common version of the knot Floer complex which takes the form of a filtered chain complex $\CFK^{\infty}(Y,K)$ over $\F[U, U^{-1}]$.  This is generated over $\F$ by tuples $[\bfx,i,j]$ with $\bfx$ an intersection point, and $i,j$ integers so that $A(\bfx)-j+i=0$, where $A$ is the Alexander grading.  In the setting of $\CFK(Y,K)$, the element $[\bfx,i,j] \in \CFK^\infty(Y,K)$ corresponds to $U^{-i}V^{-j}\bfx$. With our present notation, there is thus an identification
\[
\CFK^\infty(Y,K)=((U, V)^{-1}\CFK(Y,K))_0
\]
where the righthand side denotes the $\F[\hat{U}]$-submodule of $(U, V)^{-1}\CFK(Y,K)$ in Alexander grading $0$, where $\smash{\hat{U}=UV}$; note that multiplication by $\smash{\hat{U}}$ preserves the Alexander grading.  On $(U, V)^{-1}\CFK(Y,K)$, by construction $\gr_U=\gr_V$, and so the complex has a natural ``Maslov grading" given by either of $\gr_U$ or $\gr_V$. 

We also consider how $\CFK(Y,K)$ relates to other standard versions of knot Floer homology. First consider the $\F$-vector space $\smash{\widehat\HFK(Y, K)}$, which is defined by not allowing holomorphic disks in the definition of $\d$ to cross either the $w$ or the $z$ basepoint. In our context, this is isomorphic to $H_*(\CFK(Y,K)/(U, V))$, where $(U,V)$ denotes the ideal generated by $U$ and $V$. The Alexander grading on $\smash{\widehat{HFK}}$ is given as before by $\smash{A = (\gr_U-\gr_V)/2}$ and the Maslov grading is given by $M = \gr_U$. 

Next, consider the $\F[U]$-module $\HFK^-(Y, K)$, which is defined by taking the homology of the associated graded complex of $\CFK^-(Y, K)$ with respect to the Alexander filtration. This is equivalent to allowing holomorphic disks to cross the $w$ but not the $z$ basepoint. In our context, this yields $H_*(\CFK(Y,K)/V)$, where the Alexander and Maslov gradings are as before. It is a standard fact that for knots in $S^3$, the $\F[U]$-module $\HFK^-(S^3, K) \cong H_*(\CFK(S^3,K)/V)$ has a single $U$-nontorsion tower.\footnote{By this, we mean that $H_*(\CFK(S^3,K)/V)/U\text{-torsion} \cong \F[U]$. Note, however, that this copy of $\F[U]$ is not required to be generated by an element with $\gr_U = 0$.} 

It will also be convenient to have a description of knot Floer homology once some of the variables are inverted.  Indeed, $\CFK(Y,K)\otimes_{\F[U,V]}  \F[U,V,V^{-1}]$, where $\F[U,V,V^{-1}]$ is an $\F[U,V]$-module in the natural way, is bigraded homotopy equivalent over $\F[U,V,V^{-1}]$ to $\CFm(Y)\otimes_\F \F[V,V^{-1}]$, where $\CFm(Y)$ is regarded as a $\F[U]$-module concentrated in $\gr_V$-grading zero, and $\CFm(Y)\otimes_\F \F[V,V^{-1}]$ is regarded as a $\F[U,V]$-complex by letting $V$ act on the $\F[V,V^{-1}]$ factor.  By symmetry, we similarly have a description of $\smash{\CFK(Y,K)\otimes_{\F[U,V]} \F[U,U^{-1}, V]}$ as bigraded homotopy equivalent to $\smash{\CFm(Y)\otimes_\F \F[U,U^{-1}]}$. 

If $Y$ is an integer homology sphere, then $\CFm(Y)$ is a $\mathbb{Z}$-graded $\F[U]$-chain complex, with 
\[
\CFm(Y)\otimes_{\F[U]}  F[U,U^{-1}]\simeq \F[U,U^{-1}],
\]
where $\simeq$ denotes homotopy equivalence. In particular, if $K\subset Y$ is a knot in an integer homology sphere, then there is an $\F[U,V]$-equivariant homotopy equivalence
\[
\CFK(Y,K)\otimes_{\F[U, V]} \F[U,U^{-1},V,V^{-1}]\simeq \F[U,U^{-1},V,V^{-1}]. 
\] 

The following definition is particularly useful in applications of knot Floer homology to homology concordance:

\begin{definition}\label{def:loceq}
Let $K_1$ and $K_2$ be knots in integer homology three-spheres $Y_1$ and $Y_2$ respectively. We say that $\CFK(Y_1,K_1)$ and $\CFK(Y_2,K_2)$ are \emph{locally equivalent} if there exist absolutely $U$-graded, absolutely $V$-graded $\F[U,V]$-equivariant chain maps
\[ f \co \CFK(Y_1, K_1) \to \CFK(Y_2, K_2) \quad \text{ and }  \quad g \co \CFK(Y_2, K_2) \to \CFK(Y_1, K_1) \]
such that $f$ and $g$ induce homotopy-equivalences
\[
f\otimes \id \co \CFK(Y_1, K_1)\otimes \F[U,U^{-1},V,V^{-1}] \to \CFK(Y_2, K_2)\otimes \F[U,U^{-1},V,V^{-1}] \] and \[ \quad g\otimes \id \co \CFK(Y_2, K_2)\otimes \F[U,U^{-1},V,V^{-1}] \to \CFK(Y_1, K_1)\otimes \F[U,U^{-1},V,V^{-1}] .
\]
\end{definition}

In previous work \cite{DHSTmoreconcord}, the authors studied knot Floer homology over the ring $\R = \F[U, V]/(UV)$ and defined an appropriate notion of local equivalence for such complexes. 

\begin{theorem}[{\cite[Theorem 1.5]{Zemkeconnsuminv}, cf.\ \cite[Theorem 2]{Homconcordance}}] \label{thm:concle}
If $K_1$ and $K_2$ are two knots in $S^3$ which are concordant, then $\CFKUV(K_1)$ and $\CFKUV(K_2)$ are locally equivalent.
\end{theorem}
\noindent
Theorem~\ref{thm:concle} follows from \cite[Theorem 1.5]{Zemkeconnsuminv} by forgetting the involutive component and quotienting by $UV$, or from \cite[Theorem 2]{Homconcordance} by translating from $\varep$-equivalence and bifiltered chain complexes to local equivalence and $\cR$-modules.

\begin{remark}
Note that $\CFKUV(Y,K)$ is locally equivalent to $\CFKUV(O)$, where $O$ denotes the unknot, if and only if  $\CFKUV(Y, K) \simeq \CFKUV(O) \oplus A$, where $A$ is a chain complex over $\cR$ with $U^{-1}H_*(A) = V^{-1}H_*(A) = 0$. It is straightforward to verify that local equivalence over $\cR$ and $\varep$-equivalence (see \cite[Section 2]{Hominfiniterank}) are the same (after translating between $\cR$-modules and bifiltered chain complexes over $\F[U, U^{-1}]$).
\end{remark}

\section{Overview and preliminary notions}\label{sec:preliminary}

\subsection{Preliminary discussion of $\X$-complexes}

In \cite{DHSTmoreconcord}, the authors defined the notion of a knotlike complex over $\R = \F[U,V]/(UV)$ and gave a classification of such complexes up to local equivalence. The present paper should be thought of as extending this classification to complexes defined over other rings. Our primary example will be the ring $\X$ mentioned in Section~\ref{sec:intro}, since this will have applications to studying knots in homology spheres other than $S^3$. We spend a little time familiarizing the reader with the properties of $\X$, before moving on to consider a general class of rings in the next section. It will be helpful for the reader to have a broad recollection of the ideas of \cite{DHSTmoreconcord}.

We begin by imprecisely recalling some salient features of the case of complexes over $\R$. Although obvious, these should be kept in mind, as they will generalize to properties of complexes over $\X$.
\begin{enumerate}
\item First note that $\R$ contains two subrings $\F[U]$ and $\F[V]$, which have preferred maximal ideals $(U)$ and $(V)$, respectively. (The sense in which these are preferred will soon be made explicit.) The product of these ideals is zero in $\R$. The reader familiar with the results of \cite{DHSTmoreconcord} will broadly recall the importance of this point: arrows in the standard complexes of \cite[Section 4.1]{DHSTmoreconcord} are labeled by homogenous elements of $(U)$ and $(V)$, with the labels of successive arrows being drawn first from one ideal and then the other. For the manipulations in \cite{DHSTmoreconcord}, it is important that the product of $(U)$ and $(V)$ is zero.
\item There is an obvious ordering on the set of homogenous elements of $\F[U]$, where $x$ is larger than $y$ if(f) $x$ divides $y$. A similar statement holds for $\F[V]$. The ordering on the set of standard complexes (see \cite[Section 4.3]{DHSTmoreconcord}) is (roughly speaking) induced from the ordering on the homogenous elements of $\F[U]$ and $\F[V]$. 
\item Let $K$ be a knot in $S^3$ and let $\CFK_{\R}(S^3, K)$ be the knot Floer complex of $K$ over $\R$. Inverting $U$ and taking the homology of the resulting complex yields a single bi-infinite tower $\F[U, U^{-1}]$. The notion of a local map in \cite[Definition 3.3]{DHSTmoreconcord} is dependent on the fact that the localization of the homology at $U$ takes this particularly simple form. A similar statement holds for localizing at $V$.
\end{enumerate}
As we will see, appropriately formalizing each of these properties will suffice to establish a general classification result. To give the reader some intuition, we first draw some parallels between $\X$ and $\R$. Recall:

\begin{definition} \label{def:xring}
Let $\X$ be the $\F$-algebra given by
\[
\X = \frac{
\F[U_B, \{W_{B,i}\}_{i\in \mathbb{Z}}, V_T, \{W_{T,i}\}_{i\in \mathbb{Z}}]} { (U_BV_{T}, \ U_BW_{B,i}-W_{B,i+1}, \ V_TW_{T,i}-W_{T,i+1})},
\]
with a bigrading $\gr = (\gr_1, \gr_2)\in \mathbb{Z}\times \mathbb{Z}$ given by
\[
\gr(U_B) = (-2, 0) \qquad \text{and} \qquad \gr(W_{B,i}) = (-2i, -2)
\]   
and
\[
\gr(V_T) = (0, -2) \qquad \text{and} \qquad \gr(W_{T, i}) = (-2, -2i).
\]
Note that $\X$ has two particular subrings, which we denote by $\R_U$ and $\R_V$. These are
\[
\R_U=\F[U_B,\{W_{B,i}\}_{i\in \Z}]/(\{U_BW_{B,i}=W_{B,i+1}\})
\] 
and
\[
\R_V=\F[V_T,\{W_{T,i}\}_{i\in \Z}]/(\{V_TW_{T,i}=W_{T,i+1}\}).
\]
There is an obvious skew-graded isomorphism between $\R_U$ and $\R_V$, sending $U_B$ to $V_T$ and $W_{B,i}$ to $W_{T,i}$.
\end{definition}

The algebraic structure of $\R_U$ and $\R_V$ is displayed in Figure~\ref{fig:Xring}. Once again, $\R_U$ has a preferred maximal ideal $(U_B)$ and $\R_V$ has a preferred maximal ideal $(V_T)$; the product of these is zero in $\X$. Note that every nontrivial (that is, non-identity) homogeneous element of $\R_U$ is a multiple of $U_B$ (and similarly for $\R_V$), and that the intersection of $\R_U$ and $\R_V$ in $\X$ is a copy of the field $\F$. As in the case of $\F[U]$, the homogenous elements of $\R_U$ are totally ordered by divisibility. Visually, this total order can be seen on the left of Figure~\ref{fig:Xring} as follows: moving along a red arrow goes from one homogenous element to the next-largest homogenous element, with the caveat that all of the elements in the $(n+1)$st row are less than (or equivalently, divisible by) any element in the $n$th row. A similar statement holds for homogenous elements of $\R_V$.

\begin{figure}[h!]
\includegraphics[scale = 0.7]{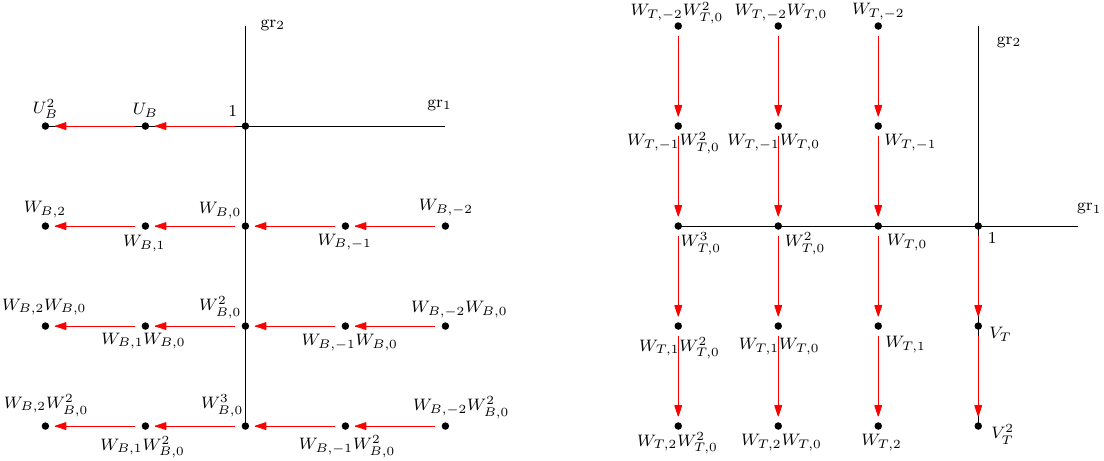}
\caption{The homogenous generators (over $\F$) of $\R_U$ (left) and $\R_V$ (right), displayed in the $(\gr_1, \gr_2)$-plane. Red arrows are drawn as a visual aid to represent the action of $U_B$ (left) and $V_T$ (right), but may also be interpreted as helping describe the total order on the set of homogenous elements. Note (for example) that in $\R_U$, the first row dominates the second row; that is, any $\smash{W_{B, i}}$ is divisible by any $\smash{U_B^j}$.}\label{fig:Xring}
\end{figure}

\begin{definition}\label{def:rureformulation}
We may also write the elements of $\R_U$ uniquely as $\smash{U_B^i W_{B,0}^j}$, where
\begin{equation}\label{eq:ruregion}
(i, j) \in (\Z \times \Z^{\geq 0}) - (\Z^{< 0} \times \{0\}).
\end{equation}
Note that $i$ is required to be non-negative if $j = 0$, but we otherwise allow negative powers of $U_B$. The element $\smash{U_B^i W_{B,0}^j}$ is the unique nonzero element of $\R_U$ in bigrading $(-2i, -2j)$. We thus identify elements of $\R_U$ as points $(i, j)$ in the subregion \eqref{eq:ruregion}; this is (up to a factor of two) just the set of lattice points on the left of Figure~\ref{fig:Xring}, rotated 180 degrees. For elements of $\R_V$, we write $(i, j)$ to represent $\smash{V_T^iW_{T,0}^j}$, where $(i, j)$ similarly lies in \eqref{eq:ruregion}. In this case, $(i, j)$ is the unique nonzero element of $\R_V$ in bigrading $(-2j, -2i)$. The obvious skew-graded isomorphism between $\R_U$ and $\R_V$ sends $(i, j)$ in $\R_U$ to $(i, j)$ in $\R_V$.\footnote{We could instead use the convention that $(i,j)$ in $\R_V$ represents $\smash{V_T^jW_{T,0}^i}$. However, it will be notationally convenient for the elements of $\R_U$ and $\R_V$ to both correspond to points in the same subregion of $\Z \times \Z$.} Note that \eqref{eq:ruregion} is (excluding the origin) exactly the same as the index domain on which the $\varphi_{i,j}$ in Theorem~\ref{thm:main} are defined.
\end{definition}

We also need to consider the \textit{field of homogenous fractions} of $\R_U$ and $\R_V$, which is obtained by inverting all (nonzero) homogenous elements. We denote these by $\mathcal{L}(\R_U)$ and $\mathcal{L}(\R_V)$, respectively. It is straightforward to check that 
\begin{align*}
\mathcal{L}(\R_U)& = \F[U_B, U_B^{-1}, W_{B, 0}, W_{B, 0}^{-1}] \\
\mathcal{L}(\R_V)& = \F[V_T, V_T^{-1}, W_{T, 0}, W_{T, 0}^{-1}].
\end{align*}
As in Definition~\ref{def:rureformulation}, we may think of the elements of $\mathcal{L}(\R_U)$ or $\mathcal{L}(\R_V)$ as points in $\Z \times \Z$, where $(i, j)$ represents $\smash{U_B^i W_{B,0}^j}$ or $\smash{V_T^i W_{T,0}^j}$, respectively. It will be helpful for us to extend the total order of Figure~\ref{fig:Xring} to the non-unit elements of $\mathcal{L}(\R_U)$ and $\mathcal{L}(\R_V)$, and translate this into the setting of Definition~\ref{def:rureformulation}. We give a general explanation of how to extend the total order to the field of fractions in Definition~\ref{def:totalorder}; for now we simply record the result:

\begin{definition}\label{def:latticetotalorder}
Define a total order on $\Z \times \Z - \{(0,0)\}$ as follows. To compare $(i, j)$ and $(k, \ell)$, there are two cases. First suppose that $j$ and $\ell$ are distinct. Then $(i,j) \lebang (k,\ell)$ if(f) $1/j < 1/\ell$, with the interpretation
\[
1/j = \begin{cases}
                        + \infty \text{ if } j = 0, i > 0 \\
                        - \infty \text{ if } j = 0, i < 0
                    \end{cases}
 \text{ and } \quad
1/\ell = \begin{cases}
                        + \infty \text{ if } \ell = 0, k > 0 \\
                        - \infty \text{ if } \ell = 0, k < 0.
                    \end{cases}                    
\]
If $j = \ell$, then $(i,j) \lebang (k,\ell)$ under the following circumstances:
\begin{enumerate}
\item If $j \neq 0$, we require $i < k$.
\item If $j=0$, we require $1/k < 1/i$.
\end{enumerate} 
Note that $1/k < 1/i$ is not always equivalent to $i < k$, as $i$ and $k$ may not be positive.
\end{definition}

Definition~\ref{def:latticetotalorder} is graphically depicted in Figure~\ref{fig:totalorder}. Moving along the red arrows in Figure~\ref{fig:totalorder} enumerates the points of $\Z \times \Z - \{0,0\}$ in descending order, with $(1, 0)$ being the greatest element and $(-1, 0)$ being the least. The solid red arrows go from one lattice point to the next-largest, which always lies immediately to the right. The dashed arrows schematically represent further inequalities between different parts of the diagram. For example, the points on a given horizontal line generally all dominate all the points on the line immediately above it. There are three exceptions/additions to this rule: only the positive $x$-axis dominates the line above it, only the negative $x$-axis is dominated by the line below it, and all horizontal lines above the $x$-axis dominate all horizontal lines below the $x$-axis. Figure~\ref{fig:totalorder} should be compared with Figure~\ref{fig:Xring}.

\begin{figure}[h!]
\includegraphics[scale = 0.8]{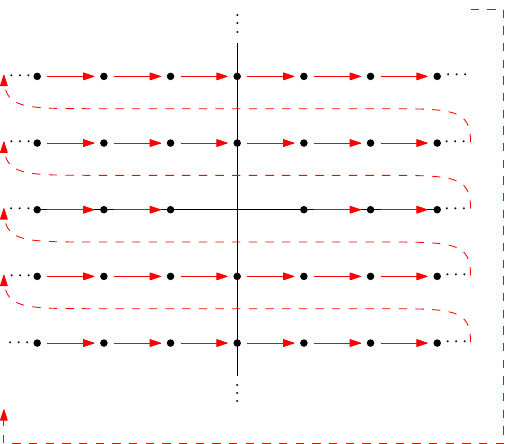}
\caption{The total order $\leqbang$ on $\Z \times \Z - \{(0,0)\}$. Following arrows in the diagram corresponds to decreasing in the total order. Note that $(1,0)$ is the greatest element and $(-1,0)$ is the least.}\label{fig:totalorder}
\end{figure}

Finally, note the existence of a (unital, bigraded) morphism from $\F[U, V]$ to $\X$, given by
\begin{equation}\label{eq:substitution}
\begin{aligned}
&U\mapsto U_B+W_{T,0} \\
&V\mapsto V_T+W_{B,0}.
\end{aligned}
\end{equation}
This will be of importance in the sequel.

\subsection{Homology spheres other than $S^3$: an example} 

We now discuss how to extend the results of \cite{DHSTmoreconcord} to knots in homology spheres other than $S^3$. As stated in Section~\ref{sec:background}, if $K\subset Y$ is a knot in an integer homology sphere, then there is a $\F[U,V]$-equivariant homotopy equivalence
\[
\CFK(Y,K)\otimes_{\F[U, V]} \F[U,U^{-1},V,V^{-1}]\simeq \F[U,U^{-1},V,V^{-1}]. 
\] 
This leads to Definition~\ref{def:loceq}. However, for the definition of local equivalence in \cite[Definition 3.3]{DHSTmoreconcord}, a slightly different localization result is needed. Observe that since $UV = 0$ in $\R$, inverting $U$ in $\R$ sets $V = 0$. It follows from this that
\[
H_*(U^{-1} \CFK_\R(Y,K)) \cong \F[U, U^{-1}] \otimes \HFhat(Y).
\]
If $Y = S^3$, then we obtain a single tower $\F[U, U^{-1}]$, which leads to a clear notion of local equivalence for such complexes over $\R$. However, if $Y$ is not $S^3$, then in general we will have multiple towers. Explicitly, if $Y$ is not $S^3$, then $\CFK_\R(Y, K)$ is \textit{not} always a knot-like complex in the sense of \cite[Definition 3.3]{DHSTmoreconcord}.

We illustrate this with a simple example. Let $M_2$ denote $+1$-surgery on the torus knot $T_{2, 7}$ and let $Y_2 = M_2 \# -M_2$. Let $K_2$ denote the connected sum of the core of surgery in $M_2$ and the unknot in $-M_2$.  (This notation will become clear in Section~\ref{sec:computations}, where this example is generalized.) According to \cite{Zhou}, the knot Floer complex $\CFK(Y_2, K_2)$ is locally equivalent (over $\F[U, V]$) to a complex generated by $x_0$, $x_1$, and $y$, with bigradings 
\begin{align*}
\gr(x_0) &= (2, 0) 
\\
\gr(x_1) &= (0, 2)
\\
\gr(y) &= (3, 3)
\end{align*}
and differential 
\[
\partial x_0  = UV^2 y \qquad \text{and} \qquad \partial x_1 = U^2Vy.
\]
We focus on what happens to this subcomplex after passing to the quotient ring $\R$. The important point is that modulo $(UV)$, the above differential is trivial. Hence localizing with respect to $U$ and taking the homology clearly yields \textit{three} copies of $\F[U, U^{-1}]$, rather than one. (Equivalently, modding out by the ideal $(V)$ and taking the homology gives three nontorsion towers.)

Instead of passing to the quotient ring $\R$, our strategy will instead be to consider what happens when we perform the substitution of the previous subsection. This should be thought of as changing the base ring from $\F[U, V]$ to $\X$ using the morphism (\ref{eq:substitution}). Keeping in mind the relations in $\X$, the differential in the above (sub)complex becomes
\[
\partial x_0 = (U_B W_{B, 0}^2 + W_{T,0} V_T^2) y \qquad \text{and} \qquad \partial x_1 = (U_B^2 W_{B, 0} + W_{T,0}^2 V_T) y.
\]
We can then perform a (bigraded) change-of-basis and set 
\begin{align*}
a_0 &= x_0 + W_{B, -1} x_1 ,
\\
a_1 &= x_1 + W_{T, -1} x_0 ,
\\
b & = y. 
\end{align*}
This turns the differential into 
\[ 
\partial a_0 = W_{T,0} V_T^2 b  \qquad \text{and} \qquad \partial a_1 = U_B^2 W_{B,0} b.
\]
This (sub)complex should be thought of as having two arrows, one decorated with a homogenous element of the ideal $(U_B)$, and the other decorated with a homogenous element of the ideal $(V_T)$. Moreover, suppose that we now invert all homogeneous elements of $\R_U$. (This is the analogue of localizing at $U$ in the case of complexes over $\R$.) Since the product of $(U_B)$ and $(V_T)$ is zero in $\X$, this sets all elements of $(V_T)$ to zero. It is easily checked that the resulting localization has homology consisting of a single copy of $\mathcal{L}(\R_U)$, generated by $a_0$.

The analogy with complexes over $\R$ should now be clear. In order to study knots in general homology spheres, we first perform the change of base ring given by (\ref{eq:substitution}) to obtain the knot Floer complex over $\X$. In Section~\ref{sec:algebra-for-knot-floer} we verify that $\X$-complexes obtained in this way have the appropriate tower behavior after inverting homogenous elements of $\R_U$ (similarly for $\R_V$), and that homology concordances induce the obvious notion of local maps. 

Our subsequent analysis then proceeds almost exactly as in \cite{DHSTmoreconcord}: we show that every such $\X$-complex decomposes into the sum of a part whose homology is acyclic upon inverting $\R_U$ or $\R_V$, and a summand supporting the single tower. Up to a change of basis, this latter piece takes the form of a standard complex as in \cite[Section 4.1]{DHSTmoreconcord}, except that now arrows are decorated by homogenous elements of $(U_B)$ or $(V_T)$, rather than powers of $U$ or $V$. We can thus parameterize this via a sequence of lattice points, as in Definition~\ref{def:rureformulation}. In the above example, we have the standard complex
\[
C(-(2, 1), (2, 1))
\]
obtained by considering the sequence of generators $a_1, b$, and $a_0$. (This ordering is explained in Definition~\ref{def:standard}.) The first entry $-(2, 1)$ corresponds to the arrow decorated with $U_B^2 W_{B,0}$ which goes from $a_1$ to $b$; the negative sign a convention due the direction of the arrow. The second entry $(2,1)$ corresponds to the arrow decorated with $V_T^2 W_{T,0}$ going from $a_2$ to $b$. See Definition~\ref{def:standard} for a precise description of this procedure and its conventions.

The analysis of the total order on the set of standard complexes proceeds the same as before, except that we use the ordering of Definition~\ref{def:latticetotalorder}. Given a standard complex and a choice of decoration, we may then form the (signed) count of arrows in that complex with that decoration. These define the homomorphisms $\varphi_{i,j}$ described in Theorem~\ref{thm:main}. Because of the symmetry of the parameter sequence, by convention we only count odd-index parameters. Hence in the above example, we obtain
\[
\varphi_{2,1}(K) = -1
\]
and all other $\varphi_{i,j}(K) = 0$. See Section~\ref{sec:homs} for further discussion and Section~\ref{sec:computations} for more examples.

In fact, all of our results will hold for complexes over a much more general class of rings. Defining these will take up the bulk of the next section.

\section{Grid rings and their properties}\label{sec:knotlike}

In this section, we define a general class of rings with many of the same properties as $\R$ and $\X$. We then extend the notion of a knotlike complex (as defined in \cite[Definition 3.3]{DHSTmoreconcord}) to complexes over such rings. Our overarching goal will be to show that the set of knotlike complexes modulo an appropriate notion of local equivalence forms a totally ordered abelian group.

\subsection{Graded valuation rings} We begin by generalizing some properties of the rings $\F[U]$ and $\R_U$. Chief among these is the existence of a total order on the set of homogenous elements. We first introduce some preliminary notions:

\begin{definition}\label{def:gradedlocalring}
Throughout this paper, a \emph{graded domain} $R$ will be a $\Z\times \Z$-graded integral domain. We denote the two components of the grading by $\gr = (\gr_1, \gr_2)$ and denote the degree-$(i,j)$ summand of $R$ by $\smash{R^{(i, j)}}$. We furthermore impose the following conditions on $\gr$. Firstly, we require that $\gr_1(x) \equiv \gr_2(x) \bmod{2}$ for any homogenous $x \in R$. Secondly, we require that either:
\begin{enumerate}
\item $\gr$ takes values in $2\Z \times 2\Z$ (in which case the previous condition is of course trivial); or,
\item $-1 = 1$ in $R$.
\end{enumerate} 
By convention, every integral domain is assumed to be unital and commutative.
\end{definition} 

\begin{remark}
The grading conditions on $R$ may seem rather opaque, but will be necessary in order to form tensor products and duals of complexes over $R$. Roughly speaking, the above conventions allow us to view $R$ as being simultaneously commutative and grading-commutative. See Section~\ref{sec:localequiv}.
\end{remark}  


\begin{definition}\label{def:gradedfield}
A \emph{graded field} is a graded domain for which every nonzero homogeneous element has a homogeneous inverse. 
\end{definition}

Note that in general a graded field is not a field, although the degree-$(0,0)$ piece of a graded field is indeed a field.  

\begin{definition}\label{def:homogenouslocalization}
Given a graded domain $R$ and a multiplicative subset $S$ of homogeneous elements in $R$,  
we define the \emph{homogeneous localization} $S^{-1}R$ of $R$ at $S$ to be the set of pairs
\[
(x,y)\in R\times S, 
\] 
 with $y$ homogeneous, subject to the equivalence relation:  
\[
(x,y) \sim (z,w) \mbox{ if } xw-zy = 0.
\]
The product and sum of pairs is defined in the usual way.  Moreover, $S^{-1}R$ admits a $\Z\times \Z$-grading, where homogeneous elements are given by pairs $(x,y)$ with $x,y$ both homogeneous, and with grading given by $\gr(x)-\gr(y)$. In the special case that $S$ consists of all nonzero homogenous elements of $R$, this construction yields the \textit{field of homogenous fractions} of $R$, which we denote by $\loc(R)$. Note that $\loc(R)$ is not in general a field, but rather a graded field in the sense of Definition~\ref{def:gradedfield}.
\end{definition}

\begin{example}
The field of homogenous fractions of $\F[U]$ is given by $\F[U, U^{-1}]$. The field of homogenous fractions of $\R_U$ is given by 
\[
\loc(\R_U) = \F[U_B, U_B^{-1}, W_{B,0}, W_{B,0}^{-1}].
\]
\end{example}

We now formalize the analogue of the total ordering property discussed in Figure~\ref{fig:Xring} and surrounding discussion:

\begin{definition}\label{def:homogenousvaluationring}
Let $R$ be a graded domain. We say that $R$ is a \textit{graded valuation ring} if any one of the following three equivalent conditions hold:
\begin{enumerate}
\item For any nonzero homogenous $x$ in $\loc(R)$, either $x \in R$ or $x^{-1} \in R$ (or both).
\item The set of homogenous ideals in $R$ is totally ordered by inclusion.
\item The set of homogenous elements in $R$, modulo homogenous units in $R$, is totally ordered by divisibility.
\end{enumerate}
\end{definition}
For the equivalence of the three conditions for (ungraded) valuation rings, see \cite{warfield} (also \cite[Lemma 15.119.2]{Stacksproject}). For more details on graded valuation rings, see \cite[Lemma I.3.2]{NO-gradedringtheory}. 

\begin{remark}\label{rem:structure}
Using the divisibility condition, it is easily checked that every finitely-generated homogenous ideal in a graded valuation ring is principal. (Note that $\X$ is neither a PID nor Noetherian.) It turns out that much of the intuition for modules over a principal ideal domain thus carries over to modules over graded valuation rings. In particular, we will need the following structure theorem: suppose that $M$ is a finitely-generated (graded) module over a graded valuation ring $R$. Then $M$ is isomorphic to
\[
M \cong R^m \oplus \left( \bigoplus_{i=1}^n R/(\nu_i) \right)
\]
for some homogenous elements $\nu_i \in R$, just as in the the case of a module over a principal ideal domain. More generally, if $(C, \partial)$ is any free, finitely-generated chain complex over a graded valuation ring $R$, then it is possible to do a homogenous change-of-basis to obtain a basis $\{x_j\}_{j=1}^m \cup \{y_i, z_i\}_{i=1}^n$ such that
\begin{align*}
	&\d x_j = 0, \\
	&\d y_i = \nu_i z_i, \text{ and} \\
	&\d z_i = 0
\end{align*}
for some homogenous $\nu_i \in R$. We refer to such a basis as a \textit{paired basis} (see also Definition~\ref{def:pairedbasis}). The existence of paired bases can be shown via the standard algorithm for putting presentation matrices into Smith normal form.  
\end{remark}

\begin{definition}\label{def:valuationgroup}
Let $R$ be a graded valuation ring. Let 
\[
\Gamma(R) = (\text{homogenous elements of }\loc(R)^{\times})/(\text{homogenous units in } R)
\]
be the set of nonzero homogenous elements in $\loc(R)$, modulo homogenous units in $R$.\footnote{Actually, it is possible to show that in a graded domain, \textit{every} unit is necessarily homogenous.} This forms an abelian group (under multiplication), which we refer to as the \textit{valuation group of} $R$. Throughout this paper, we sometimes abuse notation slightly and identify elements of $\loc(R)^{\times}$ with their classes in $\Gamma(R)$, suppressing the quotient. Write $\Gamma_{\geq 1}(R)$ for the classes in $\Gamma(R)$ coming from homogenous elements of $R$, and $\Gamma_{\leq 1}(R)$ for classes coming from $x \in \loc(R)$ such that $x^{-1} \in R$. We write
\[
\Gamma_{+}(R) =\Gamma_{\geq 1}(R) - \{[1]\} 
\]
and
\[
\Gamma_{-}(R) =\Gamma_{\leq 1}(R) - \{[1]\}. 
\]
\end{definition}

\begin{definition}\label{def:absvalue}
For $x$ in the value group $ \Gamma(R)$, define 
\begin{align*}
|x| = \begin{cases} 
x \quad &\text{if } x \in \Gamma_{\geq 1}(R),
\\
x^{-1} \quad & \text{if } x \in \Gamma_{\leq 1}(R).
\end{cases}
\end{align*}
For $x \neq [1]$, we say that $x$ is \textit{positive} or \textit{negative} if $x$ is in $\Gamma_+(R)$ or $\Gamma_-(R)$, respectively. For such $x$, define
\begin{align*}
\sgn(x) = \begin{cases} 
+1 \quad &\text{if } x \in \Gamma_{+}(R),
\\
-1 \quad & \text{if } x \in \Gamma_{-}(R).
\end{cases}
\end{align*}
\end{definition}

Note that $\Gamma_+(R)$ is totally ordered. We use this to put a total order $\lebang$ on the set $\Gamma_+(R) \cup \Gamma_-(R)$:

\begin{definition}\label{def:totalorder}
Let $x, y \in \Gamma_+(R) \cup \Gamma_-(R)$. We define:
\begin{enumerate}
\item If $x, y \in \Gamma_+(R)$, then $x \leqbang y$ if(f) $xy^{-1} \in R$; that is, $y$ divides $x$ in $R$.
\item If $x, y \in \Gamma_-(R)$, then $x \leqbang y$ if(f) $xy^{-1} \in R$; that is, $|x|$ divides $|y|$ in $R$.
\item If $x \in \Gamma_-(R)$ and $y \in \Gamma_+(R)$, then $x \lebang y$.\footnote{Note that if we were to simply define $x \leqbang y$ if(f) $xy^{-1} \in R$, then $\Gamma_-(R)$ would dominate $\Gamma_+(R)$.}
\end{enumerate}
This defines a total order on the set $\Gamma_+(R) \cup \Gamma_-(R)$. By convention, we place $[1]$ in between these two submonoids in the total order, so that $\leqbang$ defines a total order on all of $\Gamma(R)$. Note that this means when one of $x$ and $y$ is the identity element of $\Gamma(R)$, the order $\leqbang$ coincides with the symbol $\leq$ in Definition~\ref{def:valuationgroup}.
\end{definition}

\begin{remark}\label{rem:totalorder}
Note that Definition~\ref{def:totalorder} is \textit{not} the total order on $\Gamma(R)$ which is usually presented in the literature. This latter order is defined as follows: for $x, y \in \Gamma(R)$, we set $x \leq y$ if(f) $x^{-1}y \in R$. (Contrast with Definition~\ref{def:totalorder}.) This turns $\Gamma(R)$ into a totally ordered abelian group, and is usually what is meant in the literature by the valuation group. However, this is \textit{not} the total order which we will consider in this paper. Instead, the total order of Definition~\ref{def:totalorder} may be obtained by reversing the total order on $\Gamma_+(R)$ and $\Gamma_-(R)$ described above, and then declaring any element of $\Gamma_-(R)$ to be less than any element of $\Gamma_+(R)$. We stress that in contrast to the usual total order, Definition~\ref{def:totalorder} does \textit{not} always interact naturally with the group structure when $x$ and $y$ are not in the same submonoid. Thus, although we refer to $\Gamma(R)$ as the valuation \textit{group}, in our context this is slightly misleading.
\end{remark}

The reason we need a total order on $\Gamma_+(R) \cup \Gamma_-(R)$ is the follows: in our ordering of standard complexes, we must take into account the direction of arrows, as well as their decoration. (See \cite[Definition 4.3]{DHSTmoreconcord}.) It will thus be convenient for us to think of a decoration $x$ on an arrow in the ``negative" direction instead as the decoration $x^{-1}$.

\begin{example}
The ring $\F[U]$ is a graded valuation ring. We have:
\begin{align*}
\Gamma_{\geq 1 } (\F[U]) &= \{ 1, U, U^2, U^3, \dots \},
\\
\Gamma (\F[U]) &= \{ \dots, U^{-3}, U^{-2}, U^{-1}, 1, U^{1}, U^{2}, U^{3}, \dots \}.
\end{align*}
The total order is displayed in Figure~\ref{fig:order}; compare with \cite[Section 4.2]{DHSTmoreconcord}. The ring $\R_U$ is also a graded valuation ring. We have:
\begin{align*}
\Gamma_{\geq 1 } (\R_U) &= \{ U_B^i \}_{i\in \mathbb{Z}^{\geq 0}} \cup \{W_{B,i} W_{B,0}^j  \}_{i \in \mathbb{Z}, j \in \mathbb{Z}^{\geq 0}},
\\
\Gamma (\R_U) &=   \{U_B^i W_{B,0}^j  \}_{i, j \in \mathbb{Z}}.
\end{align*}
The total order is displayed in Figure~\ref{fig:order}. This is precisely the same as Figure~\ref{fig:totalorder} under the identification of $\smash{U_B^i W_{B,0}^j}$ with the lattice point $(i, j)$.
\end{example}

\begin{figure}[h!]
\includegraphics[scale = 0.75]{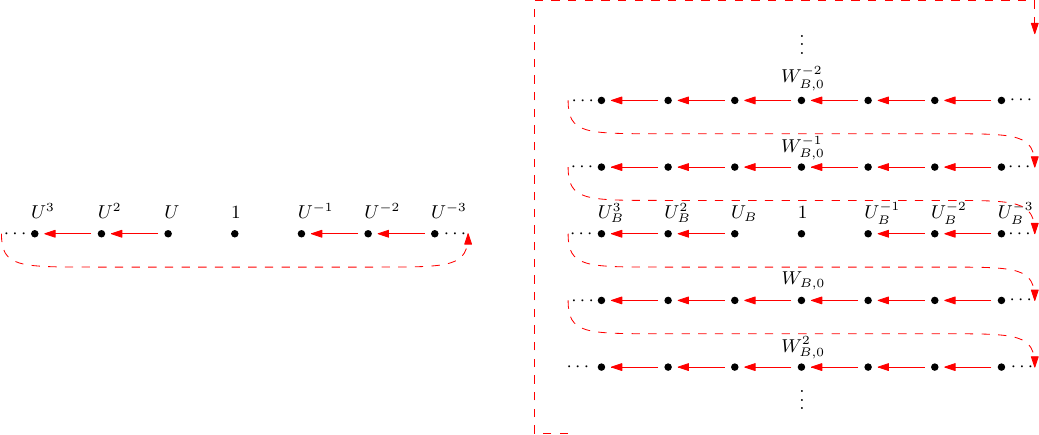}
\caption{Left: total order on $\Gamma(\F[U])$, with largest element $U$. Successively smaller elements are obtained by following the red arrows; the dashed red arrow schematically represents the fact that $\Gamma_+(\F[U])$ dominates $\Gamma_-(\F[U])$. Although $1$ is not drawn within the total ordering, we may place it in the middle of the dashed arrow. Right: total order on $\Gamma(\R_U)$, with largest element $U_B$. Again, $1$ may be placed in the middle of the long dashed arrow running from the bottom to the top of the diagram. For brevity, we have only labeled elements along the primary axes.}\label{fig:order}
\end{figure}

In this paper, we will consider a particular type of graded valuation ring, which we call \textit{type-zero}. 

\begin{definition}\label{def:typezero}
Let $R$ be a graded valuation ring. We say that $R$ is \textit{type-zero} if either of the following two equivalent conditions hold:
\begin{enumerate}
\item The degree-$(0,0)$ piece of $R$ consists precisely of the invertible elements of $R$, together with the zero element. 
\item The unique maximal homogenous ideal $\fm$ of $R$ consists of everything outside the degree-$(0,0)$ piece of $R$; that is,
\[
\fm =  \bigoplus_{(i,j) \neq (0,0)} R^{(i,j)}.
\]
\end{enumerate}
Note that this means the degree-$(0,0)$ piece of $R$ is a field. We denote this field by $\K$ and refer to it as the \textit{residue field} of $R$. If $R$ is type-zero, then $R$ is an algebra over $\K$, and the quotient map from $R$ to $R/\fm \cong \K$ is just projection onto the degree-$(0,0)$ piece of $R$.
\end{definition}

To make sense of the second condition, note that any graded valuation ring $R$ has a unique maximal homogenous ideal $\fm$. (That is, the set of homogenous ideals has a unique maximal element. To see this, consider the union of all nontrivial homogenous ideals in $R$.) We leave the equivalency of the above two conditions as an exercise for the reader. The central point is to show that (in general) if $x$ is a homogenous element of $R$, then $x \in \fm$ if and only if $x$ is not invertible.

\begin{example}\label{ex:typezero}
Both $\F[U]$ and $\R_U$ are type-zero graded valuation rings with residue field $\F$. 
\end{example}

Although the condition of being type-zero might seem rather restrictive, in practice many graded valuation rings fall into this category. For instance, one may consider adjoining (non-invertible) variables to a field in degrees other than $(0,0)$. We close this subsection with a simple fact which will useful in later sections:

\begin{lemma}\label{lem:onedimensional}
Let $R$ be a type-zero graded valuation ring. Then $R$ is a one-dimensional $\K$-vector space in each bigrading.
\end{lemma}
\begin{proof}
Let $x$ and $y$ be any two elements in $R^{(i,j)}$. Since $R$ is a graded valuation ring, we have that $x = cy$ for some $c \in R$. Since $x$ and $y$ lie in the same bigrading, $c$ lies in grading $(0,0)$; hence $c \in \K$.
\end{proof}

\subsection{Grid rings} We now introduce the notion of a grid ring. As we will see, these should be thought of as being built from two graded valuation rings, in analogy to the way that $\X$ is built from $\R_U$ and $\R_V$. Roughly speaking, the idea will be to take two type-zero graded valuation rings which share the same residue field and glue them together. 

\begin{definition}\label{def:stairs}
	
Let $\R_U$ and $\R_V$ be two type-zero graded valuation rings with maximal homogenous ideals $\fm_U$ and $\fm_V$.\footnote{We will often abuse notation and use $\R_U$ and $\R_V$ to denote two general type-zero graded valuation rings, rather than the specific graded valuation rings $\R_U$ and $\R_V$ used in the construction of $\X$. In practice, this will cause little confusion.} Suppose that the residue fields $\R_U/\fm_U$ and $\R_V/\fm_V$ are isomorphic. For clarity, we assume both of these admit fixed isomorphisms to some field $\K$, so that we have bigraded quotient maps
\[
f \colon \R_U \rightarrow \K \quad \text{and} \quad g \colon \R_V \rightarrow \K,
\]
where $\K$ is concentrated in degree $(0,0)$. (Note that $\ker f = \fm_U$ and $\ker g = \fm_V$.) We define the \textit{grid ring} $\stair = \stair(\R_U,\R_V)$ to be the pullback
\[
\R_U \times_\K \R_V = \{(r_U, r_V) : f(r_U) = g(r_V)\} \subset \R_U \times \R_V.
\]
This inherits a $\Z \times \Z$ grading, where $\gr (r_U, r_V) = (\gr(r_U), \gr(r_V))$.
\end{definition}

\begin{definition}\label{def:gradingnontrivial}
We say that $\stair(\R_U, \R_V)$ is \textit{grading-nontrivial} if there exists some $\mu \in \R_U$ with $\gr_1(\mu) \neq 0$ and $\nu \in \R_V$ with $\gr_2(\nu) \neq 0$. This (mild) restriction will be useful in Section~\ref{sec:homs}.
\end{definition}

It may not be immediately obvious that our previous examples $\R$ and $\X$ are grid rings; we give a more concrete construction that will make this clear presently. 

\begin{definition}\label{def:maxideal}
Consider the ideals
 \[
 \{(\mu, 0) : \mu \in \fm_U\} \subset \stair(\R_U, \R_V)
 \]
 and
 \[
 \{(0, \nu) : \nu \in \fm_V\} \subset \stair(\R_U, \R_V).
 \]
Viewed as subrings of $\stair$, these are obviously isomorphic to $\fm_U$ and $\fm_V$; by abuse of notation, we denote these by $\fm_U$ and $\fm_V$ also. 
\end{definition}

\begin{remark}
We make little distinction between the ideals $\fm_U \subset \R_U$ and $\fm_U \subset \stair$, relying on context to distinguish these when needed. In particular, given two homogenous elements $x, y \in \fm_U$, the reader may verify that the question of whether $x$ divides $y$ (or, indeed, whether $x$ is a unit multiple of $y$) is independent of whether $x$ and $y$ are considered as elements of $\R_U$ or elements of $\stair$. We thus freely use the fact that homogenous elements of $\fm_U$ (modulo units) are totally ordered by divisibility, without regard to the setting.
\end{remark}

An alternative definition of $\stair(\R_U, \R_V)$ can be given as follows. Keeping in mind that $\R_U$ and $\R_V$ are $\K$-algebras, consider
\[
\R_U \otimes_\K \R_V/(\fm_U\otimes_\K \fm_V).
\]
We claim that this is isomorphic to the grid ring $\stair(\R_U, \R_V)$. Indeed, any element $(r_U, r_V) \in \stair(\R_U, \R_V)$ can be written uniquely as $(k + \mu, k + \nu)$ with $k \in \K$ and $\mu \in \fm_U, \nu \in \fm_V$. Here, we are using that $\K$ appears as a summand of $\R_U$ and $\R_V$. The desired isomorphism sends $(k + \mu, k + \nu)$ to the element
\[
k (1 \otimes 1) + \mu \otimes 1 + 1 \otimes \nu
\]
in the quotient of the above tensor product. The reader may verify that this is a ring isomorphism. While the tensor product construction is rather easier to get a handle on, we have presented Definition~\ref{def:stairs} in the hope that the categorically-minded reader will find the class of grid rings to be a more natural construction and/or indicate possible avenues for generalization outside of the type-zero case.

\begin{remark}\label{rem:sproperties}
It is clear from the tensor product point of view that $\stair$ contains subrings isomorphic to $\R_U$ and $\R_V$, as well as the field $\K$. Note that in general, the pullback of two rings need not contain a copy of either ring. (This is one reason we have imposed the type-zero condition.) Instead, projection onto either coordinate gives quotient maps
\[
p_U \colon \stair(\R_U, \R_V) \rightarrow \R_U \quad \text{and} \quad p_V \colon \stair(\R_U, \R_V) \rightarrow \R_V
\]
whose kernels are $\fm_V$ and $\fm_U$, respectively. We will thus sometimes write
\[
\stair(\R_U, \R_V)/\fm_V \cong \R_U \quad \text{and} \quad \stair(\R_U, \R_V)/\fm_U \cong \R_V.
\]
Similarly, we have a map 
\[
\pi \colon \stair(\R_U, \R_V) \rightarrow \K
\]
given by sending $(r_U, r_V)$ to $f(r_U) = g(r_V)$. The kernel of this is $(\fm_U, \fm_V)$; we thus write
\[
\stair(\R_U, \R_V)/(\fm_U, \fm_V) \cong \K.
\]
Note that $\R_U$, $\R_V$, and $\K$ are thus all $\stair$-modules. (The action of $s \in \stair$ on $r_U \in \R_U$ is given by $p_U(s) \cdot r_U$, and so on.) Clearly, $\fm_V \subset \stair$ acts on $\R_U$ as zero, with similar statements holding for $\R_V$ and $\K$.
\end{remark}

\begin{example}
The discussion of this subsection shows that $\R$ and $\X$ are both grading-nontrivial grid rings. The former is built from the graded valuation rings $\F[U]$ and $\F[V]$; the latter is built from the graded valuation rings $\R_U$ and $\R_V$ of Definiton~\ref{def:xring}. In both cases the field $\K$ is given by $\F$.
\end{example}

\subsection{Complexes over grid rings} We now translate the machinery of local equivalence into the setting of complexes over grid rings. For the remainder of this section, fix a grid ring $\stair = \stair(\R_U,\R_V)$.

\begin{definition}\label{def:scomplex}
Let $\stair$ be a grid ring. Throughout this paper, an \textit{$\stair$-complex} $(C, \partial)$ will mean a free, finitely-generated, $\Z\times\Z$-graded chain complex over $\stair$. We denote the two components of the grading by $\gr = (\gr_1, \gr_2)$. The differential $\partial$ is required to have degree $(-1, -1)$. We also impose the condition that $\gr_1(x) \equiv \gr_2(x) \bmod 2$ for each homogenous $x \in C$; this will be necessary when we construct products and duals of complexes in Section~\ref{sec:localequiv}. Note that due to the grading condition of Definition~\ref{def:gradedlocalring}, this parity condition is preserved by the action of $\stair$.
\end{definition}

Given an $\stair$-complex $C$, we may form the tensor product $C \otimes_\stair \ru$, where $\ru$ is given the action of $\stair$ discussed in Remark~\ref{rem:sproperties}. This may alternatively be thought of as setting $\fm_V$ to be zero inside $C$. We thus sometimes write
\[
C \otimes_\stair \ru \cong C/\fm_V.
\]
The resulting complex is an $\R_U$-module and thus satisfies the structure theorem of Remark~\ref{rem:structure}. This procedure is of course analogous to setting $V = 0$ for a complex over $\R$. We similarly have
\[
C \otimes_\stair \rv \cong C/\fm_U
\]
and
\[
C \otimes_\stair \K \cong C/(\fm_U, \fm_V).
\]

\begin{definition}\label{def:knotlike-complex}
A \emph{left knotlike complex} $C$ is an $\stair$-complex for which 
\[
H_*(C \otimes_\stair \ru)/\ru\text{-torsion} \cong \ru
\]
by an isomorphism which is absolutely $\gr_2$-graded and relatively $\gr_1$-graded. We refer to this as the \textit{$U$-tower condition}. We say that $x \in C$ is a \textit{$U$-tower class} if the image of $x$ in $H_*(C \otimes_\stair \ru)/\ru\text{-torsion}$ generates $\ru$ under the above isomorphism. This means that the image of $x$ in $C \otimes_\stair \ru$ is necessarily a cycle; we will sometimes be imprecise with distinguishing between $x$ and its image in $C \otimes_\stair \ru$, $H_*(C \otimes_\stair \ru)$, or $H_*(C \otimes_\stair \ru)/\ru\text{-torsion}$. Note that any $U$-tower class has $\gr_2 = 0$. A \emph{right knotlike complex} is defined analogously, replacing $U$ with $V$ and $\gr_2$ with $\gr_1$; we similarly define the \textit{$V$-tower condition} and the notion of a \textit{$V$-tower class}. A \emph{totally knotlike complex} is a $\stair$-complex which is both left and right knotlike
\end{definition}

\begin{remark}\label{rem:gradingshift}
To motivate the grading condition in Definition~\ref{def:knotlike-complex}, suppose that $C$ is an $\stair$-complex such that $H_*(C \otimes_\stair \ru)/\ru\text{-torsion}$ and $H_*(C \otimes_\stair \rv)/\rv\text{-torsion}$ are isomorphic to $\ru$ and $\rv$, respectively, but via maps which are only relatively graded. Then there exists a unique grading shift on $C$ that makes the first isomorphism $\gr_2$-preserving (and $\gr_1$-homogenous), and simultaneously makes the second isomorphism $\gr_1$-preserving (and $\gr_2$-homogenous).
\end{remark}

\begin{remark}
The extremely observant reader may wonder whether it is necessary to remember the data of the actual map constituting the isomorphism in Definition~\ref{def:knotlike-complex}. This turns out not to be important due to the fact that any $\R_U$-module isomorphism from $\R_U$ to itself is multiplication by a unit, which necessarily lies in degree $(0,0)$. To see this, observe that any $\R_U$-module map $h$ from $\R_U$ to itself is just multiplication by the element $h(1)$. If $h$ has an inverse, it is clear that $h(1)$ is invertible.
\end{remark}

If we have a morphism $f\colon C_1\to C_2$ of $\stair$-complexes, then we obtain a morphism
\[
f \otimes \id \colon C_1 \otimes_\stair \ru \rightarrow C_2 \otimes_\stair \ru
\]
which we often also denote by $f$. Similar statements hold with $\rv$ in place of $\ru$.

\begin{definition}\label{def:local-map}
A \textit{left local map} of (left) knotlike complexes $C_1\to C_2$ is an $\stair$-equivariant chain map 
	\[f\colon C_1 \to C_2\]
such that $f$ is $\gr_2$-preserving and $\gr_1$-homogenous, and $f$ induces an isomorphism on $H_*(C \otimes_\stair \ru)/\ru\text{-torsion}$. A \textit{right local map} is defined analogously, replacing $U$ with $V$ and $\gr_2$ with $\gr_1$. For $C_1$ and $C_2$ (totally) knotlike complexes, we say that $f$ is \textit{totally local} if it is a local map of both left and right knotlike complexes.
\end{definition}

\begin{definition}\label{def:localequivalence}
We say that two (left) knotlike complexes  $C_1$ and $C_2$ are \emph{left locally equivalent} if there exist (left) local maps $C_1\to C_2$ and $C_2\to C_1$. The set of left local equivalence classes of (left) knotlike complexes admits a partial order $\leq$ by declaring $C_1\leq C_2$ if there exists a (left) local map $C_1\to C_2$. Similar definitions hold for right knotlike complexes.
\end{definition}

The reader familiar with the definitions of \cite[Section 3]{DHSTmoreconcord} will have no trouble checking that the above notions generalize the case of knotlike complexes over $\R$. Note that with our conventions, \cite[Definition 3.3]{DHSTmoreconcord} corresponds to the definition of a right local map. We will generally work with right local maps throughout this paper, and we thus often refer to a right local map simply as a ``local map". Here there may be some confusion, in that we may consider a right local map between two right knotlike complexes, or a right local map between two totally knotlike complexes. In the latter case, however, it turns out that a left or right local equivalence between two totally knotlike complexes is in fact totally local; see Lemma~\ref{lem:loceqsymmetric}.

Given any free, finitely-generated $\stair$-complex $C$, we may consider the image of the action of $(\fm_U, \fm_V)$ on $C$. Since $C$ is a free module and $\fm_U \cap \fm_V = \{0\}$ in $\stair$, every element $x$ in $(\fm_U,\fm_V) \cdot C$ can be uniquely expressed as a sum $x_U+x_V$ for $x_U\in \fm_U \cdot C$ and $x_V\in \fm_V \cdot C$. Usually we will write $(\fm_U, \fm_V)$ to mean the image $(\fm_U,\fm_V) \cdot C$ (and so on), as this will rarely cause confusion.

\begin{definition}\label{def:reduced}
	We say an $\stair$-complex is \emph{reduced} if $\partial \equiv 0 \bmod{(\fm_U,\fm_V)}$. Let $\partial_U = \partial \bmod \fm_U$ and $\partial_V = \partial \bmod \fm_V$.  In a reduced complex, we may write 
	\[
	\partial=\partial_U+\partial_V. 
	\]
	Note that $\partial^2_U = \partial^2_V = 0$. We refer to $\partial_U$ the \emph{$U$-differential} and elements with $\partial_Ux=0$ as $U$-\emph{cycles}, and similarly for $\partial_V$.  
\end{definition}

\begin{remark}
Note that writing $\partial=\partial_U+\partial_V$ is basis-independent, since $C$ is a reduced complex over a grid ring $\stair$, i.e., $\partial \equiv 0 \bmod{(\fm_U,\fm_V)}$ and every element $x$ in $(\fm_U,\fm_V) \cdot C$ can be uniquely expressed as a sum $x_U+x_V$ for $x_U\in \fm_U \cdot C$ and $x_V\in \fm_V \cdot C$
\end{remark}

\begin{lemma}\label{lem:reduced}
	Every free, finitely-generated $\stair$-complex is homotopy equivalent to a reduced $\stair$-complex.
\end{lemma}

\begin{proof}
	Suppose that $C$ is not reduced. Let $\{e_1, \ldots, e_m\}$ be a homogenous basis for $C$. Without loss of generality, we claim that we may assume $\partial e_1 = e_2$. Indeed, since $C$ is not reduced, there is some $i$ such that $\partial e_i$ is not in $(\fm_U, \fm_V)$; let $i = 1$. Then $\partial e_1$ is some $\stair$-linear combination of the $e_i$ with at least one coefficient not contained in $(\fm_U, \fm_V)$. Moreover, it cannot be the case that $\partial e_1 = ce_1 \bmod{(\fm_U,\fm_V)}$ (where $c \in \K$), as this would contradict $\partial^2 = 0$. Hence we can choose the coefficient in question to correspond to some $i \neq 1$; let this be $e_2$. Then $\partial e_1$ is equal to $ce_2$ (where $c \in \K$) plus some $\stair$-linear combination of the other $e_i$. It is then clear that $\{e_1, e_2, e_3, \ldots, e_m\}$ and $\{e_1, \partial e_1, e_3, \ldots, e_m\}$ are related to each other by a homogenous change-of-basis. This establishes the claim.
	
	We now perform a further change-of-basis, as follows. For each $i \geq 2$, write $\partial e_i$ as a linear combination of $e_1$, $\partial e_1$, and the other basis elements $e_j$. By adding multiples of $e_1$ to each such $e_i$, we may assume that $\partial e_1$ does not appear in any $\partial e_i$ with $i \geq 2$. This also shows that $e_1$ does not appear in any $\partial e_i$, since this would contradict $\partial^2 e_i=0$. Denoting this new basis by $\{e_1, \partial e_1, e_3', \ldots, e_m'\}$, it is clear that $\{e_1, \partial e_1\}$ and $\{e_3', \ldots, e_m'\}$ span subcomplexes of $C$. The former is clearly acyclic, and the inclusion and projection maps onto the latter give a homotopy equivalence. Since $C$ is finitely generated, we may iterate this procedure to arrive at a reduced complex.
\end{proof}

\begin{remark}
From now on, we will assume that all of our knotlike complexes are reduced.
\end{remark}

We now establish some convenient formalism that will be useful in the sequel. Let $C$ be a knotlike complex and fix any $\stair$-basis $\cB = \{e_1, \ldots, e_m\}$ for $C$. Then we may identify
\[
C \cong \Span_\K \{e_1, \ldots, e_m\} \otimes_\K \stair.
\]
We likewise identify
\begin{align*}
C/\fm_V &\cong \Span_\K \{e_1, \ldots, e_m\} \otimes_\K \R_U \\
C/\fm_U &\cong \Span_\K \{e_1, \ldots, e_m\} \otimes_\K \R_V
\end{align*}
as well as
\[
C/(\fm_U, \fm_V) \cong \Span_\K \{e_1, \ldots, e_m\}.
\]
Using these isomorphisms, we have obvious inclusion maps from $C/(\fm_U, \fm_V)$ to $C/\fm_V$ and $C/\fm_U$, and from $C/\fm_V$ and $C/\fm_U$ to $C$.\footnote{Here, we are using in the first case that there is an inclusion map from $\K$ into $\R_U$ and $\R_V$, and in the second that there are inclusion maps from $\R_U$ and $\R_V$ to $\stair$. Tensoring these maps with the identity map on $\Span_\K \{e_1, \ldots, e_m\}$ gives the desired maps.} These are sections of the usual quotient maps between these complexes. At risk of overloading the reader with notation, we have labeled each of these maps in the diagram below. Although the quotient maps are canonical, we stress that the maps $i_U$ and $j_U$ (and so on) are not: they depend our choice of $\cB$. Note that $i_U \circ p_U = \id \bmod \fm_V$, although this congruence is an equality on any $\R_U$-linear combination of the elements of $\cB$. Similar statements hold for the other pairs of maps in the diagram.

\[
\begin{tikzcd}
	&& {C/\mathfrak{m}_V} \\
	C && {} && {C/(\mathfrak{m}_U,\mathfrak{m}_V)} \\
	&& {C/\mathfrak{m}_U}
	\arrow["{q_V}"', shift right=4, from=3-3, to=2-5]
	\arrow["{q_U}", shift left=4, from=1-3, to=2-5]
	\arrow["{p_U}", shift left=4, from=2-1, to=1-3]
	\arrow["{p_V}"', shift right=4, from=2-1, to=3-3]
	\arrow["{i_U}", shift right=2, from=1-3, to=2-1]
	\arrow["{i_V}"', shift left=2, from=3-3, to=2-1]
	\arrow["{j_U}", shift right=2, from=2-5, to=1-3]
	\arrow["{j_V}"', shift left=2, from=2-5, to=3-3]
	\arrow["\iota", shift left=1, from=2-5, to=2-1]
	\arrow["\pi", shift left=1, from=2-1, to=2-5]
\end{tikzcd}
\]
\vspace{0.1cm}

Note that (for example) $p_U$ is an $\stair$-linear chain map, while $i_U$ is only $\R_U$-linear. Moreover, $i_U$ is not a chain map, but instead intertwines the differential $\partial_U$ on $C/\fm_V$ with the operator $\partial_U$ on $C$. Similarly, $q_U$ is $\R_U$-linear, while $j_U$ is only $\K$-linear. In this latter case, we consider $C/(\fm_U, \fm_V)$ as a bigraded vector space with trivial differential. Similar statements hold for the other pairs of maps in the diagram.

\begin{definition}\label{def:pairedbasis}
If $C$ is a (reduced) right knotlike complex, then we say that an $\stair$-basis $\cB = \{x, y_i, z_i\}_{i=1}^n$ for $C$ is a \textit{$V$-paired} basis if 
\begin{align*}
	&\d_V x = 0, \\
	&\d_V y_i = \nu_i z_i, \text{ and} \\
	&\d_V z_i = 0
\end{align*}
for some $\nu_i \in \fm_V$. Similarly, if $C$ is a (reduced) left knotlike complex, then we say that $\cB = \{x, y_i, z_i\}_{i=1}^n$ is a \textit{$U$-paired} basis if 
\begin{align*}
	&\d_U x = 0, \\
	&\d_U y_i = \mu_i z_i, \text{ and} \\
	&\d_U z_i = 0.
\end{align*}
for some $\mu_i \in \fm_U$.
\end{definition}

The fact that $V$- and $U$-paired bases exist is a straightforward consequence of the structure theorem for modules over graded valuation rings: 

\begin{lemma}\label{lem:pairedbasis}
Let $C$ be a right knotlike complex which is reduced. Then there exists a $V$-paired basis for $C$. Similarly, if $C$ is a left knotlike complex which is reduced, then there exists a $U$-paired basis for $C$.
\end{lemma}
\begin{proof}
Consider the quotient $C/\fm_U$. Since $\R_V$ is a graded valuation ring, we can find an $\R_V$-basis $\cB' = \{x', y_i', z_i'\}_{i=1}^n$ of $C/\fm_U$ which is paired in the sense of Remark~\ref{rem:structure}. Set $\cB = i_V(\cB')$, where $i_V$ is the map constructed using any arbitrary $\stair$-basis $\{e_i\}_{i=1}^m$ for $C$. It is straightforward to check that $\cB$ is an $\stair$-basis for $C$. Explicitly, each $p_V(e_i)$ is an $\R_V$-linear combination of the elements of $\cB'$; applying $i_V$ to both sides then shows that each $e_i$ is a linear combination of the elements of $\cB$. Since $i_V$ intertwines the differential $\partial_V$ on $C/\fm_U$ with the operator $\partial_V$ on $C$, this completes the proof. The claim for $\partial_U$ is similar.
\end{proof}

Note that if $C$ is a totally knotlike complex, then there is both a $U$-paired basis for $C$ and a $V$-paired basis for $C$, but in general these will not be the same. Finally, we also record the following straightforward fact:

\begin{lemma}\label{lem:finitedimensional}
Let $C$ be an $\stair$-complex with an $\stair$-basis $\{e_i\}_{i=1}^n$ of length $n$. Then $C$ is a $\K$-vector space of dimension at most $2n$ in each bigrading.
\end{lemma}
\begin{proof}
Let $e_i$ be any basis element and fix any bigrading. It is clear from Lemma~\ref{lem:onedimensional} that up to multiplication by $\K$, there is at most one $\R_U$-multiple of $e_i$ and at most one $\R_V$-multiple of $e_i$ which can lie in this bigrading. 
\end{proof}

\subsection{The local equivalence group of knotlike complexes}\label{sec:localequiv}
In this subsection (and throughout the rest of the paper) we will refer to totally knotlike complexes simply as knotlike complexes. However, for the moment it will be convenient for us to consider \textit{right} local maps and \textit{right} local equivalences between such complexes; we thus refer to these simply as local maps and local equivalences. This apparent asymmetry will be resolved in Section~\ref{sec:char}, where we show that a right local equivalence between two totally knotlike complexes is necessarily left local.

Our goal will be to show that knotlike complexes modulo local equivalence form an abelian group, with the group operation induced by tensor product. Moreover, we will show that the partial order $\leq$ on knotlike complexes is, in fact, a total order. We begin with some routine formalism:

\begin{definition}
Let $C_1$ and $C_2$ be knotlike complexes. The \emph{product} of $C_1$ and $C_2$ is the usual tensor product complex $C_1 \otimes_\stair C_2$. We suppress the subscript on the tensor product when no confusion is possible. We remind the reader (who may be used to working over $\F=\Z/2\Z$) that the differential on this complex is given by
\[
\partial(a \otimes b) = \partial a \otimes b + (-1)^{|a|} a \otimes \partial b
\]
where $|a|$ is either component of $\gr(a)$. Here, we use the condition that $\gr_1 \equiv \gr_2 \bmod$ $2$ for any $\stair$-complex, as discussed in Definition~\ref{def:scomplex}. Note that in order to make the tensor product differential an $\stair$-module map, we utilize the second grading requirement in Definition~\ref{def:gradedlocalring}.
\end{definition}

\begin{lemma}\label{lem:product-check}
The product of two knotlike complexes is a knotlike complex and the product of two local maps is a local map.
\end{lemma}

\begin{proof}
Let $C_1$ and $C_2$ be knotlike complexes. Then $C_1\otimes_\stair C_2$ is free and finitely-generated. To check the $V$-tower condition, note that
\[
(C_1\otimes_\stair C_2)\otimes_{\stair} \ru=(C_1\otimes_{\stair} \ru)\otimes_{\stair} (C_2\otimes_{\stair} \ru).
\]
Choose $V$-paired bases $\cB_1 = \{x^1, y^1_i, z^1_i\}_{i=1}^{n_1}$ and $\cB_2 = \{x^2, y^2_i, z^2_i\}_{i=1}^{n_2}$ for $C_1$ and $C_2$. Then it is easily checked that $[x^1\otimes x^2]$ generates $H_*((C_1\otimes_{\stair}C_2)\otimes\ru)/\ru\text{-torsion}$. Moreover, $\gr_1(x^1 \otimes x^2) = \gr_1(x^1) + \gr_1(x^2) = 0$. A similar computation holds for the $U$-tower condition. Hence $C_1\otimes_\stair C_2$ is a knotlike complex. The claim about local maps is standard.
\end{proof}

\begin{definition}
Let $C_1$ and $C_2$ be two $\stair$-complexes. The \emph{complex of graded maps} from $C_1$ to $C_2$ is the graded chain complex whose underlying module is given by
\[
\Hom_\stair(C_1, C_2) = \bigoplus_{\vec{n} \in \Z \times \Z} \Hom_{\vec{n}}(C_1,C_2),
\]
where $\Hom_{\vec{n}}(C_1,C_2)$ is the set of $\stair$-module maps from $C_1$ to $C_2$ of bidegree $\vec{n}$. This has a $\Z \times \Z$-grading given by the degree shift $\vec{n}$. The differential is defined to be
\[
\partial_\Hom f = \partial f-(-1)^{|f|} f\partial
\]
where $|f|$ is either component of $\gr(f)$. Here, we use the condition that $\gr_1 \equiv \gr_2 \bmod$ $2$ for any $\stair$-complex, as discussed in Definition~\ref{def:scomplex}. Note that in order to make the dual differential an $\stair$-module map, we utilize the second grading requirement in Definition~\ref{def:gradedlocalring}.
\end{definition}

\begin{definition}\label{def:dual}
Let $C$ be a $\stair$-complex. The \emph{dual} of $C$ is the complex
\[
C^\vee = \Hom_\stair(C, \stair)
\]
where $\stair$ is considered as a complex with trivial differential. For an explicit construction of $C^\vee$, fix any $\stair$-basis $\{e_1, \ldots, e_m\}$ for $C$. Then $C^\vee$ is free with basis $\{e_1^\vee, \ldots, e_m^\vee\}$, where $\gr(e_i^\vee) = - \gr(e_i)$. The differential on $C^\vee$ is given by
\[
\partial^\vee e_i^\vee = \sum_j c_j e_j^\vee
\]
where $e_j^\vee$ appears in the above sum if and only if $e_i$ appears in $\partial e_j$. The above coefficient $c_j$ is $\smash{-(-1)^{|e_i|}}$ times the coefficient of $e_i$ in $\partial e_j$.
\end{definition}

\begin{lemma}\label{lem:dual-check}
The dual of a knotlike complex is a knotlike complex and the dual of a local map is a local map. Moreover, if $C$ is a knotlike complex, then $C \otimes_\stair C^\vee$ is locally equivalent to the trivial complex $\stair$.
\end{lemma}
\begin{proof}
Let $C$ be a knotlike complex. The fact that $C^\vee$ satisfies the $V$-tower condition is clear from choosing a $V$-paired basis for $C$ and applying the explicit construction in Definition~\ref{def:dual}; the $U$-tower condition is similar. The claim regarding local maps is standard. We thus show that $C \otimes_\stair C^\vee$ is locally trivial. We first produce a local map from $\stair$ to $C \otimes_\stair C^\vee$. Fix any basis $\{e_1, \ldots, e_m\}$ for $C$ and consider the element
\[
x_0 = \sum_{i=1}^m e_i \otimes e_i^\vee.
\]
We claim that this element is a cycle in $C \otimes_\stair C^\vee$. 

Indeed, for any $j$ and $k$, consider the coefficient of $e_j \otimes e_k^\vee$ in $\partial x_0$. A consideration of the tensor product rule for $\partial$ immediately shows that any contribution to this coefficient can only come from the two terms $\partial (e_j \otimes e_j^\vee)$ and $\partial (e_k \otimes e_k^\vee)$ in $\partial x_0$. Moreover, this only happens when $e_j$ appears in $\partial e_k$. Thus, suppose that $e_j$ appears in $\partial e_k$ with coefficient $c$. Then (keeping track of all signs) the first term contributes a coefficient of $(-1)^{|e_j|}(-1)(-1)^{|e_j|}c = -c$, while the second term just contributes a coefficient of $c$. These cancel, so $x_0$ is a cycle in $C \otimes_\stair C^\vee$. We can easily guarantee that $x_0$ is a $V$-tower generator by choosing $\{e_i\}_{i=1}^m$ to be a $V$-paired basis. Since $x_0$ has degree $(0,0)$, we obtain a local map from $\stair$ to $C \otimes_\stair C^\vee$ by sending the generator of $\stair$ to $x_0$. Dualizing this construction yields a local map from $(C \otimes_\stair C^\vee)^\vee \cong C^\vee \otimes_\stair C \cong C \otimes_\stair C^\vee$ to $\stair^\vee \cong \stair$.
\end{proof}

\begin{remark}
Lemma \ref{lem:dual-check} may also be proved in a basis-free manner, using the contraction map; see, for example, \cite[Proof of Lemma 3.14]{DHSThomcobord} for the analogous basis-free proof in the involutive setting.
\end{remark}

Since tensor products preserve local equivalence, we are finally ready to make the following definition:

\begin{definition}
Let $\KL$ denote the set of (right) local equivalence classes of (totally) knotlike complexes over $\stair$, with the operation induced by $\otimes$.
\end{definition}

\begin{proposition}\label{prop:ab-group}
The pair $(\KL, \otimes)$ forms an abelian group.  
\end{proposition}

\begin{proof}
This follows immediately from Lemma~\ref{lem:product-check} and Lemma~\ref{lem:dual-check}.
\end{proof}

In later sections, it will also be helpful to have the following subclass of knotlike complexes.

\begin{definition}\label{def:xi}
Suppose there exists a skew-graded isomorphism of $\K$-algebras $\xi \co \R_U \rightarrow \R_V$.  Define a skew-graded involution on $\stair$ by sending $x \otimes y \in \R_U \otimes \R_V$ to $\xi^{-1}(y) \otimes \xi(x) \in \stair$; by abuse of notation we denote this by $\xi: \stair \to \stair$. Note that $\xi$ interchanges $\R_U \subset \stair$ and $\R_V \subset \stair$. We define the \textit{pullback complex $\xi^*(C)$ (of $C$ along $\xi$)} as follows.  The $\F$-chain complex of $\xi^*(C)$ is exactly $C$, but where the $(i,j)$-graded component $\xi^*(C)_{i,j}$ is now $C_{j,i}$, and where $z\in \stair$ acts on $c\in \xi^*(C)$ by $c\mapsto \xi(z) \cdot c$, in terms of the action of $\stair$ on $C$.  It is clear that if $C$ is right knotlike, then $\xi^*(C)$ is left knotlike, and so on.
\end{definition} 

The pullback of $C$ should be thought of as reflecting $C$ across the line $\gr_1 = \gr_2$ and then turning all of the $\R_U$-decorations into $\R_V$-decorations (and vice-versa) via $\xi$. In the usual knot Floer setting, this corresponds to reflection across the diagonal, combined with the map interchanging $U$ and $V$. 

\begin{definition}\label{def:symmetric}
We say that a totally knotlike complex $C$ is \textit{symmetric} (with respect to $\xi$) if $\xi^*(C)$ is isomorphic to $C$ via an $\stair$-equivariant, grading-preserving isomorphism. Note that $\xi^*(C)$ is tautologically isomorphic to $C$ via a $\stair$-skew-equivariant (via $\xi$), skew-graded isomorphism.
\end{definition}

Unfortunately, the property of being symmetric is not preserved under local equivalence, as one can always add an asymmetric summand that does not change the local equivalence class. Thus it is helpful to introduce the following weaker notion:

\begin{definition}\label{def:locallysymmetric}
We say that a totally knotlike complex $C$ is \textit{locally symmetric} (with respect to $\xi$) if $\xi^*(C)$ is locally equivalent to $C$. It is easily checked that the set of locally symmetric complexes is a subgroup of $\KL$, which we denote by $\KLs$. Here, we use the fact (to be shown in Lemma~\ref{lem:loceqsymmetric}) that a right-local equivalence between totally knotlike complexes is in fact also left-local. 
\end{definition}

\subsection{A total order on $\KL$}\label{sec:totalorder}


We now establish the central result of this section, which is that $\KL$ is a totally ordered abelian group. This depends on the following sequence of technical lemmas, regarding a special choice of lift from $C/(\fm_U, \fm_V)$ to $C$. 

Let $\cB$ and $\cB'$ be two bases for $\cB$. Recall that a choice of lift from, say, $C/(\fm_U, \fm_V)$ to $C/\fm_U$ depends on a choice of a basis. In what follows, we will consider lifts between $C/(\fm_U, \fm_V)$, $C/\fm_U$, $C/\fm_V$, and $C$ defined using both $\cB$ and $\cB'$. The quotient and section maps we will consider are recorded below, specifying which bases we are using in each lift:

\[
\begin{tikzcd}
	&& {C/\mathfrak{m}_V} \\
	C && {} && {C/(\mathfrak{m}_U,\mathfrak{m}_V)} \\
	&& {C/\mathfrak{m}_U}
	\arrow["{q_V}"', shift right=4, from=3-3, to=2-5]
	\arrow["{q_U}", shift left=4, from=1-3, to=2-5]
	\arrow["{p_U}", shift left=4, from=2-1, to=1-3]
	\arrow["{p_V}"', shift right=4, from=2-1, to=3-3]
	\arrow["{i_{U, \cB}}", shift right=2, from=1-3, to=2-1]
	\arrow["{i_{V, \cB'}}"', shift left=2, from=3-3, to=2-1]
	\arrow["{j_{U, \cB'}}", shift right=2, from=2-5, to=1-3]
	\arrow["{j_{V, \cB}}"', shift left=2, from=2-5, to=3-3]
	\arrow["\iota", shift left=1, from=2-5, to=2-1]
	\arrow["\pi", shift left=1, from=2-1, to=2-5]
\end{tikzcd}
\]
\vspace{0.1cm}

Here, we define the lift $\iota$ to be the composition $i_{U, \cB} \circ j_{U, \cB'}$, which we show below is equal to the composition $i_{V, \cB'} \circ j_{V, \cB}$. Roughly speaking, the reader should think that in order to specify a lift $\iota \colon C/(\fm_U, \fm_V) \to C$, we must pin down the $\fm_U$ and $\fm_V$ indeterminancy. We use the basis $\cB$ to pin down the $\fm_V$ indeterminancy and use the basis $\cB'$ to pin down the $\fm_U$ indeterminancy; the equality $i_{U, \cB} \circ j_{U, \cB'} = i_{V, \cB'} \circ j_{V, \cB}$ is the claim that these indeterminacies can be pinned down in either order. The following lemma makes this precise: 

\begin{lemma}\label{lem:liftinglemma}
Let $\cB$ and $\cB'$ be two bases for $C$. Then $i_{U, \cB} \circ j_{U, \cB'} = i_{V, \cB'} \circ j_{V, \cB}$. Furthermore, $\iota=i_{U, \cB} \circ j_{U, \cB'} = i_{V, \cB'} \circ j_{V, \cB}$ is uniquely characterized by the property that for any $v \in C/(\fm_U, \fm_V)$, when $\iota(v)$ is expressed in terms of $\cB$ the coefficients are all in $\cR_U$, and that when $\iota(v)$ is expressed in terms of $\cB'$ the coefficients are all in $\cR_V$.
\end{lemma}

Note that na\"ively, one would expect to need coefficients in all of $\stair$, rather than just $\cR_U$, in order to express $\iota(v)$ in terms of $\cB$, and similarly for $\cB'$.

\begin{proof}
The proof consists of two main steps:
\begin{enumerate}
	\item\label{it:liftinglemma1} We first show that $i_{U, \cB} \circ j_{U, \cB'} (v)$ is an $\cR_U$-linear combination of elements in $\cB$ and that it is also an $\cR_V$-linear combination of elements in $\cB'$.
	\item\label{it:liftinglemma2} We then show that any lift $\iota$ satisfying $\iota  (v)$ is an $\cR_U$-linear combination of elements in $\cB$ and an $\cR_V$-linear combination of elements in $\cB'$ must be equal to $i_{V, \cB'} \circ j_{V, \cB}$.
\end{enumerate}
Let $\cB = \{e_i \}$ and $\cB' = \{e'_i \}$. 

To show \eqref{it:liftinglemma1}, first express $v$ as a $\K$-linear combination of elements in $\pi(\cB')$:
\[ v = \sum b_i \pi(e'_i). \]
Then 
\[ j_{U, \cB'} (v) =  \sum b_i p_U(e'_i) \]
by the definition of $j_{U, \cB'}$, and 
\[ i_{U, \cB} \circ j_{U, \cB'} (v) = \sum b_i e'_i + \mu_V \cdot c \]
for some $\mu_V \in \fm_V$ and $c \in C$, since $i_{U, \cB} \circ p_U = \id_C \mod \fm_V$. (Note that a priori one may expect the right-hand side to be $\sum b_i e'_i + \sum \mu_{V,i} c_i$ with $\mu_{V,i} \in \fm_V$ and $c_i \in C$. Since $\cR_V$ is a graded valuation ring, the $\mu_{V,i}$ are totally ordered by divisibility and we may take $\mu_V$ to the minimum of the $\mu_{V,i}$ with respect to divisibility and $c = \sum  \mu_V^{-1}\mu_{V,i} c_i$.) Writing $c$ as an $\stair$-linear combination $\sum d_i e'_i$ of elements in $\cB'$ and using the fact that $\fm_U \cdot \fm_V  = 0$, it follows that $i_{U, \cB} \circ j_{U, \cB'} (v)$ is an $\cR_V$-linear combination of elements in $\cB'$.

We now consider $v$ in terms of $\pi(\cB)$, i.e., $v=\sum a_i \pi(e_i)$, where $a_i \in \K$. We have 
\[ j_{U, \cB'} (v) =  \sum a_i p_U(e_i) + \mu_U \cdot c' \] 
for some $\mu_U \in \fm_U$ and $c' \in C/\fm_V$, since $j_{U, \cB'} \circ q_U = \id_{C/\fm_V} \mod \fm_U$. Writing $c'$ as an $\cR_U$-linear combination $\sum d'_i p_U(e_i)$ of elements in $p_U(\cB)$, we have that
\[ i_{U, \cB} \circ j_{U, \cB'} (v) =  \sum a_i e_i + \mu_U \cdot \sum d'_i e_i, \]
since $i_{U, \cB} \circ p_U$ is the identity on each $e_i$. Thus, $i_{U, \cB} \circ j_{U, \cB'} (v)$ is an $\cR_U$-linear combination of elements in $\cB$. This completes the proof of \eqref{it:liftinglemma1}.

To show \eqref{it:liftinglemma2}, suppose that $\iota(v)$ is an $\cR_V$-linear combination of elements in $\cB'$ and an $\cR_U$-linear combination of elements in $\cB$, in which case we can write
\begin{equation}\label{eqn:iotapi}
\iota (v) =  \sum b_i e'_i + \mu_V \cdot \sum d_i e'_i =  \sum a_i e_i + \mu_U \cdot \sum d'_i e_i,
\end{equation}
where $a_i, b_i \in \K$, $d_i \in \cR_V$ and $d'_i \in \cR_U$.
We have that 
\begin{equation}\label{eqn:piiotapi}
 	p_V \circ \iota (v) =  \sum a_i p_V(e_i) = j_{V, \cB} (v)
\end{equation}
where the first equality follows from quotienting by $\fm_U$ and the second equality follows from the definition of $j_{V, \cB}$. We then have
\begin{align*}
	 i_{V, \cB'} \circ  j_{V, \cB} (v) &= i_{V, \cB'} \circ p_V \circ \iota (v) \\
	 		&= i_{V, \cB'} \circ p_V \Big( \sum b_i e'_i + \mu_V \cdot \sum d_i e'_i \Big) \\
	 		&= \sum b_i e'_i + \mu_V \cdot \sum d_i e'_i							
\end{align*} 
the first equality follows from applying $i_{V, \cB'}$ to \eqref{eqn:piiotapi}, the second equality from \eqref{eqn:iotapi}, and the third equality from the fact that $i_{V, \cB'} \circ p_V$ is the identity on each $e'_i$. Hence $i_{V, \cB'} \circ  j_{V, \cB} = \iota$, as desired.
\end{proof}

Next, we prove that a basis always lifts to a basis.
\begin{lemma}\label{lem:basislift}
Let $\cB_0$ be a basis for $C/(\fm_U, \fm_V)$. Any graded lift of $\cB_0$ to $C$ is a basis for $C$.
\end{lemma}

\begin{proof}
Let  $\tilde{\cB_0} = \{f_i\}$ be a lift of $\cB_0$ to $C$ and $\cB = \{ e_i \}$ a basis for $C$. Note that both $\pi(\tilde{\cB_0})$ and $\pi(\cB)$ are bases for $C/(\fm_U, \fm_V)$. Since $\cB$ is a basis for $C$, there is a matrix $A$ with coefficients in $\stair$ taking $\{e_i\}$ to $\{f_i\}$. We wish to show that $A$ is invertible. 

We first observe that $\det A \neq 0$. Indeed, $\pi(A)$ is a matrix taking $\{\pi(e_i)\}$ to $\{\pi(f_i)\}$, and $\det(\pi(A)) \neq 0$ since $\pi(\tilde{\cB_0})$ and $\pi(\cB)$ are bases for $C/(\fm_U, \fm_V)$.

We now show that $\gr (\det A) = (0, 0)$. The following is a straightforward linear algebra argument. By possibly reordering, we may assume that $\gr (\pi(e_i)) = \gr (\pi(f_i))$, and hence $\gr(e_i) = \gr(f_i)$. Now, we consider the grading of the entry $a_{ij}$ in $A$: it is clear that $\gr (a_{ij}) = \gr(e_i) - \gr (f_j)$, or equivalently, $\gr (a_{ij}) = \gr(e_i) - \gr (e_j)$. We recall that $\det A = \sum_{\sigma \in S_n} (\sgn (\sigma)) \prod_i a_{i \sigma(i)}$, and observe that 
\begin{align*}
	\gr \Big(\prod_i a_{i \sigma(i)}\Big) &= \sum_i \gr (a_{i \sigma(i)}) \\
					&= \sum_i (\gr(e_i) - \gr(e_{\sigma(i)}).
\end{align*}
Since $\sigma$ is a permutation, it follows that the last sum is zero, as desired. It follows that $A$ is invertible, since our rings are type-zero. 
\end{proof}

The next lemma shows that if we choose $\cB$ to be a $V$-paired basis and we project and then lift $\cB$, it remains a $V$-paired basis. (The proof similarly works for a $U$-paired basis $\cB'$, but we do not need that fact.)

\begin{lemma}\label{lem:liftsimp}
Let $\cB$ be a $V$-paired basis for $C$. Let $\iota$ be as in Lemma \ref{lem:liftinglemma}, defined with respect to $\cB$ and some other basis $\cB'$. Then $\iota \circ \pi (\cB)$ is still a $V$-paired basis.
\end{lemma}

\begin{proof}
By Lemma \ref{lem:basislift}, we have that $\iota \circ \pi (\cB)$ is still a basis for $C$. We must show that it is still $V$-paired.

Let $e_i \in \cB$. Then by Lemma \ref{lem:liftinglemma}, we have $\iota \circ \pi(e_i) = e_i + \mu_U c$ for some $\mu_U \in \fm_U$ and $c \in C$. Suppose $\d_V e_i = \mu_V e_j$ for some $\mu_V \in \fm_V$ and $e_j \in \cB$. (We allow $\mu_V$ to be zero.) Then
\begin{align*}
	\d_V (\iota \circ \pi(e_i)) & = \d_V ( e_i + \mu_U c) \\
			&= \d_V (e_i) \\
			&= \mu_V e_j \\
			&= \mu_V \cdot \iota \circ \pi(e_j) 
\end{align*}
where the second equality follows from the fact that $C$ is reduced and $\fm_U \cdot \fm_V = 0$, and the last equality follows from the fact that $\iota \circ \pi (e_j) = e_j \mod \fm_U$ and that $\fm_U \cdot \fm_V = 0$.
\end{proof}

Armed with these lemmas, we are now ready to prove the following technical result:

\begin{lemma}\label{lem:totalorder}
Let $C$ be a knotlike complex. Suppose there does not exist a local map $f \co \stair \to C$. Then there exists a local map $g \co C \to \stair$.
\end{lemma}

In the lemma statement and proof, when we refer to a local map, we mean a right local map.

\begin{proof}
We begin by applying Lemma~\ref{lem:pairedbasis} to obtain a $V$-paired basis 
\[
\cB =\{x, y_i, z_i\}_{i = 1}^n
\]
and a $U$-paired basis
\[
\cB' =\{x', y'_i, z'_i\}_{i = 1}^n
\]
for $C$. A rough roadmap of the proof will be as follows:
\begin{itemize}
	\item begin by building a special basis for $Z=\Span_\K\{ \pi(z_i), \pi(z'_i)\}$,
	\item use the hypothesis (i.e., the lack of a local map $f \co \stair \to C$) to show that $\pi(x)$ is not in $Z$ and adjoin $\pi(x)$ to this basis,
	\item use these facts to build a special basis $\cB_0$ for $C/(\fm_U, \fm_V)$,
	\item lift $\cB_0$ to a basis for $C$, using the lift from Lemma \ref{lem:liftinglemma} with respect to $\cB$ and $\cB'$ above, and
	\item use $\iota(\cB_0)$ to define a local map $C \to \stair$.
\end{itemize}

We now proceed to implement this strategy.\\
\\
\noindent
\textit{Step 1: Working in $C/(\fm_U, \fm_V)$.} We begin by building a basis for the $\K$-vector space $Z=\Span_\K\{ \pi(z_i), \pi(z'_i)\}$. Since $\cB=\{x, y_i, z_i\}$ is a basis for $C$, it follows that $\pi(\cB)$ is a basis for $C/(\fm_U, \fm_V)$. Take the linearly independent set $\{\pi(z_i)\}_{i=1}^{n}$, and adjoin elements of the form $\pi(z'_j)$ to obtain a basis $\{\pi(z_i), \pi(z'_{j_s})\}$ for $Z$. Here, the $\{\pi(z'_{j_s})\}$ form a subset of $\{\pi(z'_{j})\}$, where $s = 1, \dots, r$, for some $r \leq n$. \\
\\
\noindent
\textit{Relation with $\pi(x)$:} We now show that $\pi(x)$ does not lie in $Z$. Indeed, suppose to the contrary that $\pi(x) = \sum a_i \pi(z_i) + \sum b_s \pi(z'_{j_s})$. Apply $\iota$ to both sides. (Here, $\iota$ may be taken to be any lift to $C$; if the reader prefers, they may take $\iota$ as in Lemma \ref{lem:liftinglemma}.)  Then 
\[ x -  \sum a_i z_i + \mu_U w_1 =  \sum b_s z'_{j_s} + \mu_V w_2 \]
for some $\mu_U \in \fm_U, w_1 \in C, \mu_V \in \fm_V$, and $w_2 \in C$. Now, the left-hand side is a $\d_V$-cycle, since (a) $x$ and $z_i$ are $\d_V$-cycles, and (b) $\d_V(\mu_U w_1)= \mu_U \d_V w_1 = 0$ using the fact that $C$ is reduced and $\fm_U \cdot \fm_V = 0$. The right-hand side is a $\d_U$-cycle, since (a) the $z'_{j_s}$ are $\d_U$-cycles, and (b) $\d_U(\mu_V w_2) = \mu_V \d_U(w_2) = 0$ using the fact that $C$ is reduced and $\fm_U \cdot \fm_V = 0$. Hence $x - \sum a_i z_i + \mu_U w_1 $ is a cycle in $C$ which is easily checked to be non-$\cR_V$-torsion. We may then construct a (right) local map $f \colon \stair \to C$ by mapping a generator of $\stair$ to $x - \sum a_i z_i + \mu_U w_1 $, a contradiction. 

Extend the linearly independent set $\{ \pi(x), \pi(z_i), \pi(z'_{j_s}) \}$ to a basis 
\[ \cB_0=\{ \pi(x), \pi(z_i), \pi(z'_{j_s}), t_k \} \]
for $C/(\fm_U, \fm_V)$.
\\
\\
\noindent
\textit{Step 2: Working in $C$.} Let $\cB$ and $\cB'$ be the $V$- and $U$-paired bases, respectively, for $C$. Let $\iota \colon C/(\fm_U, \fm_V) \to C$ be the lift specified by Lemma \ref{lem:liftinglemma} with respect to $\cB$ and $\cB'$. Now Lemma \ref{lem:basislift} implies that $\iota(\cB_0)$ is a basis for $C$.

We claim that $\im \d$ is contained in the $\stair$-span of 
\[ \cB_1 = \{ \iota \circ \pi(z_i), \iota \circ\pi(z'_{j_s})\}. \]
Indeed, $\im \d_V$ is contained in the span of $\{ \iota \circ \pi(z_i) \}$ since (a) Lemma \ref{lem:liftinglemma} implies that $\iota \circ \pi(z_i) = z_i + \mu_U w$ for some $\mu_U \in \fm_U$ and $w \in C$, (b) the $z_i$ span $\im \d_V$, and (c) $C$ is reduced and $\fm_U \cdot \fm_V = 0$.

Similarly, we claim that $\im \d_U$ is contained in the span of $\{ \iota \circ \pi(z_i), \iota \circ\pi(z'_{j_s})\}$. By construction, 
\[ \Span_\K \{ \pi(z'_i)\} \subseteq \Span_\K \{ \pi(z_i), \pi(z'_{j_s})\}, \] 
so
\[ \pi(z'_i) = \sum a_i \pi(z_i) + \sum b_s \pi(z'_{j_s}), \]
for some $a_i, b_s \in \K$. Then 
\[ \sum a_i \iota \circ \pi(z_i) + \sum b_s \iota \circ \pi(z'_{j_s}) = \iota \circ \pi(z'_i) = z'_i + \mu_V w \]
for some $\mu_V \in \fm_V$ and $w \in C$, where the first equality follows from the $\K$-linearity of $\iota$ and the second equality from Lemma \ref{lem:liftinglemma}. The proof now follows as in the $\im \d_V$ case; namely, the $z'_i$ span $\im \d_U$, the complex $C$ is reduced, and $\fm_U \cdot \fm_V = 0$. Hence $\im \d \in \Span_\stair \cB_1$.  

To complete the proof, we finally let 
\[ C' = \Span_\stair \{ \iota \circ \pi(z_i), \iota \circ\pi(z'_{j_s}), \iota(t_k)\}. \]
We have seen that $\im \d \subset \Span_\stair \cB_1 \subset C'$ and that $\iota \circ \pi (x)$ is not contained in $C'$. By Lemma \ref{lem:liftsimp}, the element $\iota \circ \pi (x)$ is a $V$-tower class. It is then clear that up to grading shift, $C/C' \cong \stair$ (with the class of $\iota \circ \pi (x)$ generating the copy of $\stair$) and that the quotient map from $C$ to $C/C'$ is the desired local map.
\end{proof}


\begin{proposition}\label{prop:totalorder}
	The relation $\leq$ defines a total order on $\KL$ compatible with the group structure.
\end{proposition}

\begin{proof}[Proof of Proposition \ref{prop:totalorder}]
	We need to show totality of $\leq$. Let $C_1$ and $C_2$ be two knotlike complexes. Consider $C_1 \otimes C_2^\vee$. By Lemma \ref{lem:totalorder}, we have that either $C_1 \otimes C_2^\vee \geq \stair$ or  $C_1 \otimes C_2^\vee \leq \stair$. By tensoring with $C_2$, this shows that either $C_1 \geq C_2$ or $C_1 \leq C_2$, as desired. Since the product of local maps is local, we have that if $C_1 \leq C_2$, then $C_1 \otimes C_3 \leq C_2 \otimes C_3$; hence $\leq$ is compatible with the group structure.
\end{proof}

	
	

\section{Standard complexes and their properties}\label{sec:standards}
In this section, we define an important family of knotlike complexes called standard complexes. The reader should compare this section with \cite[Section 4]{DHSTmoreconcord}, which carries out a special case of this construction. 


\subsection{Standard complexes}\label{subsec:standards}
Throughout, fix a grid ring $\stair = \stair(\R_U, \R_V)$.

\begin{definition}\label{def:standard}
Let $n \in \N$ be even. Let $(b_1, \ldots, b_n)$ be a sequence such that 
\begin{enumerate}
\item $b_i \in \Gamma_+(\R_U) \cup \Gamma_-(\R_U)$ for $i$ odd; and,
\item $b_i \in \Gamma_+(\R_V) \cup \Gamma_-(\R_V)$ for $i$ even.
 \end{enumerate}
 The \emph{standard complex of type $(b_1, \dots, b_n)$}, denoted by $C(b_1, \dots, b_n)$, is the knotlike complex freely generated over $\cS$ by
\[ \{x_0, x_1, \dots, x_n\}. \]
We call $n$ the \emph{length} of the standard complex and $\{x_i\}_{i = 0}^n$ the \emph{preferred basis}, or sometimes the \textit{preferred generators}. The differential on $C = C(b_1, \dots, b_n)$ is defined as follows. Each $b_i$ is an equivalence class in either $\Gamma(\R_U)$ or $\Gamma(\R_V)$; choose explicit representatives for the $b_i$ in either $\loc(\R_U)$ or $\loc(\R_V)$, as appropriate. By abuse of notation, we denote these also by $b_i$. For each $i  > 0$ odd, define
\begin{align*}
&\d_U x_{i-1} = |b_i| x_{i}  \quad \text{ if } b_i \lebang 1 \\
&\d_U x_i = b_i x_{i-1} \ \ \quad \text{ if } b_i \gebang 1, 
\end{align*}
while for $i > 0$ even, define
\begin{align*}
&\d_V x_{i-1} = |b_i| x_{i} \quad \text{ if } b_i \lebang 1 \\
&\d_V x_{i} = b_i x_{i-1} \ \ \quad \text{ if } b_i \gebang 1. 
\end{align*}
All other differentials are zero. Note that each $|b_i|$ lies in $(\fm_U, \fm_V)$, so $C$ is reduced. It is easily checked that the isomorphism type of $C$ does not depend on the choice of representative for each $b_i$, since any two representatives differ by multiplication by an element of $\K^\times$. We will often think of the $b_i$ as elements of $\loc(\R_U)$ or $\loc(\R_V)$, defined up to multiplication by $\K^\times$. Note that $x_0$ generates $H_*(C \otimes_\X \R_V)/\R_V\text{-torsion}$, and similarly $x_n$ generates $H_*(C \otimes_\X \R_U)/\R_U\text{-torsion}$. Following Remark~\ref{rem:gradingshift}, there is a unique grading on $C$ which makes it into a totally knotlike complex; this has $\gr_1(x_0) = 0$ and $\gr_2(x_n) = 0$. Explicitly, the definition of $\partial$ implies
\begin{align}\label{eq:gr1diff}
\gr_1(x_n) - &\gr_1(x_0) = \sum_{i=1}^n \sgn (b_i) + \sum_{i=1}^n \gr_1(b_i)
\end{align}
and
\begin{align}\label{eq:gr2diff}
\gr_2(x_n) - \gr_2(x_0) = \sum_{i=1}^n \sgn (b_i) + \sum_{i=1}^n \gr_2(b_i).
\end{align}
Hence the grading shift in question is such that $\gr_1(x_n)$ is the right-hand side of \eqref{eq:gr1diff}, while $\gr_2(x_0)$ is the negative of the right-hand side of \eqref{eq:gr2diff}.
\end{definition}

The reader should think of the generators of $C(b_1, \dots, b_n)$ as being connected by arrows recording the action of the differential. Each pair of generators $x_i$ and $x_{i+1}$ are connected by $|b_{i+1}|$-arrows. The direction of each arrow is determined by whether $b_{i+1}\in \Gamma_+(\R_U)$ or $\Gamma_-(\R_U)$ (for $i$ even) or $b_{i+1}\in \Gamma_+(\R_V)$ or $\Gamma_-(\R_V)$ (for $i$ odd). If $b_{i+1}$ is positive, then the arrow goes from $x_{i+1}$ to $x_i$, and if $b_{i+1}$ is negative, then the arrow goes from $x_i$ to $x_{i+1}$. 

\begin{definition}
We define the trivial standard complex $C(0) \cong \cS$ to be the complex generated over $\cS$ by a single element $x_0$ in grading $(0,0)$.
\end{definition}

\begin{lemma}\label{lem:dual}
The dual of $C(b_1, \dots, b_n)$ is the standard complex $C(b_1^{-1}, \dots, b_n^{-1})$.
\end{lemma}

\begin{proof}
This is a straightforward consequence of the explicit procedure of Definition~\ref{def:dual}, keeping in mind that standard complex parameters are defined up to multiplication by elements of $\K^\times$.
\end{proof}
\noindent
Note that the negative signs $-b_i$ in \cite[Lemma 4.9]{DHSTmoreconcord} correspond to inverses $b_i^{-1}$ in Lemma~\ref{lem:dual}, as we use multiplicative rather than additive notation.

\begin{lemma}\label{ex:symmetricstandard}
Let $\xi$ be a skew-graded involution of $\stair$, as in Definition~\ref{def:xi}. Then $C(b_1, \ldots, b_n)$ is symmetric with respect to $\xi$ if and only if $b_i = \xi(b_{n+1-i}^{-1})$ for each $i$. Moreover, $C(b_1, \ldots, b_n)$ is symmetric if and only if it is locally symmetric.
\end{lemma}
\begin{proof}
It is straightforward to check that the pullback complex of $C(b_1, \ldots, b_n)$ is the standard complex $C(\xi(b_n^{-1}), \ldots, \xi(b_1^{-1}))$. Hence the equality $b_i = \xi(b_{n+1-i}^{-1})$ certainly implies that our complex is symmetric. The converse (as well as the second claim) follows immediately from the fact that the $b_i$ completely parameterize standard complexes up to local equivalence; this is established in the next subsection.
\end{proof}

\subsection{Ordering standard complexes}\label{subsec:ordering-standard}
We now show that the total order on the set of standard complexes can be understood explicitly in terms of the parameters $b_i$. To do this, we consider the lexicographic order on the set of parameter sequences induced by $\leqbang$. We take the convention that in order to compare two sequences of different lengths, we append sufficiently many trailing copies of $1$ to the shorter sequence so that the sequences have the same length. 

Explicitly, given two (distinct) sequences $(a_1, \dots, a_m)$ and $(b_1, \dots, b_n)$, we first locate the first index $i$ such that the parameters $a_i$ and $b_i$ differ. Again, to be pedantic: we may either interpret $a_i$ and $b_i$ as equivalence classes in the same valuation group, or we may consider them as actual elements of $\stair$. In the latter case, the notion $a_i = b_i$ should be interpreted as up to multiplication by an element of $\K^\times$. If one sequence appears as a prefix of the other, then one of $a_i$ and $b_i$ will be equal to $1$. We then declare $(a_1, \dots, a_m) \lebang (b_1, \dots, b_n)$ if and only if $a_i \lebang b_i$.

The central claim of this subsection is that the lexicographic order on the set of standard complexes coincides with the total order of Proposition~\ref{prop:totalorder}. Note that implicitly, this means two standard complexes are locally equivalent if and only if they have the same parameter sequence.

\begin{proposition}\label{prop:lex}
Standard complexes are ordered lexicographically as sequences with respect to the total order on $\KL$.
\end{proposition}

\begin{proof}[Proof of Proposition \ref{prop:lex}]
The proof of Proposition~\ref{prop:lex} is very similar to the proof of \cite[Proposition 4.10]{DHSTmoreconcord}; we will not reproduce all of the details here, but we list some constituent lemmas from which the proof follows. Use Lemma~\ref{lem:leqlex} and Lemma~\ref{lem:>lex} below.
\end{proof}

\begin{lemma}\label{lem:leqlex}
Let $(a_1, \dots, a_m) \leqbang (b_1, \dots, b_n)$ in the lexicographic order. Then $C(a_1, \dots, a_m) \leq C(b_1, \dots, b_n)$ in $\KL$.
\end{lemma}

\begin{proof}
Assume that $(a_1, \dots, a_m) \lebang (b_1, \dots, b_n)$. Let $k$ be such that $a_i = b_i$ for $1 \leq i < k$ and $a_k \lebang b_k$; we assume for simplicity that $k \leq \min\{m, n\}$. For $\{x_i\}_{i=0}^m$ and $\{y_i\}_{i = 0}^n$ the preferred basis elements of $C(a_1, \dots, a_m)$ and $C(b_1, \dots, b_n)$, define 
\[ 
f \colon C(a_1, \dots, a_m) \to C(b_1, \dots, b_n) 
\]
by
\begin{align*}
	f(x_i) = \begin{cases}
		y_i &\text{ if } 0 \leq i < k \\
		0 &\text{ if } i > k.
	\end{cases}
\end{align*}
In order to define $f(x_k)$, there is some casework:
\begin{enumerate}
	\item If $a_k \lebang b_k \lebang 1$, define $f(x_k) = a_kb_k^{-1} y_k$. 
	\item If $a_k \lebang 1 \lebang b_k$, define $f(x_k) = 0$. 
	\item If $1 \lebang a_k \lebang b_k$, (again) define $f(x_k) = a_kb_k^{-1} y_k$. 
\end{enumerate}
The reader may verify that this makes $f$ into a chain map which is evidently local. The case $k > \min \{ m ,n \}$ follows via similar casework. 
\end{proof}

\begin{lemma}\label{lem:localsupport}
Let $C(a_1, \dots, a_m)$ and $C(b_1, \dots, b_n)$ be standard complexes with preferred bases $\{x_i\}_{i=0}^m$ and $\{ y_i\}_{i=0}^n$, respectively. Suppose that $a_i = b_i$ for all $1 \leq i \leq k$ and that $f \co C(a_1, \dots, a_m) \to C(b_1, \dots, b_n)$ is a local map. Then $f(x_i)$ is supported by $y_i$ for all $0 \leq i \leq k$.
\end{lemma}

\begin{proof}
Omitted; see \cite[Lemma 4.2]{DHSTmoreconcord}.
\end{proof}

\begin{lemma}\label{lem:>lex}
Let $(a_1, \dots, a_m) \gebang (b_1, \dots, b_n)$ in the lexicographic order. Then there is no local map from $C(a_1, \dots, a_m)$ to $C(b_1, \dots, b_n)$.
\end{lemma} 

\begin{proof}
Omitted; see \cite[Lemma 4.3]{DHSTmoreconcord}. Use Lemma~\ref{lem:localsupport}.
\end{proof}

\subsection{Semistandard complexes}
We will also find it useful to have the following generalization of standard complexes:

\begin{definition}\label{def:semis}
Let $n \in \N$ be odd. Let $(b_1, \ldots, b_n)$ be a sequence such that:
\begin{enumerate}
\item $b_i \in \Gamma_+(\R_U) \cup \Gamma_-(\R_U)$ for $i$ odd; and,
\item $b_i \in \Gamma_+(\R_V) \cup \Gamma_-(\R_V)$ for $i$ even.
 \end{enumerate}
The \emph{semistandard complex of type $(b_1, \dots, b_n)$}, denoted by $C'(b_1, \dots, b_n)$, is defined using the obvious generalization of Definition~\ref{def:standard}. We may also (somewhat artificially) define this as the subcomplex of the standard complex $C(b_1, \dots, b_n, b_{n+1})$ generated by $\{x_0, x_1, \dots, x_n\}$, for any choice of $b_{n+1} \gebang 1$. (Up to grading shift, the choice of $b_{n+1}$ clearly does not affect the definition.) We again call $n$ the \emph{length} of the semistandard complex and $\{x_i\}_{i = 0}^n$ the \emph{preferred basis}. For concreteness, we normalize gradings so that $\gr(x_0) = (0,0)$.
\end{definition}

We usually use the symbol $'$ to distinguish semistandard complexes from standard complexes, but may omit it in situations where our arguments apply equally well to semistandard complexes as standard complexes. Note, however, that a semistandard complex $C'$ is \textit{not} a knotlike complex of any kind. Indeed it is easily checked that $H_*(C' \otimes_\stair \R_V)/\R_V\text{-torsion}$ is isomorphic to two copies of $\R_V$. Hence we cannot use Remark~\ref{rem:gradingshift} to normalize gradings on semistandard complexes, in contrast to Definition~\ref{def:standard}.

\begin{definition}
Let $C'$ be a semistandard complex and $C$ be a standard complex. Let $f$ be a $\gr_1$-preserving, $\gr_2$-homogenous, $\cS$-equivariant chain map
\[ 
f \co C' \to C. 
\]
We say $f$ is \emph{local} if the class of $f(x_0)$ generates $H_*(C \otimes_\stair \R_V)/\R_V\text{-torsion}$.
\end{definition}

\subsection{Short maps}
In this subsection, we introduce several technical results that will be useful later in the paper. These deal with maps from standard or semistandard complexes which satisfy the chain map condition on all generators except for (possibly) the last:

\begin{definition}
Let $C_1=C(b_1, \dots, b_n)$ be a standard or semistandard complex and $C_2$ be a knotlike complex. Let $f \co C_1 \rightarrow C_2$ be a $\gr_1$-preserving, $\gr_2$-homogeneous $\cS$-module map. We call $f$ a \emph{short map} if:
\begin{enumerate}
\item $f \d (x_i) = \d f (x_i) $ for $0 \leq i \leq n-1$; and,
\item\label{it:shortmap} $f \d_V(x_n) = \d_V f(x_n) $ in the case that $C_1$ is standard, or \\
$f \d_U(x_n) = \d_U f(x_n) $ in the case that $C_1$ is semistandard.
\end{enumerate}
 
Note that condition \eqref{it:shortmap} is strictly weaker than $f \d (x_n) = \d f (x_n)$, and indeed a short map $f$ will fail to be a chain map if $f \d_U(x_n) \neq \d_U f(x_n)$ in the case $C_1$ is standard, or $f \d_V(x_n) \neq \d_V f(x_n)$ in the case that $C_1$ is semistandard. More informally, a short map $f$ can fail to be a chain map at $x_n$ in the direction away from $x_0$.
 
We indicate the presence of a short map by the notation
\[ 
f \co C_1 \leadsto C_2. 
\]
Moreover, we say that $f$ is a \emph{short local map} if the homology class of $f(x_0)$ generates $H_*(C_2 \otimes_\stair \R_V)/\R_V\text{-torsion}$. In the case where $C_1$ is standard, this is the same as saying $f$ induces an isomorphism on $H_*(C_i \otimes_\stair \R_V)/\R_V \text{-torsion}$.
\end{definition}

We now come to two important technical lemmas. For the proof of the first, it will be useful to have a construction reminiscent of the pullback of a standard complex:

\begin{definition}
Let $C = C(b_1, \ldots, b_n)$ be a standard or semistandard complex. We define the \textit{reverse} of $C$ to be the standard or semistandard complex
\[
r(C) = C(b_n^{-1}, \ldots, b_1^{-1}).
\]
Here, there is a slight discrepancy in that if $C$ is a standard complex over $\stair(\R_U, \R_V)$, then the first parameter in the reversed complex lies in $\fm_V$, rather than $\fm_U$. Hence technically we must modify Definition~\ref{def:standard} in the obvious way; alternatively, we may formally consider $r(C)$ to be over the grid ring $\stair(\R_V, \R_U)$.
\end{definition}

To understand the reversed complex, observe that the basis $\{x_i\}_{i=0}^n$ of $C$ is enumerated in a preferred order, from $x_0$ to $x_n$. If we traverse the arrows of $C$ in the reverse order, from $x_n$ to $x_0$, then we obtain the sequence $\smash{(b_n^{-1}, \ldots, b_1^{-1})}$. Hence we may equally well view $C$ as being constructed from the sequence $\smash{(b_n^{-1}, \ldots, b_1^{-1})}$ as in Definition~\ref{def:standard} or Definition~\ref{def:semis}, except with the grading convention that the initial generator of the new construction has fixed grading $\gr(x_n)$. Thus, up to grading shift, $C$ and $r(C)$ are actually exactly the same complex; we write the latter when we wish to think of the generators of $C$ as being enumerated in the reverse order. Note that a homogenous chain map from $C_1$ to $C_2$ thus gives a homogenous chain map from $r(C_1)$ to $r(C_2)$ (and vice-versa), simply by shifting the grading.

The next lemma is a technical result; see e.g.\ the discussion preceding \cite[Lemma 5.11]{DHSThomcobord} for motivation.

\begin{lemma}[Merge Lemma]\label{lem:blebangc}
Let $C_1 = C(a_1, \ldots, a_m)$ and $C_2 = C(b_1, \ldots, b_n)$ be two standard complexes or two semistandard complexes with preferred bases $\{x_i\}_{i=0}^m$ and $\{y_i\}_{i=0}^n$. Let $C$ be any knotlike complex, and suppose we have two short maps
\[ 
f \co C_1 \leadsto C \quad \text{ and } \quad g \co C_2 \leadsto C.
\]
Assume that $\gr f(x_m)=\gr g(y_n)$, and suppose we have the inequality $r(C_1) > r(C_2)$. Then for any  $c_1$ and $c_2$ in $\K$, there is a short map 
\[
h \colon C_2 \leadsto C
\]
with the property that 
\[
h(y_n) = c_1 f(x_m) + c_2 g(y_n).
\] 
Moreover, if $g$ is local and $c_2 \neq 0$, then $h$ is local.
\end{lemma}
\begin{proof}
The proof of this lemma is essentially equivalent to the proof of Lemma \cite[Lemma~6.8]{DHSTmoreconcord}, but with some notational improvements to make it more transparent. The map $h$ will be defined as a composition
\[
h = (c_1f\oplus c_2g)\circ j 
\] 
where
\[
j \colon C_2 \to C_1 \oplus C_2.
\]
We define the map $j$ itself as the direct sum of two maps $j = j_1 \oplus j_2$. The second summand $j_2$ is just the identity map on $C_2$. For the first summand, consider the reversed complexes $r(C_1)$ and $r(C_2)$. By Lemma~\ref{lem:leqlex} and its analog for semistandard complexes, the inequality $r(C_2) < r(C_1)$ gives a homogenous chain map from $r(C_2)$ to $r(C_1)$. (This map is in fact local, but this will not be relevant here.) By the discussion preceding the lemma, we thus obtain a homogenous chain map $j_1 \co C_2 \rightarrow C_1$ by shifting the grading.

An examination of the proof of Lemma~\ref{lem:leqlex} shows that $j_1(y_n) = x_m$. More generally, if $k$ is the least index for which $a_{m-k} \neq b_{n-k}$, then the proof of Lemma~\ref{lem:leqlex} (combined with some careful tracking of indices) shows that
\[
j_1(y_{n-i}) = \begin{cases}
		x_{m-i} &\text{ if } 0 \leq i < k + 1\\
		0 &\text{ if } i > k + 1.
	\end{cases}
\]
It is easy to check that $h$ is a short map, using the fact that $j$ is a chain map and $f$ and $g$ are short. Note that $j_1$ and $j_2$ are homogenous maps which in general will not have the same grading shift. However, the compositions $f \circ j_1$ and $g \circ j_2$ do have the same grading shift, as can be seen by evaluating them both on $y_n$ and using the fact that $\gr(f(x_m)) = \gr(g(y_n))$. In particular, $h$ is relatively graded with the same grading shift as $g \circ j_2 = g$, so $h$ is $\gr_1$-preserving and $\gr_2$-homogenous. It remains to show that if $g$ is local, then $h$ satisfies the $V$-tower condition. This will require some casework depending on the value of $k$.

If $k < n - 1$, then $j(y_0) = 0 \oplus y_0$ by the above. Composing with $c_1f \oplus c_2g$ shows that $h(y_0)$ is a $V$-tower class. If $k = n - 1$, then a further examination of the proof of Lemma~\ref{lem:leqlex} (again combined with careful tracking of indices) shows that $j_1(y_0)$ is either zero or equal to $\smash{b_1^{-1} a_{m-n+1} x_{m-n}}$, where $\smash{b_1^{-1} a_{m-n+1}}$ is an element of $\fm_U$. It follows that
\[
h(y_0) = c_1f(j_1(y_0)) + c_2g(y_0) = c_2g(y_0) \bmod \fm_U.
\]
The fact that $g$ is local thus implies that $h$ is local. Finally, suppose $k = n$. Then $m > n$ and $a_{m-n} \lebang 1$. In this case we have $j_1(y_0) = x_{m-n}$. Noting that $m = n \bmod 2$, the fact that $a_{m-n} \lebang 1$ means that some $\fm_V$-multiple of $x_{m-n}$ lies in the image of $\partial_V$. It follows from this that 
\[
h(y_0) = c_1f(j_1(y_0)) + c_2g(y_0) = c_1f(x_{m-n}) + c_2g(y_0)
\]
is a $V$-tower class, being the sum of a $V$-tower class with a $V$-torsion class. This completes the proof.
\end{proof}

The next lemma is a generalization of \cite[Lemma~4.19]{DHSTmoreconcord} to complexes over $\stair$. Note that there is an error in the grading argument used in the proof in \cite{DHSTmoreconcord}; fortunately, the proof below does not rely on this grading argument, and thus provides a corrected proof of \cite[Lemma~4.19]{DHSTmoreconcord}.

\begin{lemma}[Extension Lemma]\label{lem:shortextension}
Let $C(b_1, \ldots, b_n)$ be a standard or semistandard complex, and let $C$ be any knotlike complex. Suppose we have a short map
\[ f \co C(b_1, \dots, b_n) \leadsto C. \]
Then there exists some complex of the form $C(b_1, \ldots, b_n, b_{n+1}, \ldots, b_m)$ which admits a genuine chain map
\[ g \co C(b_1, \dots, b_n, b_{n+1}, \dots, b_m) \to C. \]
This complex can be chosen to be either standard or semistandard. Moreover, if $f$ is local, then $g$ is local.
\end{lemma}

\begin{proof}
We consider the case when $C(b_1, \ldots, b_n)$ is a standard complex; the proof for semistandards is similar. Consider $f(x_n)$. If $\d_U f(x_n) = 0$, then $f$ is already a chain map. If desired, we may extend the domain of $f$ to the semistandard complex $C(b_1, \dots, b_n, 1/\mu)$ for any $\mu \in \fm_U$ by setting $f(x_{n+1}) = 0$.  We thus suppose that $\d_U f(x_n) = \mu s$ for some $s \in C$ and $\mu \in \fm_U$.  We consider several cases. Suppose $\mu$ is not maximal in $\fm_U$; i.e., there exists some $\mu' \in \fm_U$ with $\mu' \gebang \mu \gebang 1$. Then we define a short map 
\[
f' \co C'(b_1, \dots, b_n, 1/\mu') \leadsto C
\]
by setting $f'(x_i) = f(x_i)$ for $0 \leq i \leq n$ and $f'(x_{n+1}) = (\mu/\mu') s$. It is straightforward to check that this is in fact already a chain map from a semistandard complex using the fact that $\mu/\mu' \in \fm_U$; if desired, we may extend the domain of $f'$ to the standard complex $C(b_1, \dots, b_n, 1/\mu', 1/\nu)$ for any $\nu \in \fm_V$ by setting $f'(x_{n+2}) = 0$. Then $f'$ then provides the chain map in question.  

We thus suppose that $\mu$ is maximal in $\fm_U$. In this case, we extend the domain of $f$ to $C'(b_1, \dots, b_n, 1/\mu)$ by defining $f'(x_{n+1}) = s$. We now have the following cases:
\begin{enumerate}
\item Suppose $\d_V s = 0$. Then $f'$ is in fact already a chain map from a semistandard complex. If desired, we can further extend the domain of $f'$ to the standard complex $C(b_1, \dots, b_n, 1/\mu, 1/\nu)$ for any $\nu \in \fm_V$ by setting $f'(x_{n+2}) = 0$. Then $f'$ then provides the chain map in question.  
\item Suppose $\d_V s = \nu t$ for some $t \in C$ and $\nu$ which is not maximal in $\fm_V$. Then we proceed as in the beginning of the proof, replacing the role of $\fm_U$ with $\fm_V$. This produces the desired chain map.
\item  Suppose $\d_V s= \nu t$ for some $t \in C$ and $\nu$ which is maximal in $\fm_V$. Then we extend the domain of $f'$ to obtain a short map
\[
f''\colon C(b_1,\ldots, b_n, 1/\mu ,1/\nu)\leadsto C
\]
by defining $f''(x_{n+2}) = t$.
\end{enumerate}
Note that in the final case, $f''$ is not in general a chain map. We thus proceed as in the beginning of the proof again, except replacing $f$ with $f''$. This either produces the desired chain map, or once again leads to the final possibility above. In the latter case, we iterate this procedure. If this continues indefinitely, we obtain a sequence of short maps
\[
f_i \co C_i=C(b_1,\ldots,b_n,1/\mu,1/\nu, \ldots, 1/\mu,1/\nu) \leadsto C.
\]
Here, $\mu$ is the maximal element of $\fm_U$, and similarly for $\nu$; note that up to multiplication by $\K^\times$, $\mu$ and $\nu$ are unique (if they exist). The pair $(1/\mu,1/\nu)$ appears $i$ times at the end of the parameter sequence for $C_i$.

Consider the sequence of elements $t_i = f_i(x_{n+2i}) \in C$. This is formed by taking the image under $f_i$ of the final generator of each $C_i$. We claim that there must exist a nontrival $\K$-linear relation among the $t_i$. To see this, fix an $\stair$-basis $\{e_k\}$ of $C$. Now, the only condition on $t_i$ is that $\nu t_i = \partial_V s_i$, where $s_i = f_i(x_{n+2i-1})$. Hence without loss of generality, we may assume that at each stage we have chosen $t_i$ to lie in the $\R_V$-span of $\{e_k\}$, as any coefficients in $\fm_U$ may be discarded without changing this condition. Moreover, $t_i$ cannot be in the image of $\fm_V$, as this would imply the second alternative in the casework above. Thus, each $t_i$ is a linear combination of the $\{e_k\}$, with at least one nonzero coefficient being drawn from $\K$; this means that each $t_i$ must have the same grading as some $e_k$. By Lemma~\ref{lem:finitedimensional}, $C$ is a finite-dimensional $\K$-vector space in each grading. Since the set of $e_k$ is finite, all the $t_i$ lie in a single finite-dimensional $\K$-vector space, which gives the claim.

Now select any subset of indices $\{i_1, i_2, \ldots, i_r\}$, $i_1<i_2< \dots < i_r$, for which the $t_i$ have a $\K$-linear relation $c_{i_1}t_{i_1} + \cdots + c_{i_r}t_{i_r} = 0$ with $c_{i_1}$ nonzero. Note that
\[
r(C_1) \le r(C_2) \le r(C_3) \le \cdots
\]
due to the fact that $\mu$ and $\nu$ are maximal. Then we may repeatedly apply the merge lemma to obtain a short map $g \co C_{i_1} \leadsto C$ whose final generator maps to
\[
g(x_{n + 2i_1}) = \sum_{i=1}^{r} c_{i_j} t_{i_j} = 0.
\]
This is clearly a genuine chain map. If desired, we may truncate $C_{i_1}$ by deleting the last parameter; the restriction of $g$ is a chain map from a semistandard complex. Finally, note that if $f$ is local, then each $f_i$ is local. By the merge lemma we then have that $g$ is local. This completes the proof.
\end{proof}
\noindent

\section{Numerical invariants $a_i$}\label{sec:ai}

In this section, we define a sequence of invariants $a_i(C)$, lying in the valuation groups associated to our grid rings, for any knotlike complex $C$. These are analogous (up to a sign) to the invariants defined in \cite[Section 3]{Hominfiniterank}; the construction here is almost identical to \cite[Section 5]{DHSTmoreconcord}. In the following section, we will observe that the $a_i$ are equivalent to describing a standard complex representative of $C$.  However, the proofs in \cite[Section~5]{DHSTmoreconcord}, as well as their analogs in \cite{DHSThomcobord}, do not directly translate to this setting as they utilize specific grading properties of the rings used therein.

\begin{definition}\label{def:aiinvariants}
Let $C$ be a knotlike complex. Define
\[ a_1(C) = \max \{ b_1\in \Gamma(\R_U) \mid C(b_1, \dots, b_n) \leq C \}. \]
Here, the maximum is taken using the total order on $\Gamma(\R_U)$. We define $a_k(C)$ for $k \geq 2$ inductively, as follows. Suppose that we have already defined $a_i = a_i(C)$ for $1 \leq i \leq k$. If $a_k(C) = 1$, define $a_{k+1}(C) = 1$. Otherwise, define
\[ a_{k+1}(C) = \max \{ b_{k+1} \in \Gamma(\R_*) \mid C(a_1, \dots, a_k, b_{k+1}, \dots b_n) \leq C \}, \]
where $*$ is either of $U$ or $V$, according to the parity of $k+1$.  
\end{definition}

We stress that the complexes appearing in Definition~\ref{def:aiinvariants} are required to be standard (rather than semistandard), and that we take the convention of appending trailing copies of $1$. This means that if $k$ is even, then $a_{k+1}(C) = 1$ if $C(a_1, \ldots, a_k) \leq C$ and no larger standard complex of the form $C(a_1, \dots, a_k, b_{k+1}, \dots b_n)$ is less than or equal to $C$. In such a situation, note that all further invariants of $C$ are also equal to $1$. For $k$ odd, we cannot have $a_{k+1}(C) = 1$ unless $a_k(C) = 1$.

In general, it is not clear that $a_i(C)$ is well-defined, since \textit{a priori} we may have to consider (infinite) subsets of $\Gamma(\R_U)$ or $\Gamma(\R_V)$ that may not achieve a maximum. Our goal for this subsection will be to show that this does not occur. We begin with the following lemma:

\begin{lemma}\label{lem:notinmumv}
Let $C$ be any knotlike complex. Let $C(b_1, \ldots, b_n)$ be any standard or semistandard complex which admits a short local map to $C$:
\[
f \co C(b_1, \ldots, b_n) \leadsto C. 
\]
Suppose $f(x_i) \in (\fm_U, \fm_V)$ for some $i$. Then there exists a complex of the form $C(b_1, \ldots, b_{j-1}, c_j)$, where $1 \leq j \leq i$ and $c_j \gebang b_j$, which also admits a short local map to $C$
\[
g \co C(b_1, \ldots, b_{j-1}, c_j) \leadsto C.
\]
Here, the complex $C(b_1, \ldots, b_{j-1}, c_j)$ is standard or semistandard depending on the parity of $j$.\footnote{We cannot in general specify which of these alternatives hold; however, due to the extension lemma, this will not really be important for the usage of Lemma~\ref{lem:notinmumv} in this paper.}
\end{lemma}
\begin{proof}
Assume there is some $i$ for which $f(x_i) \in (\fm_U, \fm_V)$. Let $j$ be the minimal index with this property. Since $f$ is local, it is easily checked that $f(x_0)$ cannot lie in $(\fm_U, \fm_V)$, so $j > 0$. We assume $j$ is even; the proof for $j$ odd is analogous. There are two cases. First suppose $b_j \lebang 1$. Let 
\[
s = f(x_{j-1}) \quad \text{and} \quad t = f(x_j).
\]
These satisfy the relations $\partial_V s = |b_j| t$ and $\partial_V t = 0$. Fix any basis $\cB$ for $C$. By the hypotheses of the lemma, we may write $t$ as $t_U$ + $t_V$, where $t_U$ is in the $\fm_U$-span of $\cB$ and $t_V$ is in the $\fm_V$-span of $\cB$. Let $\gamma \in \fm_V$ be the greatest common divisor of the coefficients appearing in $t_V$. Set $t' = t_V/\gamma$, with the understanding that $t' = 0$ if $t_V = 0$. Then we may define the desired map
\[
g \co C(b_1, \dots, b_{j-1}, b_j/\gamma) \leadsto C
\]
by altering $f$ on $x_j$ such that $g(x_j) = t'$.  The fact that $\partial_V s = |b_j| t$ implies
\[
\partial_V s = |b_j| (t_U + t_V) = |b_j| t_V = (\gamma |b_j|) t'
\]
so that $g$ is a short local map. Note that since $b_j \lebang 1$ and $\gamma \in \fm_V$, we have $\smash{b_j \lebang b_j/\gamma}$.

Now suppose $b_j \gebang 1$. Define $s$ and $t$ as before; now $\partial_V s = 0$ and $\partial_V t = b_j s$. We again write $t = t_U + t_V$ and let $\gamma$ be the greatest common divisor of the coefficients appearing in $t_V$. Let $\gamma'$ further be the greatest common divisor of $b_j$ and $\gamma$. Set $t' = t_V/\gamma'$, with the understanding that $t' = 0$ if $t_V = 0$. Then we may define
\[
g \co C(b_1, \dots, b_{j-1}, b_j/\gamma') \leadsto C
\]
by altering $f$ on $x_j$ such that $g(x_j) = t'$.  Suppose $b_j/\gamma' \in \fm_V$. Then we may divide both sides of
\[
\partial_V t_V = \partial_V t = b_js
\]
by $\gamma'$ to obtain $\partial_V t' = (b_j/\gamma') s$. If instead $b_j/\gamma'$ is an element of $\smash{\K^\times}$, we require a slightly different argument. In this case, let $s = s_U + s_V$, where $s_U$ is in the $\fm_U$-span of $\cB$ and $s_V$ is in the $\R_V$-span of $\cB$; note that we make all $\K$-coefficient terms appear in $s_V$. Then the fact that $\partial_V t = b_j s$ implies $\partial_V t' = (b_j/\gamma')s_V$. However, because $C$ is reduced, this means that $s_V$ is actually in the image of $\fm_V$. Hence $s \in (\fm_U, \fm_V)$, contradicting the minimality of $j$. This completes the proof.
\end{proof}

The import of Lemma~\ref{lem:notinmumv} is the following. In the definition of $a_{k+1}(C)$, when we consider the collection of standard complexes used to define
\[
a_{k+1}(C) = \max \{ b_{k+1} \in \Gamma(\R_*) \mid C(a_1, \dots, a_k, b_{k+1}, \dots b_n) \leq C \},
\]
we may as well require that the inequality $C(a_1, \dots, a_k, b_{k+1}, \dots b_n) \leq C$ be realized by a local map $f$ such that $f(x_i) \notin (\fm_U, \fm_V)$ for all $i \leq k+1$. Indeed, suppose $f$ is a local map which does not satisfy this condition. Truncate $f$ to $C(a_1, \dots, a_k, b_{k+1})$, apply Lemma~\ref{lem:notinmumv}, and then apply the extension lemma to the result to obtain a local map from another standard complex. This either gives a contradiction to the maximality of $a_1, \ldots, a_k$, or shows that the parameter $b_{k+1}$ in question can be replaced by a larger parameter coming from this new standard complex. 

In order to leverage this, we introduce the following notion:

\begin{definition}\label{def:extant}
Let $C$ be a knotlike complex. We say that a homogenous element $\nu \in \fm_V$ is a \textit{$\partial_V$-extant coefficient for $C$} if there exist homogenous $s$ and $t$ in $C$ such that:
\begin{enumerate}
\item $s$ and $t$ do not lie in $(\fm_U, \fm_V)$; and,
\item $\partial_V s = \nu t$.
\end{enumerate}
We say that a homogenous element $\mu \in \fm_U$ is a \textit{$\partial_U$-extant coefficient for $C$} if there exist similar $s$ and $t$ with $\partial_U s = \mu t$.
\end{definition}

\begin{lemma}\label{lem:extant}
Let $C$ be a knotlike complex. Up to multiplication by $\K^\times$, there are only finitely many $\partial_U$-extant and $\partial_V$-extant coefficients for $C$.
\end{lemma}
\begin{proof}
Assume $\nu$ is $\partial_V$-extant; the $\partial_U$-extant case is similar. Let $\partial_Vs = \nu t$ for some homogenous $s$ and $t$ in $C$ not lying in $(\fm_U, \fm_V)$. First observe that there are only finitely many possibilities for $\gr(s)$; this follows from the fact that $C/(\fm_U, \fm_V)$ is supported in a finite set of gradings. We thus fix a particular bigrading and consider the $\partial_V$-extant coefficients arising from $s$ within this bigrading.

Fix a $V$-paired basis $\{x, y_i, z_i\}_{i=1}^n$ for $C$. For each $i$, consider whether there exists a nonzero $c \in \R_V$ such that $\gr(cy_i) = \gr(s)$. If such a $c$ exists, define $c_i$ to be equal to $c$; this is unique up to multiplication by $\K^\times$. Otherwise set $c_i = 0$. We stress that $c_i$ does not depend on $s$, but only on $\gr(s)$. For any $s$ in our fixed bigrading, write
\[
s = s_U + bx + \sum_{i=1}^n k_ic_i y_i + \sum_{i=1}^n d_i z_i
\]
where $s_U$ is in the image of $\fm_U$, the coefficients $b$ and $d_i$ are in $\R_V$, and the $k_i$ are in $\K$. This is just the usual expansion of $s$ in terms of our paired basis, but we have taken care to indicate that the $\R_V$-coefficients of the $y_i$ only have freedom up to multiplication by $\K$. This is of course true for all of the other coefficients as well, but these will not be important for this proof. We thus obtain
\[
\partial_V s = \sum_{i=1}^n (\nu_ik_ic_i) z_i.
\]
We know that $\partial_V s = \nu t$ and that $t$ does not lie in $(\fm_U, \fm_V)$. Attempting to solve for $t$ shows that $\nu$ must in fact be the greatest common divisor of the set $\{\nu_i k_ic_i\}_{i=1}^n$, and that $t$ must be of the form 
\[
t = t_U + \sum_{i=1}^n ((\nu_ik_ic_i)/\nu) z_i
\]
for some $t_U$ in the image of $\fm_U$. However, up to multiplication by $\K^\times$, the greatest common divisor of $\{\nu_i k_ic_i\}_{i=1}^n$ only depends on the $k_i$ inasmuch as each $k_i$ is zero or nonzero. Indeed, up to multiplication by $\K^\times$, the possible greatest common divisors are simply the nonzero elements of $\{\nu_ic_i\}_{i = 1}^n$. Hence there are only finitely many $\partial_V$-extant coefficients associated to each possible bigrading of $s$, and thus only finitely many $\partial_V$-extant coefficients overall.
\end{proof}


We thus finally obtain:

\begin{lemma}\label{lem:aiwelldefined}
Without loss of generality, the definition of the $a_i(C)$ may be altered so that the maximum is taken over a finite set. Hence the $a_i(C)$ are well-defined.
\end{lemma}
\begin{proof}
As discussed previously, by Lemma~\ref{lem:notinmumv} we may restrict the definition of $a_{k+1}(C)$ to only involve complexes $C(a_1, \dots, a_k, b_{k+1}, \dots b_n)$ which admit a local map 
\[
f \co C(a_1, \dots, a_k, b_{k+1}, \dots b_n) \rightarrow C
\]
with $f(x_i) \notin (\fm_U, \fm_V)$ for all $i \leq k+1$. Setting $s = f(x_{k})$ and $t = f(x_{k+1})$, it is clear that for such complexes, $|b_{k+1}|$ is either a $\partial_U$-extant or a $\partial_V$-extant coefficient for $C$, depending on the parity of $k + 1$. By Lemma~\ref{lem:extant}, up to multiplication by $\K^\times$ there are only finitely many such coefficients. Hence $a_{k+1}(C)$ can be defined by taking the maximum over a finite set. With only minor notational changes, the same argument shows $a_1(C)$ is well-defined.
\end{proof}

\section{Characterization of knotlike complexes up to local equivalence}\label{sec:char}

We now sketch the proof that every knotlike complex is locally equivalent to a standard complex. In fact, it turns out that the numerical invariants $a_i(C)$ from Section~\ref{sec:ai} are none other than the desired standard complex parameters. Many of the claims in this section have proofs which are almost exactly the same as those in \cite[Section 6]{DHSTmoreconcord}. In such cases, we simply indicate the logical dependence of the proofs and refer the reader to the analogous statements in \cite[Section 6]{DHSTmoreconcord}.

\begin{theorem}\label{thm:char}
Every knotlike complex is locally equivalent to a unique standard complex.
\end{theorem}
\begin{proof}
Omitted; see \cite[Theorem 6.1]{DHSTmoreconcord}. Use Lemma~\ref{lem:aizero} below.
\end{proof}

\begin{lemma}\label{lem:aizero}
Let $C$ be a knotlike complex. Then $a_n(C) = 1$ for all $n$ sufficiently large.
\end{lemma}
\begin{proof}
We proceed by contradiction. Suppose that $a_n(C) \neq 1$ for all $n$. For each $n$, there exists a local map $f_n$ from some standard complex $C(a_1, \ldots, a_n, b_{n+1}, \ldots, b_m)$ to $C$. Truncating this gives a local short map
\[ 
f_n \co C(a_1, \dots, a_n) \leadsto C.
\]
By the discussion of the previous section, we may assume that the original map $f_n$ (and thus its truncation) has $f_n(x_i) \notin (\fm_U, \fm_V)$ for all $i \leq n$. In particular, consider the nonzero sequence formed by $f_n(x_n) \bmod{(\fm_U, \fm_V)}$. As in the proof of the extension lemma, since $C/(\fm_U, \fm_V)$ is a finite-dimensional $\K$-vector space, there exists a nontrivial linear relation among some finite subset
\[
k_1 f_{i_1}(x_{i_1}) + \cdots + k_r f_{i_r}(x_{i_r}) = 0 \bmod{(\fm_U,\fm_V)}
\]
with each $k_i \in \K^\times$. Note that implicitly, all of the above $f_{i_j}(x_{i_j})$ lie in the same bigrading. Consider the index $i_j$ from the above set for which the reversed sequence $r(C(a_1, \ldots, a_{i_j}))$ is minimized in the lexicographic order; let this be $i_1$. As in the proof of the extension lemma, by iterated use of Lemma \ref{lem:blebangc} we obtain a short local map 
\[
f\colon C(a_1,\ldots,a_{i_1}) \leadsto C
\]
such that 
\[
f(x_{i_1}) = \sum_{j = 1}^r k_{i_j} f_{i_j}(x_{i_j}) = 0 \bmod{(\fm_U,\fm_V)}.  
\]
Applying Lemma~\ref{lem:notinmumv} and extending to a standard complex contradicts the maximality of the $a_n(C)$.
\end{proof}

We also have a slightly stronger version of Theorem~\ref{thm:char}:

\begin{corollary}\label{cor:splitting}
Let $C$ be a knotlike complex, and assume $C$ is locally equivalent to $C(a_1, \dots, a_n)$. Then $C$ is homotopy equivalent to $C(a_1, \dots, a_n) \oplus A$, for some $\stair$-complex $A$ such that $H_*(\R_U^{-1} A) = H_*(\R_V^{-1} A) = 0$.
\end{corollary}
\begin{proof}
Omitted; see \cite[Corollary 6.2]{DHSTmoreconcord}. Use Lemma~\ref{lem:selflocalmapisomorphism} below.
\end{proof}
\noindent
For a discussion of this corollary, see \cite[Section 6]{DHSTmoreconcord}. The statements necessary for the proof of Corollary~\ref{cor:splitting} are recorded below:

\begin{lemma}\label{lem:cutsequence}
Let $f$ be a local map from a standard complex to itself
\[ f \co C(b_1, \dots, b_n) \to C(b_1, \dots, b_n) \]
such that $f(x_i)$ is supported by $x_j$ for some $i \neq j$. If $i\equiv j\bmod{2}$, then the following hold:
\begin{enumerate}
\item $(b_{i+1}, \dots, b_n) \lebang (b_{j+1}, \dots, b_n)$ and
\item $(b_{i}^{-1},\ldots, b_1^{-1})\lebang (b_j^{-1},\ldots, b_1^{-1})$.
\end{enumerate}
If $i\not\equiv j \bmod{2}$, then the following hold:
\begin{enumerate}
\item $(b_{i+1}, \dots, b_n) \lebang (b_{j}^{-1}, \dots, b_1^{-1})$ and 
\item $(b_i^{-1},\ldots,b_1^{-1})\lebang (b_{j+1},\ldots,b_n)$.
\end{enumerate}
In both cases, we mean that $(b_{k+1}, \dots, b_n) = (1)$ if $k = n$.
\end{lemma}

\begin{proof}
Omitted; see \cite[Lemma 6.4]{DHSTmoreconcord}. 
\end{proof}

\begin{lemma}\label{lem:selflocalmapinjective}
Any local map from a standard complex to itself is injective.
\end{lemma}

\begin{proof}
Let $C = C(b_1, \ldots, b_n)$ be a standard complex with preferred basis $\{x_i\}_{i=0}^n$. Let $f \co C \rightarrow C$ be a local self-map. Suppose there exists some $\stair$-linear combination $\sum_i c_ix_i$ such that $f(\sum_i c_ix_i)=0$. Since $f$ is graded, we may assume that $\sum_i c_ix_i$ is homogenous, and that each $c_i$ is homogeneous. Impose a partial order on the homogeneous elements of $\stair$ by declaring $x \gebang y$ if(f) $x$ divides $y$. Among the nonzero coefficients $c_i$, choose any maximal element $c_{i_0}$ with respect to this partial order. Let $I = \{ j \mid c_j = c_{i_0} \}$. Choose an arbitrary element $j_0 \in I$. Consider the suffix sequence
\[
K_{j_0} = (b_{j_0 + 1}, \ldots, b_n).
\]
For each other $j_i \in I$, define
\[
K_{j_i} = 
\begin{cases}
	(b_{j_i + 1}, \ldots, b_n) &\text{ if } j_i \equiv j_0\bmod{2} \\
	(b_{j_i}^{-1}, \ldots, b_1^{-1}) &\text{ if } j_i \not\equiv j_0\bmod{2}.
\end{cases}
\]
Note that in the first case, the length of $K_{j_i}$ has the same parity as the length of $K_{j_0}$, while in the second case the parities are opposite. It follows from this that $K_{j_1} \neq K_{j_2}$ if $j_1 \neq j_2$, since the lengths of the two sequences are not the same. We may thus re-index the elements of $\{ j_1, \dots, j_m \}$ such that
\[ K_{j_1}\lebang K_{j_2}\lebang\ldots\lebang K_{j_m}.\]
Consider $f(x_{j_1})$. By Lemma \ref{lem:localsupport}, $f(x_{j_1})$ is supported by $x_{j_1}$. By Lemma \ref{lem:cutsequence}, $f(x_{j_i})$ for $2 \leq i \leq m$ cannot be supported by $x_{j_1}$. But then the maximality of $c_{i_0}$ implies that there is no term in $f(\sum_{i \neq j_1} c_i x_i)$ that can cancel $c_{j_1} f(x_{j_1})$, contradicting the fact that $f(\sum_i c_ix_i)=0$. Hence $f$ must be injective.
\end{proof}

\begin{lemma}\label{lem:selflocalmapisomorphism}
Any local map from a standard complex to itself is an isomorphism.
\end{lemma}
\begin{proof}
Omitted; see \cite[Lemma 6.6]{DHSTmoreconcord}. Use Lemma~\ref{lem:selflocalmapinjective}.
\end{proof}

Finally, we formally re-state the claim that a right local equivalence between two totally knotlike complexes is left local.

\begin{lemma}\label{lem:loceqsymmetric}
Let $C_1$ and $C_2$ be knotlike complexes which are locally equivalent via $f$ and $g$. Then $f$ and $g$ induce isomorphisms on $H_*(C_i \otimes \R_U)/\R_U\text{-torsion}$.
\end{lemma}
\begin{proof}
Omitted; see \cite[Lemma 6.9]{DHSTmoreconcord}. Use Theorem~\ref{thm:char} and Lemma~\ref{lem:selflocalmapisomorphism}.
\end{proof}
Lemma~\ref{lem:loceqsymmetric} easily implies that a knotlike complex which is locally symmetric in the sense of Definition~\ref{def:locallysymmetric} has a symmetric standard complex representative:

\begin{lemma}
Let $\xi\co \stair \rightarrow \stair$ be an involution on $\stair$ as in Definition~\ref{def:xi}. Let $C$ be a knotlike complex which is locally symmetric with respect to $\xi$. Then the standard complex representative of $C$ guaranteed by Theorem~\ref{thm:char} is symmetric. 
\end{lemma}
\begin{proof}
By Theorem~\ref{thm:char}, we know that $C$ is locally equivalent to a unique standard complex $C(b_1, \ldots, b_n)$. Thus
\[
C(b_1, \ldots, b_n) \sim C \sim \xi^*(C) \sim \xi^*(C(b_1, \ldots, b_n)).
\]
Here, we are using Lemma~\ref{lem:loceqsymmetric}, together with the fact that if $C_1$ and $C_2$ are right locally equivalent, then $\xi^*(C_1)$ and $\xi^*(C_2)$ are left locally equivalent (and vice-versa). Hence $C(b_1, \ldots, b_n)$ is locally symmetric. By Lemma~\ref{ex:symmetricstandard}, it is then symmetric.
\end{proof}

\section{Homomorphisms}\label{sec:homs}

In this section, we construct the family of linearly independent homomorphisms from $\KL$ to $\Z$ described in Section~\ref{sec:intro}.

\subsection{Some $\Z$-valued homomorphisms}

We begin with the central definition:

\begin{definition}\label{def:homs}
Let $C = C(a_1, \dots, a_n)$ be a standard complex. For $\mu \in \Gamma_+(\R_U)$, define $\varphi_\mu^U(C)$ to be the signed count of odd-index parameters equal to $\mu$:
\[ \varphi_\mu^U(C) = \# \{ a_i \mid a_i = \mu, i \text{ odd} \} -  \# \{ a_i \mid a_i = \mu^{-1}, i \text{ odd} \}. \]
For $C$ an arbitrary knotlike complex, define $\smash{\varphi_\mu^U(C)}$ by passing to the standard complex representative of $C$. For $\smash{\nu \in \Gamma_+(\R_V)}$, we analogously define $\smash{\varphi_\nu^V(C)}$ by taking the signed count of even-index parameters of $C$. In the case that $C$ is a locally symmetric complex, it is clear that $\smash{\varphi_\mu^U}$ (as $\mu$ is allowed to vary) records the same information as $\smash{\varphi_\mu^V}$. In this situation, we simply write $\smash{\varphi_\mu = \varphi_\mu^U}$.
\end{definition}

The main theorem of this section is:

\begin{theorem}\label{thm:homs}
Let $\stair = \stair(\R_U, \R_V)$ be a grid ring which is grading-nontrivial in the sense of Definition~\ref{def:gradingnontrivial}. Then for each $\mu \in \Gamma_+(\R_U)$, the function
\[ \varphi_\mu^U \co \KL \to \Z \]
is a homomorphism. Similarly, for each $\nu \in \Gamma_+(\R_V)$, the function
\[ \varphi_\nu^V \co \KL \to \Z \]
is a homomorphism. 
\end{theorem}

\noindent
We prove Theorem~\ref{thm:homs} by expressing the $\varphi_\mu^U$ and $\varphi_\nu^V$ as linear combinations of other auxiliary homomorphisms. The construction of these will take up the next several subsections.

\subsection{Shift maps and paired bases} We begin by constructing an auxiliary family of endomorphisms of $\KL$, which we call the \textit{shift homomorphisms}. Our first goal will be to establish Theorem~\ref{thm:shift}, which verifies that these are in fact homomorphisms. The proof of Theorem~\ref{thm:shift} will involve a detailed analysis of the tensor product of two standard complexes. This is done in the latter half of the subsection, before proceeding with the rest of the proof in the sequel. 

\begin{definition}\label{def:shiftmap}
A \textit{shift map} is any injective, order-preserving function
\[
m_U \co \Gamma_+(\R_U) \rightarrow \Gamma_+(\R_U).
\]
We extend $m_U$ to a function on $\Gamma_-(\R_U)$ by requiring $m_U$ to commute with inversion; for convenience, we also set $m_U([1]) = [1]$. This defines an injective, order-preserving function on all of $\Gamma(\R_U)$, which we also denote by $m_U$. As usual, we will sometimes be imprecise about whether the domain of $m_U$ is $\Gamma(\R_U)$ or $\loc(\R_U)$. When we have two shift maps $m_U \co \Gamma(\R_U) \rightarrow \Gamma(\R_U)$ and $m_V \co \Gamma(\R_V) \rightarrow \Gamma(\R_V)$, we will often treat them together and refer to
\[
m = m_U \sqcup m_V \co \Gamma(\R_U) \sqcup \Gamma(\R_V) \rightarrow \Gamma(\R_U) \sqcup \Gamma(\R_V)
\]
as a shift map also.
\end{definition}

\begin{definition}\label{def:shift}
Let $m$ be a shift map. Define the associated \textit{shift homomorphism}
\[
\shift_m \co \KL \rightarrow \KL
\]
as follows. For $C=C(a_1, \dots, a_n)$ a standard complex, let $\shift_m(C)$ be the standard complex given by
\[ 
\shift_m(C) =  C(m(a_1), \dots, m(a_n)). 
\]
More generally, we define $\shift_m(C)$ by first passing to the standard complex representative of $C$. It will also be helpful to consider shifting only the parameters in $\R_U$ or only the parameters in $\R_V$. For a standard complex $C$ as above, we thus define
\[ 
\shift_{U,m}(C) = C(a_1', \dots, a_n') 
\]
where 
\[ 
a_i' =
\begin{cases} 
m(a_i) \quad &\text{if } i \text{ odd}, \\
a_i \quad &\text{if } i \text{ even},
\end{cases}
\]
and similarly
\[ 
\shift_{V,m}(C) = C(a_1', \dots, a_n') 
\]
where 
\[ 
a_i' =
\begin{cases} 
m(a_i) \quad &\text{if } i \text{ even}, \\
a_i \quad &\text{if } i \text{ odd},
\end{cases}
\]
extending to all of $\KL$ as before. Clearly, $\shift_m = \shift_{V,m} \circ \shift_{U,m}$.
\end{definition}

Our central claim is the following:
\begin{theorem}\label{thm:shift}
Let $C_1$ and $C_2$ be knotlike complexes. For any shift map $m$, 
\[ 
\shift_{U, m}(C_1 \otimes C_2) \sim \shift_{U, m}(C_1) \otimes \shift_{U, m}(C_2)
\] 
and
\[ 
\shift_{V, m}(C_1 \otimes C_2) \sim \shift_{V, m}(C_1) \otimes \shift_{V, m}(C_2).
\] 
That is, the functions $\shift_{U, m} \co \KL \to \KL$ and $\shift_{V, m} \co \KL \to \KL$ are homomorphisms. Since $\shift_m = \shift_{V,m} \circ \shift_{U,m}$, it follows that $\shift_m$ is a homomorphism also.
\end{theorem}
\noindent
The proof of this will be completed in the next subsection. It will also be helpful for us to observe:

\begin{lemma}\label{lem:shiftdual}
If $C$ is any standard complex, then $\shift_{U, m}(C^\vee) = \shift_{U,m}(C)^\vee$. Similarly, $\shift_{V, m}(C^\vee) = \shift_{V,m}(C)^\vee$
\end{lemma}
\begin{proof}
This follows immediately from Lemma~\ref{lem:dual}, together with the fact that $m$ commutes with inversion.
\end{proof}

We now establish a particularly simple basis for the tensor product of two standard complexes. Observe that if $C$ is a standard complex, then up to relabeling, its preferred basis is already $U$-paired; and, up to a (generally) different relabeling, this same basis is also already $V$-paired. Explicitly:

\begin{definition}
Let $C=C(a_1, \dots, a_n)$ be a standard complex with preferred basis $\{x_i\}_{i=0}^n$. We define a $U$-paired basis \smash{$\{ w, y_i, z_i \}_{i=1}^{n/2}$} for $C$ by setting
\[ w = x_n, \] 
and for each $1 \leq i \leq n/2$, 
\begin{align*}
	(y_i, z_i) &= \begin{cases}
			(x_{2i-2}, x_{2i-1}) \quad \text{ if } a_{2i-1} \lebang 1 \\
			(x_{2i-1}, x_{2i-2}) \quad \text{ if } a_{2i-1} \gebang 1.
		\end{cases}
\end{align*}
Note that in this basis, $\d_U y_i = |a_{2i-1}| z_i$. Define a $V$-paired basis for $C$ similarly by setting $w = x_0$ and using the even-index parameters of $C$.
\end{definition}

Now consider the tensor product of two standard complexes. It is clear that the obvious tensor product basis is not $U$- or $V$-paired. We thus instead define:

\begin{definition}\label{def:Upairedtensor}
Let $C_1 = C(a_1, \dots, a_{n_1})$ and $C_2 = C(b_1, \dots, b_{n_2})$ be standard complexes. Abusing notation slightly, let $\{w, y_i, z_i\}$ denote the $U$-paired bases for both $C_1$ and $C_2$; it will be clear from context which generators lie in $C_1$ and $C_2$. We define a $U$-paired basis for $C_1 \otimes C_2$ as follows. For $1 \leq i \leq n_2/2$, let
\begin{align*}
	\alpha_i &= w \otimes y_i\\
	\beta_i &= (-1)^{\gr(w)} w \otimes z_i,
\end{align*}
and for $1 \leq i \leq n_1/2$, let
\begin{align*}
	\gamma_i &= y_i \otimes w \\
	\delta_i &= z_i \otimes w.
\end{align*}
For $1 \leq i \leq n_1/2$ and $1 \leq j \leq n_2/2$, define
\begin{align*}
	\epsilon_{i,j} &=y_i \otimes y_j \\
	\zeta_{i,j} &= \begin{cases}
				|b_{2j-1}||a_{2i-1}|^{-1} z_i \otimes  y_j +  (-1)^{\gr(y_i)}y_i \otimes z_j   \quad \text{ if } |a_{2i-1}| \geqbang |b_{2j-1}| \\
				z_i \otimes y_j + (-1)^{\gr(y_i)}|a_{2i-1}||b_{2j-1}|^{-1} y_i \otimes z_j  \quad \text{ if } |a_{2i-1}| \lebang |b_{2j-1}| 
			\end{cases}
\end{align*}
and
\begin{align*}
	\eta_{i,j} &= \begin{cases}
				y_i \otimes z_j \quad \text{ if } |a_{2i-1}| \geqbang |b_{2j-1}| \\
				z_i \otimes y_j \quad \text{ if } |a_{2i-1}| \lebang |b_{2j-1}| 
			\end{cases}\\
	\theta_{i,j} &=  \begin{cases}
				z_i \otimes z_j \quad &\text{ if } |a_{2i-1}| \geqbang |b_{2j-1}| \\
				(-1)^{\gr(z_i)} z_i \otimes z_j \quad &\text{ if } |a_{2i-1}| \lebang |b_{2j-1}| 
			\end{cases}\\
\end{align*}
Finally, let
\[
	\omega = w \otimes w. 
\]
Note that the following basis elements are $U$-paired:
\[ \{ \alpha_i, \beta_i \},  \quad \{ \gamma_i, \delta_i \}, \quad \{ \epsilon_{i,j}, \zeta_{i,j} \}, \quad \{ \eta_{i,j}, \theta_{i,j} \}. \]
For notational convenience, we relabel the basis elements
\begin{align*}
	\{ \kappa_\ell \} &= \{ \alpha_i \} \cup \{ \gamma_i \} \cup \{ \epsilon_{i,j} \} \cup \{ \eta_{i,j} \} \\
	\{ \lambda_\ell \} &= \{ \beta_i \} \cup \{ \delta_i \} \cup \{ \zeta_{i,j} \} \cup \{ \theta_{i,j} \}
\end{align*}
so that $\{ \omega, \kappa_\ell, \lambda_\ell \}$ is a $U$-paired basis and $\d_U \kappa_\ell = e_\ell \lambda_{\ell}$ for some $e_{\ell}$. The reader should check that if $\kappa_{\ell}$ is one of $\epsilon_{i,j}$ or $\eta_{i, j}$, then
\[
e_{\ell} = \max(|a_{2i-1}|, |b_{2j-1}|).
\]
If $\kappa_{\ell}$ is an $\alpha_i$, then $e_{\ell} = |b_{2i-1}|$, while if $\kappa_\ell$ is a $\gamma_i$, then $e_{\ell} = |a_{2i-1}|$. We construct a $V$-paired basis for $C_1 \otimes C_2$ similarly.
\end{definition}

The importance of paired bases is that they help us define certain (ungraded) $\stair$-module morphisms between standard complexes and their images under $\shift_{U,m}$ and $\shift_{V,m}$. For a single standard complex this is straightforward:

\begin{definition}\label{def:sum}
Let $C = C(a_1, \dots, a_n)$ and let $\{ w, y_i, z_i \}$ and $\{ w', y'_i, z'_i \}$ be the $U$-paired bases for $C$ and $\shift_{U,m}(C)$ respectively. Define an (ungraded) $\stair$-module morphism
\[ 
s_{U,m} \co C \rightarrow \shift_{U,m}(C) 
\]
by sending
\[ 
s_{U,m}(r) = r' 
\]
for each $r \in \{ w, y_i, z_i \}$, and extending $\stair$-linearly. Then $s_{U,m}$ induces an isomorphism of (ungraded) $\stair$-modules. Moreover, the reader may check that $s_{U,m} \d_V = \d_V s_{U,m}$. On the other hand,
\begin{align*}
&s_{U,m} (\d_U y_i) = s_{U, m} (|a_{2i-1}| z_i) = |a_{2i-1}| z_i' \\
&\d_U (s_{U,m} y_i )= \d_U (y_i') = |a_{2i-1}'| z_i'
\end{align*}
so $s_{U,m}\d_U\neq \d_Us_{U,m}$. We define $s_{V,m}\co C \rightarrow \shift_{V,m}(C)$ analogously.
\end{definition}

Defining a $\stair$-module morphism from $C_1 \otimes C_2$ to $\shift_{U,m}(C_1) \otimes \shift_{U,m}(C_2)$ is slightly more subtle. There are two possibilities for producing such a map. The first is to take the tensor product of the $\stair$-module morphisms in Definition~\ref{def:sum}:
\[
s_{U, m} \otimes s_{U, m} \co  C_1 \otimes C_2 \to \shift_{U,m}(C_1) \otimes \shift_{U,m}(C_2)
\]
The second is to use the $U$-paired bases of Definition~\ref{def:Upairedtensor}. For this, let
\[
\{ \alpha'_i, \beta'_i, \gamma'_i, \delta'_i, \epsilon'_{i,j}, \zeta'_{i,j}, \eta'_{i,j}, \theta'_{i,j}, \omega' \}
\]
be the basis for $\shift_{U,m}(C_1) \otimes \shift_{U,m}(C_2)$ constructed in Definition~\ref{def:Upairedtensor}. Here, we consider the factors $\shift_{U,m}(C_1)$ and $\shift_{U,m}(C_2)$ as standard complexes in their own right, so that $\alpha'_i = w' \otimes y'_i$ (and so on). We re-label this basis $\{ \omega', \kappa'_{\ell}, \lambda'_\ell \}$ as before, so that $\kappa_\ell'$ and $\lambda_\ell'$ are $U$-paired. As above, we have $\d_U \kappa_\ell' = e_\ell' \lambda_{\ell}'$, where 
\[
e_\ell' = \max(|a_{2i-1}'|, |b_{2j-1}'|)
\]
whenever $\kappa_{\ell}'$ is one of $\epsilon_{i,j}'$ or $\eta_{i, j}'$ (and similarly for the other cases). Note that
\[
e_\ell' = m(e_\ell).
\]
We then define:

\begin{definition}\label{def:tensorproductbasis}
Let $C_1$ and $C_2$ be standard complexes and let $m$ be a shift map. Define an $\stair$-module morphism
\[ \sigma_{U,m} \co C_1 \otimes C_2 \to \shift_{U,m}(C_1) \otimes \shift_{U,m}(C_2) \]
by sending
\[ \sigma_{U,m}( \xi ) = \xi ' \]
for $\xi \in \{ \omega, \kappa_\ell, \lambda_\ell \}$, and extending $\stair$-linearly. As in Definition~\ref{def:sum}, $\sigma_{U,m}$ induces an isomorphism of ungraded $\stair$-modules. Furthermore, we claim that $\sigma_{U,m} \d_V = \d_V \sigma_{U,m}$. To see this, observe that
\[
\sigma_{U, m} \equiv s_{U, m} \otimes s_{U, m} \bmod \fm_U.
\]
Indeed, this congruence is clearly an equality for all basis elements not of the form $\zeta_{i,j}$ or $\eta_{i, j}$. For $\eta_{i, j}$, we again have equality using the fact that $\smash{|a_{2i-1}| \leqbang |b_{2j-1}|}$ if and only if $\smash{|a_{2i-1}'| \leqbang |b_{2j-1}'|}$. For basis elements of the form $\zeta_{i,j}$, a straightforward casework check, together with the fact that $m$ is injective and order-preserving, establishes the congruence. Since $s_{U,m}$ commutes with $\d_V$, this shows $\smash{\sigma_{U,m} \d_V = \d_V \sigma_{U,m}}$. Again, however, note that $\sigma_{U,m}$ does not commute with $\d_U$. Define $\sigma_{V,m} \co C_1 \otimes C_2 \to \shift_{V,m}(C_1) \otimes \shift_{V,m}(C_2)$ similarly.
\end{definition}

We stress that the maps $s_{U, m}$ and $\sigma_{U, m}$ are ungraded $\stair$-module maps which do not commute with $\partial$, although they do commute with $\partial_V$. Nevertheless, these will be useful auxiliary tools in the next subsection, as we shall see.

\subsection{Verification of the shift homomorphisms} We now prove Theorem~\ref{thm:shift}. Let $C_1$ and $C_2$ be two standard complexes, and suppose that $C_3$ is a standard complex admitting a local equivalence to $C_1 \otimes C_2$. In order to prove that $\shift_{U,m}$ is a homomorphism, we need to construct a local map from $\shift_{U,m}(C_3)$ into $\shift_{U,m}(C_1) \otimes \shift_{U,m}(C_2)$. This will be done with the help of the $\stair$-module map $\sigma_{U,m}$ of the previous subsection. However, it will actually be more convenient for us to proceed in two stages: we define an ``approximate" chain map from $\shift_{U,m}(C_3)$ into $\shift_{U,m}(C_1) \otimes \shift_{U,m}(C_2)$, and then show that any such map can be upgraded into a genuine local map. This is the content of the definition and lemma below:

\begin{definition}\label{def:almostchainmap}
Let $C(a_1, \dots, a_n)$ be a standard complex with preferred basis $\{x_i\}_{i=0}^n$ and let $C$ be any knotlike complex. An \emph{almost chain map} $f \co C(a_1, \dots, a_n) \to C$ is an ungraded $\stair$-module map such that for $1 \leq i \leq n$ odd:
		\begin{enumerate}
			\item if $a_i \lebang 1$, that is, $\d_U x_{i-1} = |a_i|x_i$, we have 
		\[ \d_U f (x_{i-1}) \equiv |a_i| f(x_{i}) \mod |a_i|\fm_U, \]
			\item if $a_i \gebang 1$, that is, $\d_U x_{i} = |a_i|x_{i-1}$, we have 
			\[ \d_U f (x_{i}) \equiv |a_i| f(x_{i-1}) \mod |a_i|\fm_U. \]
		\end{enumerate}
We impose a similar set of condition for $i$ even, replacing $U$ with $V$ above. Note that an almost chain map is not in general a chain map, and may not even be grading-homogeneous.		
		
\end{definition}

In what follows, let $[x]_{(u, v)}$ denote the homogeneous part of $x$ in bigrading $(u, v)$. 

\begin{lemma}\label{lem:almosttolocal}
Let $f \co C(a_1, \dots, a_n) \to C$ be an almost chain map. Let $(u_i, v_i)$ be the grading of $x_i$ in $C(a_1, \dots, a_n)$. Suppose that $[f(x_0)]_{(u_0, v_0)}$ represents a $V$-tower class in $C$ and $\d_U[f(x_n)]_{(u_n, v_n)} = 0$. Then there exists a genuine local map 
\[ g \co C(a_1, \dots, a_n) \to C \]
such that $g(x_i) \equiv [f(x_i)]_{(u_i, v_i)} \bmod (\fm_U,\fm_V)$ for all $0 \leq i \leq n$.
\end{lemma}

\begin{proof}
For each $0 \leq i \leq n$, consider the ansatz:
\[
g(x_i) = [f(x_i)]_{(u_i, v_i)} + p_i + q_i
\]
where $p_i$ and $q_i$ are undetermined elements with grading $(u_i, v_i)$ in the images of $\fm_U$ and $\fm_V$, respectively. In order to determine $p_i$ and $q_i$, we substitute our ansatz into the chain map condition for $g$. For simplicity, assume $i$ is odd and $a_i \lebang 1$. Then $\d_U x_{i-1} = |a_i| x_i$ and $\d_U x_i = 0$. Using Definition~\ref{def:almostchainmap}, write
\[
\d_U f (x_{i-1}) = |a_i| f(x_{i}) + |a_i|\mu \eta_i
\]
for some (possibly non-homogeneous) element $\eta_i \in C$ and $\mu \in \fm_U$.  Note that since $\d_U^2 = 0$, we have $\d_U f(x_i) + \mu \d_U \eta_i = 0$. We now compute: 
\begin{align*}
g(\d_Ux_{i-1}) &= |a_i| g(x_i) \\
&= |a_i| ([f(x_i)]_{(u_i, v_i)}  + p_i + q_i) \\
&= |a_i| [f(x_i)]_{(u_i, v_i)} + |a_i| p_i \\
\d_U g(x_{i-1}) &= \d_U ([f(x_{i-1})]_{(u_{i-1}, v_{i-1})} + p_{i-1} + q_{i-1})\\
&= |a_i|[f(x_i)]_{(u_i, v_i)} + |a_i|\mu [\eta_i]_{(u_i,v_i)-\gr(\mu)} + \d_Up_{i-1}.
\end{align*}
We likewise compute
\begin{align*}
g(\d_Ux_i) &= g(0) = 0 \\
\d_U g(x_i) &= \d_U([f(x_i)]_{(u_i, v_i)} + p_i + q_i) \\
&= \d_U [f(x_i)]_{(u_i, v_i)} + \d_U p_i.
\end{align*}
Examining the desired equality $g(\d_Ux_{i-1}) = \d_U g(x_{i-1})$, we see that it suffices to set $p_{i-1} = 0$ and $p_i = \mu[\eta_i]_{(u_i,v_i)-\gr(\mu)}$. Note that indeed $p_{i-1}$ and $p_i$ are in the image of $\fm_U$. The equality $g(\d_Ux_i) = \d_U g(x_i)$ then follows from the fact that $\d_U f(x_i) + \mu\d_U \eta_i = 0$. Doing this for each $i$ odd shows that we may choose the $p_i$ such that $g$ commutes with $\d_U$; the condition on $f(x_n)$ is used to show we may set $p_n = 0$. The analogous argument when $i$ is even allows us to choose $q_i$ such that $g$ commutes with $\d_V$. (Here we use the fact that the $p_i$ are in $\fm_U$, so that they do not enter into the equations used to determine the $q_i$.) The condition on $f(x_0)$ corresponds to checking locality of $g$.
\end{proof}

In order to construct a local map from $\shift_{U,m}(C_3)$ into $\shift_{U,m}(C_1) \otimes \shift_{U,m}(C_2)$, it thus suffices to instead construct an almost chain map satisfying the conditions of Lemma~\ref{lem:almosttolocal}. We do this below:


\begin{definition}\label{def:fum}
Let $f$ be a local map from the standard complex $C_3$ to $C_1 \otimes C_2$, where $C_1$ and $C_2$ are each standard complexes.   Let $m$ be a shift map. Define
\[ f_{U,m} \co \shift_{U,m}(C_3) \to \shift_{U,m}(C_1) \otimes \shift_{U,m}(C_2) \]
as follows. Let $\{ w, y_i, z_i \}$ and $\{ w', y'_i, z'_i \}$ be the $U$-paired bases for $C_3$ and $\shift_{U,m}(C_3)$, respectively. For $r' \in \{ w', z'_i\}$, let
\[ f_{U,m} (r') = \sigma_{U,m} f(r), \]
where $r \in \{w, z_i\}$ is the corresponding basis element in $C_3$. To define $f_{U,m}(y'_i)$, first write $f(y_i)$ in terms of the $U$-paired basis $\{\omega, \kappa_j, \lambda_j \}$ for $C_1 \otimes C_2$. We separate the coefficients for the $\kappa_j$ into summands coming from $\K$, $\fm_U$, and $\fm_V$, so that 
\begin{equation}\label{eq:fyi}
f(y_i) = \sum_{j \in J_1}k_j \kappa_j + \sum_{j \in J_2} p_j \kappa_j + \sum_{j \in J_3} q_j \kappa_j + \sum_j  P_j \lambda_j + Q \omega
\end{equation}
For some index sets $J_1, J_2$, and $J_3$. Here, $k_j \in \K$, $p_j\in\fm_U$, and $q_j \in \fm_V$, while $P_j, Q \in \stair$. We define:
\begin{align}\label{eq:fyiprime} 
	f_{U,m} (y'_i) &=  \sigma_{U,m} \Big( \sum_{j \in J_1} k_j \kappa_j + \sum_{j \in J_2} m(p_je_j)m(e_j)^{-1} \kappa_j + \sum_{j \in J_3} q_j \kappa_j + \\
	&\qquad \qquad \sum_j P_j \lambda_j + Q\omega \Big). \nonumber
\end{align}
Here, $e_j$ is as in Definition~\ref{def:tensorproductbasis}; that is, $\partial_U \kappa_j = e_j \lambda_j$. Note that $e_j \gebang p_je_j$, so the fact that $m$ is order-preserving shows that $m(e_j)$ divides $m(p_je_j)$ in $\fm_U$.
\end{definition}

For convenience, we formally record the following identities:
\begin{lemma} 
Let $f$ and $C_3 = C(c_1, \dots, c_{n})$ be as above. Then
\begin{equation}\label{eq:sumfz}
	|c_{2i-1}| \sigma_{U,m} f( z_i) = \sum_{j \in J_1}k_je_j  \sigma_{U,m} (\lambda_j) +  \sum_{j \in J_2} p_je_j  \sigma_{U,m} (\lambda_j)
\end{equation}
for each $z_i$, and
\begin{equation}\label{eq:sigmaf}
	\sigma_{U,m} f(r) \equiv f_{U,m} s_{U,m}(r) \mod \fm_U.
\end{equation}
for all $r \in \{ w, y_i, z_i \}$. 
\end{lemma}
\begin{proof}
To establish the first claim, apply $\sigma_{U,m}$ to both sides of
\[ 
f(|c_{2i-1}| z_i) = f \d_U(y_i) = \d_U f(y_i) = \sum_{j \in J_1}k_j e_j\lambda_j +  \sum_{j \in J_2} p_je_j \lambda_j. 
\]
For the second claim, note that the right-hand side is just $f_{U,m}(r')$. Thus, definition~\ref{def:fum} immediately gives equality if $r = w$ or $z_i$. If $r = y_i$, then we obtain the left-hand side of \eqref{eq:sigmaf} by applying $\sigma_{U,m}$ to \eqref{eq:fyi}. Comparing this with \eqref{eq:fyiprime} gives the congruence, keeping in mind the fact that $p_j$ and $m(p_je_j)m(e_j)^{-1}$ are both elements of $\fm_U$. 
\end{proof}

We now turn to the central technical lemma of this section:

\begin{lemma}\label{lem:fUalmost}
Let $C_1$, $C_2$, and $C_3$ be standard complexes, let 
\[
f \co C_3 \to C_1 \otimes C_2
\]
be a local map, and let $m$ be a shift map. Then
\[ 
f_{U,m} \co \shift_{U,m}(C_3) \to \shift_{U,m}(C_1) \otimes \shift_{U,m}(C_2) 
\]
is an almost chain map.
\end{lemma}

\begin{proof}
Let $\{ w, y_i, z_i \}$ and $\{ w', y'_i, z'_i \}$ be the $U$-paired bases for $C_3 = C(c_1, \dots, c_{n})$ and $\shift_{U,m}(C_3)$, respectively. We show that
\begin{equation}\label{eq:Ualmostchain}
	\d_U f_{U,m} (y'_i) \equiv m(|c_{2i-1}|)  \sigma_{U,m} f(z_i) \mod m(|c_{2i-1}|)\fm_U, 
\end{equation}
and that 
\begin{equation}\label{eq:Valmostchain}
	\d_V f_{U,m}(r') = f_{U,m} \d_V (r')
\end{equation}
for all $r' \in \{ w', y'_i, z'_i \}$. To prove \eqref{eq:Ualmostchain}, we apply $\partial_U$ to \eqref{eq:fyiprime}. This gives:
\begin{align}
	\d_U f_{U,m} (y'_i) &=  \d_U  \Big(   \sum_{j \in J_1} k_j \sigma_{U,m}(\kappa_j) + \sum_{j \in J_2}m(p_je_j)m(e_j)^{-1} \sigma_{U,m}(\kappa_j) \Big) \nonumber \\
			&= \sum_{j \in J_1}  k_jm(e_j) \lambda'_j+ \sum_{j \in J_2}m(p_je_j) \lambda'_j.\label{eq:target}
\end{align}
On the other hand, according to \eqref{eq:sumfz}, we have
\[
|c_{2i-1}| \sigma_{U,m} f( z_i) = \sum_{j \in J_1}k_je_j \lambda'_j +  \sum_{j \in J_2} p_je_j \lambda'_j.
\]
Note that $|c_{2i-1}|$ divides each coefficient $e_j$ and $p_je_j$ appearing in this sum, since the $\lambda'_j$ generators are not in the image of $\fm_U$. We consider two possibilities:
\begin{enumerate}
\item Suppose that $m(|c_{2i-1}|) \leqbang |c_{2i-1}|$. Write $m(|c_{2i-1}|) = \gamma |c_{2i-1}|$ for some $\gamma \in \R_U$. Multiplying both sides of \eqref{eq:sumfz} by $\gamma$, we obtain
\begin{equation}\label{eq:case1}
m(|c_{2i-1}|) \sigma_{U,m} f( z_i) = \sum_{j \in J_1}k_j(\gamma e_j) \lambda'_j +  \sum_{j \in J_2} (p_j \gamma e_j) \lambda'_j.
\end{equation}
We show this is congruent to \eqref{eq:target} modulo $m(|c_{2i-1}|)\fm_U$. Fix $j \in J_1$. We know $e_j \leqbang |c_{2i-1}|$. If $e_j = |c_{2i-1}|$, then tautologically $m(e_j) = \gamma e_j$ and the two terms in \eqref{eq:target} and \eqref{eq:case1} corresponding to $j$ are exactly the same. If $e_j \lebang |c_{2i-1}|$, then using the fact that $m$ is order-preserving, we see that $m(e_j)$ is in $m(|c_{2i-1}|) \fm_U$. Since $\gamma e_j \lebang \gamma |c_{2i-1}| = m(|c_{2i-1}|)$, this holds for $\gamma e_j$ also, so the two terms in \eqref{eq:target} and \eqref{eq:case1} corresponding to $j$ are both congruent to zero modulo $m(|c_{2i-1}|) \fm_U$. A similar argument holds for the terms coming from $j \in J_2$.
\item Suppose that $m(|c_{2i-1}|) \gebang |c_{2i-1}|$. Write $\gamma m(|c_{2i-1}|) = |c_{2i-1}|$ for some $\gamma \in \fm_U$. Note that $\gamma$ must divide each coefficient $e_j$ and $p_je_j$, since $\gamma$ divides $|c_{2i-1}|$. Write
\begin{equation}\label{eq:case2}
m(|c_{2i-1}|) \sigma_{U,m} f( z_i) = \sum_{j \in J_1}k_j(e_j/\gamma) \lambda'_j +  \sum_{j \in J_2} (p_je_j)/\gamma \lambda'_j.
\end{equation}
Fix $j \in J_1$. If $e_j = |c_{2i-1}|$, then tautologically $m(e_j) = e_j/\gamma$. If $e_j \lebang |c_{2i-1}|$, then $m(e_j)$ is in $m(|c_{2i-1}|) \fm_U$ as before; moreover, $e_j/\gamma \lebang |c_{2i-1}|/\gamma = m(|c_{2i-1}|)$, so this holds for $e_j/\gamma$ also. Hence the two terms in \eqref{eq:target} and \eqref{eq:case1} corresponding to $j$ are both congruent to zero modulo $m(|c_{2i-1}|) \fm_U$. A similar argument holds for $j \in J_2$.
\end{enumerate}
We now consider \eqref{eq:Valmostchain}. We have
\begin{align*}
	\d_V f_{U,m}(r')  &\equiv \d_V \sigma_{U,m} f(r) \mod \fm_U \\
			&\equiv \sigma_{U,m} f \d_V (r) \mod \fm_U \\
			&\equiv  f_{U,m} s_{U,m} \d_V (r) \mod \fm_U \\
			&\equiv f_{U,m} \d_V (s_{U,m} (r)) \mod \fm_U \\
			&\equiv f_{U,m} \d_V (r') \mod \fm_U
\end{align*}
for any $r' \in \{ w', y'_i, z'_i \}$, where the first equivalence is Definition~\ref{def:sum}, the second holds since $\d_V$ commutes with $\sigma_{U,m}$ and $f$, the third holds by \eqref{eq:sigmaf}, and the fourth is due to the fact that $\d_V$ and $s_{U,m}$ commute.
\end{proof}

Having completed the bulk of the work, we leave the verification of the remaining hypotheses of Lemma~\ref{lem:almosttolocal} to the reader. The proofs of these are exactly the same as those of \cite[Lemma 8.19]{DHSTmoreconcord} and \cite[Lemma 8.20]{DHSTmoreconcord}. The result is summarized in the following:

\begin{lemma}\label{lem:shiftlocalmap}
Let $C_1$, $C_2$, and $C_3$ be standard complexes, let $f \co C_3 \to C_1 \otimes C_2$ be a  local map, and let $m$ be a shift map.   There exists a local map
\[ g_U \co \shift_{U,m} (C_3) \to \shift_{U,m} (C_1) \otimes \shift_{U,m} (C_2). \]
Similarly, there exists a local map
\[ g_V \co \shift_{V,m} (C_3) \to \shift_{V,m} (C_1) \otimes \shift_{V,m} (C_2). \]
\end{lemma}
\begin{proof}
Omitted; see \cite[Lemma 7.21]{DHSTmoreconcord} and \cite[Lemma 7.25]{DHSTmoreconcord}.
\end{proof}

We now finally turn to the proof of Theorem \ref{thm:shift}:

\begin{proof}[Proof of Theorem \ref{thm:shift}]
Let $C_3$ be a standard knotlike complex with a local equivalence $C_3 \sim C_1 \otimes C_2$ which is realized by a local map $f \co C_3 \to  C_1 \otimes C_2$. Without loss of generality, assume $C_1$ and $C_2$ are standard complexes.  By Lemma~\ref{lem:shiftlocalmap}, there exists a local map $g_U \co \shift_{U,m} (C_3) \to \shift_{U,m} (C_1) \otimes \shift_{U,m} (C_2)$. Hence
\begin{equation}\label{eq:shiftineq}
	\shift_{U,m} (C_3) \leq \shift_{U,m} (C_1) \otimes \shift_{U,m} (C_2).
\end{equation}
Dually, there exists a local equivalence $C^\vee_3 \sim C^\vee_1 \otimes C^\vee_2$, and the same reasoning shows
\begin{equation}\label{eq:shiftineqdual}
	\shift_{U,m} (C^\vee_3) \leq \shift_{U,m} (C^\vee_1) \otimes \shift_{U,m} (C^\vee_2). 
\end{equation}
Dualizing \eqref{eq:shiftineqdual}, applying Lemma~\ref{lem:shiftdual}, and combining with \eqref{eq:shiftineq}, we conclude that
\[ \shift_{U,m} (C_1) \otimes \shift_{U,m} (C_2) \leq \shift_{U,m} (C_3) \leq \shift_{U,m} (C_1) \otimes \shift_{U,m} (C_2). \]
Thus we obtain a local equivalence
\[ \shift_{U,m} (C_3) \sim \shift_{U,m} (C_1) \otimes \shift_{U,m} (C_2). \]
A similar argument holds for $\shift_{V, m}$. Since $\shift_m = \shift_{V,m} \circ \shift_{U,m}$, the desired statement for the composition holds also.
\end{proof}

\subsection{Proof of Theorem \ref{thm:homs}}

We now finally prove Theorem~\ref{thm:homs} by writing the $\varphi_\mu^U$ in terms of the shift homomorphisms discussed above. These are tied together via the following auxiliary homomorphism:

\begin{definition}\label{def:P}
Let $C$ be a knotlike complex. Recall from Definition~\ref{def:knotlike-complex} that 
\[
H_*(C \otimes_\stair \ru)/(\ru\text{-torsion}) \cong \ru
\]
via an absolutely $\gr_2$-graded, relatively $\gr_1$-graded isomorphism. Define $P_U(C) \in \Z$ to be the $\gr_1$-grading of the element $1 \in \ru$ under the above isomorphism. Similarly, define $P_V(C) \in \Z$ to be the $\gr_2$-grading of the element $1 \in \rv$ under the isomorphism
\[
H_*(C \otimes_\stair \rv)/(\rv\text{-torsion}) \cong \rv.
\]
\end{definition}


It is clear that $P_U$ and $P_V$ are invariants of the local equivalence class of $C$. The fact that these are homomorphisms follows from proof of Lemma~\ref{lem:product-check}: if $C_1$ and $C_2$ are knotlike complexes, then $C_1 \otimes C_2$ is a knotlike complex, and the isomorphism $H_*((C_1 \otimes C_2) \otimes_\stair \ru)/(\ru\text{-torsion}) \cong \ru$ has $\gr_1$-grading shift given by the sum of the $\gr_1$-grading shifts for $C_1$ and $C_2$.

By Corollary~\ref{cor:splitting}, any knotlike complex $C$ is homotopy equivalent to a direct sum $C(a_1, \dots, a_n) \oplus A$ such that
\[
H_*(A \otimes_\stair \ru)/(\ru\text{-torsion}) = H_*(A \otimes_\stair \rv)/(\rv\text{-torsion}) = 0.
\]
It follows that for the purposes of computing $P_U$ and $P_V$, we may replace $C$ with its standard complex representative $C(a_1, \dots, a_n)$. A direct analysis in this case then shows that $P_U(C) = \gr_1(x_n)$, where $x_n$ is the final standard complex generator. Similarly, $P_V(C) = \gr_2(x_0)$. Thus we have:

\begin{lemma}
Let $C$ be a knotlike complex with standard complex representative $C(a_1, \ldots, a_n)$. Then
\begin{equation}\label{eq:PU}
P_U(C) = \sum_{\mu \in \Gamma_+(\R_U)} \gr_1(\mu) \varphi_\mu^U(C) + \sum_{\nu \in \Gamma_+(\R_V)} \gr_1(\nu) \varphi_\nu^V(C) + \sum_{i=1}^n \sgn a_{i}
\end{equation}
and
\begin{equation}\label{eq:PV}
- P_V(C) = \sum_{\mu \in \Gamma_+(\R_U)} \gr_2(\mu) \varphi_\mu^U(C) + \sum_{\nu \in \Gamma_+(\R_V)} \gr_2(\nu) \varphi_\nu^V(C) + \sum_{i=1}^n \sgn a_{i}.
\end{equation}
\end{lemma}
\begin{proof}
This follows immediately from \eqref{eq:gr1diff} and \eqref{eq:gr2diff}, together with the fact that the $\varphi_\mu^U$ and $\varphi_\nu^V$ count parameters in the standard complex representative of $C$.
\end{proof}

\begin{proof}[Proof of Theorem \ref{thm:homs}]
We consider $\smash{\varphi_\mu^U}$; the proof for $\smash{\varphi_\nu^V}$ is similar. Let $C_1 = C(a_1, \dots, a_{n_1})$ and $C_2 = C(b_1, \dots, b_{n_2})$ be two arbitrary but fixed standard complexes. Let the tensor product of $C_1 \otimes C_2$ be locally equivalent to the standard complex $C_3 = C(c_1, \dots, c_{n_3})$. Clearly, for $\mu$ outside of the finite set $S$ consisting of the odd parameters of these complexes, we have that $\smash{\varphi_\mu^U(C_1) + \varphi_\mu^U(C_2)} = \smash{\varphi_\mu^U(C_1 \otimes C_2)} = 0$. We thus need to show that
\begin{equation}\label{eq:induction}
\varphi_\mu^U(C_1) + \varphi_\mu^U(C_2) = \varphi_\mu^U(C_1 \otimes C_2)
\end{equation}
for $\mu \in S$. Proceed by strong induction on the elements of $S$. Let $M \in S$ and assume that we have shown \eqref{eq:induction} for all $\mu \lebang M$. Fix any element $\mu_0 \in \fm_U$ such that $\gr_U(\mu_0) \neq 0$, as guaranteed by Definition~\ref{def:gradingnontrivial}. Define a shift map $m_U \co \Gamma_+(\R_U) \rightarrow \Gamma_+(\R_U)$ by
\[
m_U(\mu) = 
\begin{cases} 
\mu \quad &\text{if } \mu \gebang M \\
\mu_0 \cdot \mu \quad &\text{if } \mu \leqbang M
\end{cases}
\]
and extend this to a shift map $m = m_U \sqcup m_V$ by defining $m_V$ to be the identity. The trick will be to consider the quantity $P(\shift_{U, m}(C_1 \otimes C_2)) - P(C_1 \otimes C_2)$. On one hand, using \eqref{eq:PU} and the fact that $\smash{\varphi_{m(\mu)}^U(\shift_{U,m}(C)) = \varphi_{\mu}^U(C)}$, we have
\begin{align*}
&\phantom{\sum} P(\shift_{U, m}(C_1 \otimes C_2)) - P(C_1 \otimes C_2) \\
&= \sum_{\mu \in \Gamma_+(\R_U)} \Big( \gr_1(m(\mu)) - \gr_1(\mu) \Big) \varphi_\mu^U(C_1 \otimes C_2) \\
&= \sum_{\mu \leqbang M} \gr_1(\mu_0) \varphi_\mu^U(C_1 \otimes C_2),
\end{align*}
where in the last line we have used the equality $\gr_1(m(\mu)) = \gr_1(\mu_0 \cdot \mu) = \gr_1(\mu_0) + \gr_1(\mu)$ for $\mu \leqbang M$. On the other hand, using the fact that $P$ and $\shift_{U, m}$ are both homomorphisms, a similar argument shows that the above quantity is equal to
\begin{align*}
&\phantom{\sum} P(\shift_{U, m}(C_1)) + P(\shift_{U, m}(C_2)) - P(C_1) -P(C_2) \\
&= \sum_{\mu \leqbang M} \gr_1(\mu_0) \varphi_\mu^U(C_1) + \sum_{\mu \leqbang M} \gr_1(\mu_0) \varphi_\mu^U(C_2) \\
&= \sum_{\mu \leqbang M} \gr_1(\mu_0) \Big(\varphi_\mu^U(C_1) + \varphi_\mu^U(C_2)\Big).
\end{align*}
Equating these and using the inductive hypothesis, we see that
\[
\gr_1(\mu_0) \varphi^U_M(C_1 \otimes C_2) = \gr_1(\mu_0) \Big(\varphi^U_M(C_1) + \varphi^U_M(C_2)\Big).
\]
Since $\gr_1(\mu_0) \neq 0$, this establishes the inductive step and completes the proof.
\end{proof}

\section{An algebra for knot Floer homology}\label{sec:algebra-for-knot-floer}

We now specialize to the ring $\X$ discussed in Section~\ref{sec:preliminary}. For the convenience of the reader, we review some aspects of Section~\ref{sec:preliminary} here. Recall that $\X = \stair(\R_U, \R_V)$, where
\[
\R_U=\F[U_B,\{W_{B,i}\}_{i\in \Z}]/(\{U_BW_{B,i}=W_{B,i+1}\}_{i \in \Z})
\] 
and
\[
\R_V=\F[V_T,\{W_{T,i}\}_{i\in \Z}]/(\{V_TW_{T,i}=W_{T,i+1}\}_{i \in \Z}).
\]
Abusing notation slightly, the homogeneous elements of $\R_U$ can also be written as $\smash{U_B^i W_{B,0}^j}$, where
\begin{equation}\label{eq:ruregion2}
(i, j) \in (\Z \times \Z^{\geq 0}) - (\Z^{< 0} \times \{0\}).
\end{equation}
Note that $i \geq 0$ if $j = 0$, but otherwise we allow the exponent of $U_B$ to be negative. We parameterize the elements of $\R_U$ simply as points $(i, j)$ in the region \eqref{eq:ruregion2}. For homogeneous elements of $\R_V$, we similarly write $(i, j)$ to represent $\smash{V_T^iW_{T,0}^j}$, where $(i, j)$ again lies in \eqref{eq:ruregion2}. See Section~\ref{sec:preliminary} and Figure~\ref{fig:Xring} for a review of $\X$. 

Our goal for this section will be to establish that the (grading-shifted) knot Floer complex of a knot in a homology sphere gives us a symmetric knotlike complex over $\X$. Following Section~\ref{sec:homs}, we then obtain a homomorphism for each element of $\smash{\Gamma_+(\R_U)}$. We denote these by
\[
\varphi_{i,j} \co \Chatz \rightarrow \Z
\]
where $(i, j)$ lies in \eqref{eq:ruregion2} (minus the origin) and represents $\smash{U_B^i W_{B,0}^j}$.

\begin{remark}\label{rem:fuvtox}
Importantly, there is a morphism of bigraded $\F$-algebras $\F[U,V]\to \X$ sending
\begin{equation}\label{eq:fuvtox}
U\mapsto U_B+W_{T,0} \quad \text{and} \quad V\mapsto V_T+W_{B,0}. 
\end{equation}
The substitution \eqref{eq:fuvtox} will be helpful for understanding the relationship between complexes over $\F[U, V]$ and complexes over $\X$. For any complex $C$ over $\F[U, V]$, write $C_\X = C \otimes_{\F[U, V]} \X$, where the action of $\F[U, V]$ on $\X$ is defined using \eqref{eq:fuvtox}. 

Now set $U_B = 1$ in $C_\X$. Since $U_B \cdot \fm_V = 0$, this sets $\fm_V$ to zero. The result is a complex over the ring $\F[W_{B,0}]$ with a single grading $\gr_2$; note that setting $U_B = 1$ collapses the $\gr_1$-grading. Composing the morphism $\eqref{eq:fuvtox}$ with this quotient sends $U$ to 1 and $V$ to $W_{B,0}$. Hence it is easily checked that we have a graded isomorphism of chain complexes
\begin{equation}\label{eq:invertsetone}
C/(U - 1) \cong C_{\X}/(U_B - 1).
\end{equation}
Here, the former complex is a module over $\F[V]$ and the latter complex is a module over $\F[W_{B,0}]$, and the isomorphism \eqref{eq:invertsetone} identifies $V$ with $W_{B,0}$.

Now consider the homology of $C_{\X}/(U_B - 1)$. Since inverting $U_B$ sets $\fm_V$ to zero, we may as well replace $C_{\X}$ with $C_{\X} \otimes_\X \R_U$. According to the structure theorem of Remark~\ref{rem:structure}, the homology of $C_{\X} \otimes_\X \R_U$ is isomorphic to some number of copies of $\R_U$ plus some number of torsion summands $\smash{\R_U/(U_B^{i_k}W_{B_0}^{j_k})}$. Thus, up to grading shifting individual summands,
\begin{equation}\label{eq:structureequation}
H_*(C_{\X}/(U_B - 1)) \cong \F[W_{B_0}]^m \oplus \left( \bigoplus_{k} \F[W_{B_0}]/(W_{B_0}^{j_k}) \right).
\end{equation}
Note that if $C = \CFK(Y, K)$ is a knot Floer complex, then 
\begin{equation}\label{eq:hfmiso}
H_*(C/(U - 1)) \cong \HFm(Y), 
\end{equation}
where both sides are viewed as modules over $\F[V]$. Thus if $C$ is the knot Floer complex of a knot, then \eqref{eq:invertsetone} and \eqref{eq:hfmiso} show that $H_*(C_{\X}/(U_B - 1))$ is determined.
\end{remark}

In what follows, we use $C[(a,b)]$ to denote the complex $C$ shifted by bigrading $(a, b)$, so that an element $x \in C$ which previously had bigrading $(a, b)$ now has bigrading $(0,0)$ in $C[(a,b)]$.

\begin{lemma}\label{lem:cfkxknotlike} 
Let $K$ be a knot in an integer homology sphere $Y$. Then the grading-shifted complex $\CFK_\X(Y, K)[(d(Y), d(Y))]$ is a knotlike complex over $\X$.
\end{lemma}
\begin{proof}
Let $C = \CFK(Y, K)$ and $C_\X = \CFK_\X(Y,K)$. Our first claim is that $H_*(C_\X \otimes_\X \R_U)/\R_U\text{-torsion} \cong \R_U$. Due to the structure theorem, this is equivalent to showing that $m = 1$ in \eqref{eq:structureequation}, which follows immediately from \eqref{eq:invertsetone} and \eqref{eq:hfmiso} (together with the fact that $Y$ is a homology sphere). Since \eqref{eq:invertsetone} and \eqref{eq:hfmiso} are absolutely $\gr_2$-graded, the single tower $\F[W_{B,0}]$ in \eqref{eq:structureequation} is generated by an element in $\gr_2$-grading $d(Y)$. Thus the generator of the $\R_U$-tower in $H_*(C_\X \otimes_\X \R_U)/\R_U\text{-torsion}$ has $\gr_2$-grading $d(Y)$. An analogous argument replacing $U$ with $V$ shows that $H_*(C \otimes_\X \R_V)/\R_V\text{-torsion} \cong \R_V$, with the generator of this tower having $\gr_1$-grading $d(Y)$. Shifting gradings to satisfy the normalization convention of Definition~\ref{def:standard} gives the claim. 
\end{proof}

Moreover, a local equivalence between $\CFK(Y_1,K_1)$ and $\CFK(Y_2,K_2)$ gives a local equivalence between $\CFK_\X(Y_1,K_1)$ and $\CFK_\X(Y_2,K_2)$:

\begin{lemma}\label{lem:x-local-equivalence}
If $K_i\in Y_i$ for $i=1,2$ are homology concordant, then we have a local equivalence
\[
\CFK_\X(Y_1,K_1)[d(Y_1), d(Y_1)] \sim \CFK_\X(Y_2,K_2)[d(Y_2), d(Y_2)].
\]
\end{lemma}
\begin{proof}
If $K_1$ is homology concordant to $K_2$, then we have absolutely chain graded maps $f$ and $g$ which induce isomorphisms between 
\[
H_*(\CFK(Y_i, K_i) \otimes_{\F[U,V]} \F[U, U^{-1}, V, V^{-1}])
\]
for $i = 1, 2$. A similar argument as in Remark~\ref{rem:fuvtox} shows these are isomorphic to
\[
H_*(\CFK_\X(Y_i, K_i) \otimes_\X \loc(\R_U))
\]
for $i = 1, 2$. Hence $f$ and $g$ induce isomorphisms on $\CFK_\X$  after inverting $\R_U$. It follows that $f$ and $g$ map $\R_U$-nontorsion elements of $H_*(\CFK_\X(Y_i, K_i) \otimes_\X \R_U)$ to $\R_U$-nontorsion elements. Because $f$ and $g$ are absolutely graded, this implies $f$ and $g$ must induce isomorphisms on $H_*(\CFK_\X(Y_i, K_i) \otimes_\X \R_U)/\R_U\text{-torsion}$. A similar argument holds with $V$ in place of $U$. Since $Y_1$ is homology cobordant to $Y_2$, we moreover have $d(Y_1) = d(Y_2)$, so the claim holds after shifting gradings as well.
\end{proof}


The well-known fact that $\CFK(Y, K)$ is symmetric also translates into the fact that the associated complexes over $\X$ are symmetric, in the sense of Definition~\ref{def:symmetric}:

\begin{lemma}
Let $K$ be a knot in an integer homology sphere $Y$. Then the knotlike complex $\CFK_\X(Y,K)[d(Y), d(Y)]$ is symmetric with respect to the involution on $\X$ sending
\[\xi(U_B)=V_T,\qquad \xi(W_{B,i})=W_{T,i},\qquad \xi(V_T)=U_B,\qquad \xi(W_{T,i})=W_{B,i}.\]
\end{lemma} 
\begin{proof}
Recall that (as in the construction of involutive knot Floer homology; see \cite{HM:involutive}) there is a skew-graded homotopy equivalence $\iota\colon \CFK(Y,K)\to \CFK(Y,K)$ so that $\iota(Ux)=V\iota(x)$ and $\iota(Vx)=U\iota(x)$.  Let $\xi \colon \F[U,V]\to \F[U,V]$ be the skew-graded isomorphism given by $\xi(U)=V$ and $\xi(V)=U$; then the map \eqref{eq:fuvtox} intertwines this with the involution $\xi$ given in the statement of the lemma. We can view $\iota$ as a $\F[U,V]$-equivariant, absolutely graded homotopy equivalence from $\CFK(Y,K)$ to $\xi^*\CFK(Y,K)$.  Tensoring with $\X$ we obtain an $\X$-equivariant, absolutely graded homotopy equivalence from $\CFK_\X(Y,K)$ to $\xi^*\CFK_\X(Y,K)$. 
\end{proof}

We thus finally obtain:

\begin{proof}[Proof of Theorem \ref{thm:localequivlex}:]
This follows directly from the properties of $\X$ described above along with Theorem \ref{thm:char}.
\end{proof}

\section{Applications}\label{sec:vanishing}


In this section, we prove some of the applications listed in the introduction. We begin by establishing vanishing conditions for the homomorphisms $\varphi_{i,j}$ in the case that $K$ is a knot in an integer homology sphere $L$-space. First consider a standard complex $C(b_1,\ldots,b_n)$ with $\smash{|b_{2k-1}| = U_B^{i_k}W_{B,0}^{j_k}}$. As in \eqref{eq:structureequation}, we have a $\gr_2$-graded isomorphism
\begin{equation}\label{eq:torsion-towers-from-standards}
H_*(C(b_1,\ldots,b_n)/(U_B-1)) \cong \F[W_{B_0}] \oplus \left( \bigoplus_{k} \F[W_{B_0}]/(W_{B_0}^{j_k})[\sigma_k] \right).
\end{equation}
The generator of the nontorsion tower is $[x_n]$, which lies in $\gr_2$-grading zero. The generator of the $k$th torsion tower is $[x_{2k-2}]$ or $[x_{2k-1}]$, according as $b_{2k-1} \gebang 1$ or $b_{2k-1} \lebang 1$, respectively. As in \eqref{eq:gr1diff} and \eqref{eq:gr2diff}, each grading shift $\sigma_k$ is easily calculated to be:
\begin{align*}
\sum_{j>2k-2} \sgn(b_j)+\sum_{j>2k-2} \gr_2(b_j) &\qquad \mbox{ if } b_{2k-1}\gebang 1 \\
\sum_{j>2k-1} \sgn(b_j)+\sum_{j>2k-1} \gr_2(b_j) &\qquad \mbox{ if } b_{2k-1}\lebang 1.
\end{align*}

\begin{theorem} \label{thm:vanishing}
Let $K$ be a knot in an integer homology sphere $L$-space $Y$. Then for any $j > 0$, the homomorphism $\varphi_{i, j}(K)$ vanishes. Thus the homomorphisms $\varphi_{i,j}$ for $j > 0$ descend to homomorphisms 
\[
\Chatz/\cCz\to \mathbb{Z}.
\]
\end{theorem}
\begin{proof}
Let $\CFK_\X(Y, K)[d(Y), d(Y)]$ be locally equivalent to $C(b_1, \ldots, b_n)$. Since $Y$ is an integer homology sphere $L$-space, the isomorphism \eqref{eq:invertsetone} implies that there are no torsion summands on the right-hand side of \eqref{eq:structureequation}. By Corollary~\ref{cor:splitting}, the right-hand side of \eqref{eq:torsion-towers-from-standards} appears as a summand of the right-hand side of \eqref{eq:structureequation}. Hence we must have $j_k = 0$ for each standard complex parameter $|b_{2k-1}|$. 
\end{proof}

We now show that the remaining homomorphisms $\varphi_{i, 0}(K)$ coincide with the homomorphisms defined in \cite{DHSTmoreconcord} in the case that $K$ lies in $S^3$. (With a slight modification, it is possible to generalize the results of \cite{DHSTmoreconcord} to knots in integer homology sphere $L$-spaces; then the previous statement holds replacing $S^3$ with any such space.) For this, we first show that in such a situation, local equivalence over $\X$ coincides with local equivalence over $\cR = \F[U, V]/(UV)$. Note that we have maps $\cR \rightarrow \X$ and $\X \rightarrow \cR$. The first is induced from \eqref{eq:fuvtox} by sending $U \mapsto U_B$ and $V \mapsto V_T$, while the second simply maps $U_B \mapsto U$ and $V_T \mapsto V$ (and is zero on $W_{B,0}$ and $W_{T, 0}$).

\begin{proposition}\label{prop:knots-in-l-spaces}
Let $K$ be a knot in an integer homology sphere $L$-space $Y$. Let $\CFKX(Y,K)$ be $\X$-locally equivalent to $C(b_1,\ldots,b_n)[d(Y), d(Y)]$. Then all $b_i$ lie in $\cR$ and $\CFKUV(Y,K)$ is $\cR$-locally equivalent to $C(b_1,\ldots,b_n)[d(Y), d(Y)]$.
\end{proposition}
\noindent
Here, when we say that the $b_i$ are in $\cR= \F[U, V]/(UV)$, we mean that they lie in the image of the map $\cR \rightarrow \X$. In this situation, we may consider $C(b_1,\ldots,b_m)$ as a standard complex over $\X$ or as a standard complex over $\cR$; we write $C(b_1,\ldots,b_n)_{\X}$ and $C(b_1,\ldots,b_n)_{\cR}$ to distinguish these when necessary.
\begin{proof}
We assume $d(Y)=0$ for ease of notation. The fact that the $b_i$ lie in $\cR$ is simply the proof of Theorem~\ref{thm:vanishing} (replacing $U$ with $V$ in the case of the even-index parameters). It is easily checked that
\[
\CFKUV(Y,K)=\CFKX(Y,K)\otimes_\X \cR
\]
where the action of $\X$ on $\cR$ is as defined above. Moreover, the reader may verify that $\X$-local maps descend to $\R$-local maps, so
	\[
	\CFKUV(Y,K) \simeq_{\cR} C(b_1,\ldots,b_n)_\X\otimes_\X \cR.
	\]
Finally, the fact that all the $b_i$ lie in $\cR$ implies $C(b_1,\ldots,b_n)_{\X} \otimes_\X \cR=C(b_1,\ldots,b_n)_{\cR}$. This completes the proof.
\end{proof}

\begin{proposition} 
Let $K$ be a knot in an integer homology sphere $L$-space $Y$. For any $i > 0$, we have that $\varphi_{i, 0}(K)=\varphi_i(K)$ from \cite{DHSTmoreconcord}. 
\end{proposition}
\begin{proof}
Follows immediately from Proposition \ref{prop:knots-in-l-spaces}.
\end{proof}

A slightly more refined analysis of \eqref{eq:torsion-towers-from-standards} constrains the behavior of the $\varphi_{i,j}$ for knots in Seifert fibered homology spheres:

\begin{proof}[Proof of Proposition~\ref{prop:seifert}:]
Consider each grading shift $\sigma_k$ from \eqref{eq:torsion-towers-from-standards}. Since every $\gr_2(b_j)$ is even and $n$ is even, it is clear that $\sigma_k$ is even if and only if $b_{2k-1} \gebang 1$. Note that the $k$th torsion tower has grading supported in the same parity as $\sigma_k$, with the caveat that if $j_k = 0$ then the corresponding torsion tower is empty. As in the proof of Theorem~\ref{thm:vanishing}, we have that the right-hand side of \eqref{eq:torsion-towers-from-standards} appears as a summand of $\HFm(Y)$ up to a grading shift of $d(Y)$. Now, any negative Seifert space has (minus-flavor) reduced Heegaard Floer homology concentrated in odd degrees \cite{OSplumbed}. Hence if $Y$ is a positive Seifert space, then for each $k$, we must either have $b_{2k-1} \gebang 1$ or $j_k = 0$. The argument for positive Seifert spaces is analogous.
\end{proof}

We now relate our homomorphisms to the knot invariants $\tau(Y,K)$ and $\varep(Y,K)$.  
First, we recall the definition of $\tau(Y, K)$ from \cite{OS4ball} (see also \cite{HomLidmanLevine}). Let $C = \CFK^\infty(Y,K)$ which, after choosing a filtered basis, decomposes as a direct sum $C = \oplus_{i,j \in \mathbb{Z}} C(i,j)$. For any set $X \subset \mathbb{Z}^2$, let $CX = \oplus_{(i,j) \in X} C(i,j)$. 
Let
\[
\iota_s\colon C\{i=0,j\leq s\}\to C\{i=0\}
\]
denote the inclusion map, and let $\rho \colon \widehat{\mathit{CF}}(Y) \to \mathit{CF}^+(Y)$ be the natural inclusion map $\widehat{\mathit{CF}}(Y) \simeq C\{i=0\} \to C\{ i \geq 0 \} \simeq {\mathit{CF}}^+(Y)$.  

\begin{definition}{\cite{HomLidmanLevine}}\label{def:tau-1} Let $K$ be a knot in an integer homology sphere $Y$. Define
\[
\tau(Y,K)=\min \{ s \mid \mathrm{im} (\rho_*\circ \iota_{s,*}) \cap U^N HF^+(Y)\neq 0 \qquad \forall N\gg 0\}.
\]
\end{definition}
We give an equivalent definition of $\tau(Y,K)$. Let $\rho'\colon \CF^{\leq 0}(Y) \to \widehat{\mathit{CF}}(Y)$ be the natural projection map, where $\CF^{\leq 0}(Y) = C\{i\leq 0\}$.  
 
 \begin{lemma} \label{lem:next-tau}
 	The definition of $\tau(Y,K)$ above agrees with:
 	\begin{equation*}
 	\tau(Y,K)=\min \{ s \mid \exists \ y \in \HF^{\leq 0}_{d(Y)}(Y), \ U^Ny\neq 0 \quad  \forall N>0, \mbox{ and } 0 \neq \rho'_*(y)\in \mathrm{im} (\iota_{s,*})\},
 	\end{equation*}
where $\HF^{\leq 0}_{d(Y)}(Y) = H_{d(Y)}(\CF^{\leq 0}(Y))$, where the subscript $d(Y)$ denotes the summand in grading equal to the Ozsv\'ath-Szab\'o $d$-invariant $d(Y)$.  
 \end{lemma}
A class (or underlying chain representing) $y \in \HF^{\leq 0}_{d(Y)}(Y)$ such that $ U^Ny\neq 0$ for all integers $ N>0$ will be called a tower generator.
\begin{proof}
	For now, let the expression $\tau(Y,K)$ as in the lemma be called $\tau'(Y,K)$.  We have the following diagram of exact sequences:
	\begin{equation}\label{eq:transfer}
	\begin{tikzpicture}
	\node (a0) at (0,0) {$\mathit{CF}^-(Y)$};
	\node (a1) at (3,0) {$\mathit{CF}^\infty(Y)$};
	\node (a2) at (6,0) {$\mathit{CF}^+(Y)$};
	
		\node (b0) at (0,-1) {$\mathit{CF}^{\leq 0}(Y)$};
	\node (b1) at (3,-1) {$\mathit{CF}^\infty(Y)$};
	\node (b2) at (6,-1) {$\mathit{CF}^{>0}(Y)$};

	\draw[->] (a0) -- (b0); 
	\draw[->] (a1) -- (b1);
	\draw[->] (a2) -- (b2);
	\draw[->] (a0) -- (a1);
		\draw[->] (a1) -- (a2);
			\draw[->] (b0) -- (b1);
		\draw[->] (b1) -- (b2);
	\end{tikzpicture}
	\end{equation}
	The left vertical arrow is inclusion, the middle is the identity, and the right is projection.  There is an associated diagram of exact triangles:
	
		\begin{equation*}
	\begin{tikzpicture}
	\node (a0) at (0,0) {$\mathit{HF}^-(Y)$};
	\node (a1) at (3,0) {$\mathit{HF}^\infty(Y)$};
	\node (a2) at (6,0) {$\mathit{HF}^+(Y)$};
	\node (a3) at (9,0) {$\dots$};
	
	\node (b0) at (0,-1) {$\mathit{HF}^{\leq 0}(Y)$};
	\node (b1) at (3,-1) {$\mathit{HF}^\infty(Y)$};
	\node (b2) at (6,-1) {$\mathit{HF}^{>0}(Y)$};
	\node (b3) at (9,-1) {$\dots$};

	\draw[->] (a0) -- (b0); 
	\draw[->] (a1) -- (b1);
	\draw[->] (a2) -- (b2);
	\draw[->] (a0) -- (a1)  node[pos=.5, anchor=south] {\scriptsize $\iota_-$};
	\draw[->] (a1) -- (a2) node[pos=.5, anchor=south] {\scriptsize $\pi_+$};
	\draw[->] (a2) -- (a3);
	\draw[->] (b0) -- (b1) node[pos=.5, anchor=south] {\scriptsize $\iota_{\leq 0}$};
	\draw[->] (b1) -- (b2) node[pos=.5, anchor=south] {\scriptsize $\pi_{>0}$};
	\draw[->] (b2) -- (b3);
	\end{tikzpicture}
	\end{equation*}
	It is straightforward to check that for a class $[y]\in\widehat{\mathit{CF}}(Y)$ to satisfy $0 \neq \rho_*([y]) \in U^N\mathit{HF}^+(Y)$ for all $N \gg 0$ is equivalent to $\rho_*([y])\in \mathrm{im}(\pi_+)$ (Similarly, for $y\in \HF^-(Y)$ to have $U^Ny\neq 0$ for all $N$ is equivalent to $\iota_-(y)\neq 0$).  We also note that the diagram (\ref{eq:transfer}) defines a map $\mathit{CF}^{\leq 0}(Y)\to \mathit{CF}^+(Y)$, by using that the middle vertical arrow is an isomorphism.  At the chain level, this map is $\rho\circ \rho'$, as can be checked from the definitions.  Putting this all together, if $\rho'_*([y])$ satisfies the condition in the definition of $\tau'(Y,K)$ then $\tau'(Y,K)\geq \tau(Y,K)$, since $(\rho\rho')_*([y])$ is a class as in Definition \ref{def:tau-1}.
	
	On the other hand, the following exact triangle has a morphism to the top exact triangle in the figure above:
	\[
	\mathit{HF}^-\to \mathit{HF}^{\leq 0}\to \widehat{\mathit{HF}}
	\]
	Using the connecting maps, we get the following convenient diagram::
	
			\begin{equation*}
	\begin{tikzpicture}
	\node (a0) at (0,0) {$\mathit{HF}^{\leq 0}$};
	\node (a1) at (3,0) {$\widehat{\mathit{HF}}$};
	\node (a2) at (6,0) {$\mathit{HF}^-$};

	\node (b0) at (0,-1) {$\mathit{HF}^{\infty}$};
	\node (b1) at (3,-1) {$\mathit{HF}^+$};
	\node (b2) at (6,-1) {$\mathit{HF}^{-}$};

	\draw[->] (a0) -- (b0); 
	\draw[->] (a1) -- (b1);
	\draw[->] (a2) -- (b2) node[pos=.5, anchor=west] {\scriptsize $\text{id}$};
	\draw[->] (a0) -- (a1);
	\draw[->] (a1) -- (a2) ;

	\draw[->] (b0) -- (b1) ;
	\draw[->] (b1) -- (b2) ;

	\end{tikzpicture}
	\end{equation*}
	In particular, a class $[y]$ with $\rho_*[y]$ as in Definition \ref{def:tau-1} gives rise to a class $[y']$ in $\mathit{HF}^{\leq 0}$ with nontrivial image in $\mathit{HF}^\infty$.  It remains to show that this class $[y']$ is in degree $d(Y)$.  Indeed, this follows since the image of $[y']$ in $\mathit{HF}^\infty$ is not in the image of $\mathit{HF}^-\to \mathit{HF}^{\infty}$; the latter image is given by the span of all nontrivial homogeneous classes of degree at most $d(Y)-2$.  This completes the proof.  
\end{proof}

We will reinterpret $\tau(Y,K)$ in the language of the $\mathbb{F}[U,V]$-module $\CFK(Y,K)$. First we recall some useful facts that translate between $\CFK(Y,K)$ and the filtered chain complex $C=\CFK^\infty(Y,K)$. The filtration of $\CFK(Y,K)$ satisfies:
\[
C\{i\leq i_0, j \leq j_0\}=\mathcal{G}_{(-i_0,-j_0)}\cap ((U,V)^{-1}\CFK(Y,K))_0,
\]
the latter being defined as the span over $\F[U,V]$ of generators $U^{i'}V^{j'}\mathbf{x}$ with $i'\geq -i_0$ and $j'\geq -j_0$.  The subscript of $((U,V)^{-1}\CFK(Y,K))_0$ signifies the Alexander grading zero summand of $((U,V)^{-1}\CFK(Y,K))$.  Some consequences of this are as follows.  Recall that we write $\hat{U}$ for the action of $UV$ on $\CFK(Y,K)$ (which corresponds to the action of $U$ on the flavors of $\mathit{HF}(Y)$). First, $C\{i=0\}$ is spanned, over $\mathbb{F}$, by generators $V^{j}\mathbf{x}$ so that $A(\mathbf{x})=-j$, and is identified with $(V^{-1}\CFK(Y,K)/U)_0$ (we reserve boldface variables to come from intersection points, rather than combinations thereof).  The complex $C\{i\geq 0\}$ is the quotient of $((U,V)^{-1}\CFK(Y,K))_0$ by terms $U^iV^j\mathbf{x}$, for $i>0$. Similarly, $C\{i\leq 0\}$ is identified with $(V^{-1}\CFK(Y,K))_0$ and $\hat{U}(V^{-1}\CFK(Y,K))_0=C\{i<0\}$.  The complex $C\{i=0,j\leq s\}$ is identified with the $\mathbb{F}$-span of terms $V^{j'}\mathbf{x}$ with $j'\geq -s $ and $A(\mathbf{x})=-j'$.   
\begin{proposition}\label{prop:tauUV}
Let $K$ be a knot in an integer homology sphere $Y$.  Then $\tau(Y,K)$ is the minimal Alexander grading $s$ of a cycle $\alpha \in \CFK(Y,K)/U$, such that $V^{-s}\alpha \in (V^{-1}\CFK(Y,K)/U)_0$ is the image of a tower-generator  of $(V^{-1}\CFK(Y,K))_0$.  Here, tower-generator of $(V^{-1}\CFK(Y,K))_0$ means a $UV$-nontorsion element of Maslov grading $d(Y)$. 
\end{proposition}
\begin{proof}
Let $y \in \HF^{\leq 0}_{d(Y)}(Y)$ be a class that realizes $\tau(Y,K)$, as in Lemma~\ref{lem:next-tau}, satisfying $U^Ny\neq 0 \quad  \forall N>0$ { and } $0 \neq \rho'_*(y)\in \mathrm{im} (\iota_{s,*})$. 
Note that $\rho'\colon \CF^{\leq 0}(Y) \to \widehat{\mathit{CF}}(Y)$ is surjective, so if $\rho_*'(y)\in \mathrm{im}(\iota_{s,*})$, then we can find a tower-generator chain $y' \in \CF^{\leq 0}_{d(Y)}$ so that $\rho'(y')\in \mathrm{im}(\iota_s)$.  Under the identification of $\CF^{\leq 0}_{d(Y)}$ with the $d(Y)$-graded part of $(V^{-1}\CFK(Y,K))_0$, the map $\rho' \colon \CF^{\leq 0}(Y)\to \widehat{\mathit{CF}}(Y)$   is the natural quotient map
\[
(V^{-1}\CFK(Y,K))_0\to (V^{-1}\CFK(Y,K)/U)_0.
\]
We can represent $y' \in (V^{-1}\CFK(Y,K))_0$ as a nontrivial sum of the form 
$y'=\sum V^{-A(\mathbf{x})}\mathbf{x}+U\zeta$ for some $\zeta\in V^{-1}\CFK(Y,K)$,  where each of the $\mathbf{x}$ has $A(\mathbf{x})\leq s$ since $\rho'(y') = \sum V^{-A(\mathbf{x})} \mathbf{x} \in \mathrm{im}(\iota_s)$. Moreover, $y'$ generates $UV$-nontorsion element in grading $d(Y)$ of $H_*((V^{-1}\CFK(Y,K))_0)\cong\HF^-(Y)$.

We define an Alexander grading on $V^{-1}\CFK(Y,K)/U$ by defining the homogeneous classes to be the classes in $V^{-1}\CFK(Y,K)/U$ with homogeneous lift in $V^{-1}\CFK(Y,K)$, with grading given by the lift.  It is straightforward to check that the grading of a homogeneous lift of a nonzero class in $V^{-1}\CFK(Y,K)/U$ is independent of the choice of lift.

Given the element $y'$, we have that $V^{s}(y'\bmod{U})$ is an element of $\CFK(Y,K)/U$ with Alexander grading $s$. The class $[y'\bmod{U}]\in H_*((V^{-1}\CFK(Y,K)/U)_{0})$ is the image of a tower-generator of $H_*((V^{-1}\CFK(Y,K))_0)$.    In fact, $(y' \bmod{U})$  is the image of the tower-generator $y' \in (V^{-1}\CFK(Y,K))_0$.  

Conversely, suppose that the cycle $\beta\in (\CFK(Y,K)/U)_{s}$ satisfies the condition that the cycle $V^{-s}\beta \in (V^{-1}\CFK(Y,K)/U)_0$ is the image of a tower generator $V^{-s}\hat{\beta}\in (V^{-1}\CFK(Y,K))_0$.  
We can express $\beta = \sum V^{s-A(\mathbf{x})} \mathbf{x}$ where $s-A(\mathbf{x}) \geq 0$ for all $\mathbf{x}$ appearing in the nontrivial sum, and $V^{-s}\hat\beta = \sum V^{-A(\mathbf{x})} \mathbf{x} + U \zeta$. Letting $y' = V^{-s} \hat\beta$, we have that $\rho'(y') = \sum V^{-A(\mathbf{x})} \mathbf{x}$ is in the image $\mathrm{im}(\iota_s)$ and that $y'$ is $U$-nontorsion. Thus, $\tau(Y,K) \leq s$. 
\end{proof}

Next we prove Proposition \ref{prop:tau}, showing that the homomorphisms recover the $\tau(Y,K)$ invariant. The proof uses the reinterpretation of $\tau(Y,K)$ in Proposition~\ref{prop:tauUV} in terms of the $\mathbb{F}[U,V]$-module knot Floer chain complex $\CFK(Y,K)$. 

\begin{proof}[Proof of Proposition \ref{prop:tau}.] A similar argument as in the previous section shows that we have a bigraded isomorphism
\[
V^{-1} \CFK(Y,K) \cong V_T^{-1} \CFKX(Y, K).
\]
Using Lemma~\ref{cor:splitting}, $\CFKX(Y, K)$ has a summand given by some grading-shifted standard complex $C(b_1, \ldots, b_n)[d(Y), d(Y)]$. Under the above isomorphism, any $x$ which realizes $\tau(Y,K)$ must be supported by the grading-shifted generator $x_0'$, where we write $x_0'$ to denote the grading-shifted version of the usual standard complex generator $x_0$ in $C(b_1, \ldots, b_n)$. Thus $\tau(Y,K)$ is equal to half of
\[
\gr_1(x_0') -\gr_2(x_0') = (\gr_1(x_0) - d(Y)) - (\gr_2(x_0) - d(Y)) = - \gr_2(x_0).
\]
Since our standard complex is symmetric, we have that $\sgn(b_{n+1-i}) = - \sgn(b_i)$ and that $\gr_2(b_{n+1-i}) = - \gr_1(b_i)$. Applying \eqref{eq:gr2diff} thus yields
\[
- \gr_2(x_0) = \sum_{i \text{ odd}} (\gr_2(b_i) - \gr_1(b_i)).
\]
Recall that in Definition~\ref{def:rureformulation}, the pair $(i, j)$ represents an element of $\R_U$ of grading $(-2i, -2j)$. The claim follows.
\end{proof}


\begin{proof}[Proof of Corollary \ref{cor:1.9}.] The condition $\varep(Y,K)=0$ is equivalent to the assertion that $\CFK_\X(Y,K)$ is locally equivalent to a complex $C(b_1,\ldots,b_n)$, where $|b_1|$ is not of the form $(i, 0)$, unless $\CFK_\X(Y,K)$ is locally equivalent to $\X= C(0)$.  However, the condition on $\tau(Y,K)$ guarantees that $\CFK_\X(Y,K)$ is not locally equivalent to $\cC(0)$ by Proposition \ref{prop:tau}.  The corollary follows from Proposition~\ref{prop:knots-in-l-spaces}.
\end{proof}

We next prove the results on homology concordance genus and homology concordance unknotting number from the introduction. The homology concordance genus is
\[
g_{H,c}(Y, K)=\min_{(Y', K') \mbox{ homology concordant to } (Y,K)} g_3(Y', K'),
\]
where $g_3(Y,K)$ denotes the minimum genus of a compact orientable surface $S$ embedded in $Y$ with boundary $\partial S = K$.  
The homology concordance unknotting number is
\[
u_{H,c}(Y, K)=\min_{(Y', K') \mbox{ homology concordant to }  (Y,K)} u(Y',K'),
\]
where $u(Y,K)$ is the minimum number of crossing changes to change $K$ to the unknot in $Y$, taken over all possible diagrams for $K$.   Let 
\begin{align*} 
N (Y,K)&=\sup_{(i,j)\mid \varphi_{i,j}(Y,K)\neq 0} |i-j|.
\end{align*}  
We claim that $g_{H,c}(Y,K) \geq N(Y,K)/2$ and $u_{H,c}(Y,K) \geq N(Y,K)$.

\begin{proof}[Proof of Proposition~\ref{prop:genera-bounds} (1).]
Let $K$ be a knot in an integer homology sphere $Y$. Fix any $(i, j)$ and suppose that $\varphi_{i,j}(K) \neq 0$. Then there exist generators $x$ and $y$ in $\CFKX(Y,K)$ such that 
\[
\d_U x= U_B^{i}W_{B,0}^j y.
\]
Moreover, we may take $x$ and $y$ to be standard complex generators, so $x$ and $y$ do not lie in $(\fm_U, \fm_V)$. Since we may assume our complex is reduced, both $x$ and $y$ survive to be nontrival cycles in $\CFKX(Y,K)/(\fm_U, \fm_V)$. Note that
\[
\gr(x)-\gr(y)=(-2i+1,-2j+1),
\]
so the difference in Alexander gradings is $|A(x)-A(y)|=|i-j|$. Now, it is clear that we have a bigraded isomorphism of $\F$-vector spaces
\[
\widehat{\CFK}(Y,K) \cong \CFK(Y,K)/(U,V) \cong \CFKX(Y,K)/(\fm_U, \fm_V).
\]
Hence if $\varphi_{i,j}(K) \neq 0$, then $\widehat{\HFK}(Y, K)$ has two nonzero elements which differ in Alexander grading by $|i - j|$. But it is well-known that $\smash{\widehat{\HFK}(Y, K, s)=0}$ for all Alexander gradings $|s| > g_3(Y, K)$; hence $g_3(Y, K) \geq |i-j|/2$. Since $\varphi_{i,j}$ (and thus the property of $\varphi_{i,j}$ being nonzero) is a homology concordance invariant, this bound holds replacing $(Y, K)$ with any homology concordant pair $(Y', K')$. Taking the maximum of $|i -j|$ over all such $(i,j)$ gives the claim.
\end{proof}

\begin{remark}
The result on the homology concordance $3$-genus admits a natural generalization by considering the Alexander gradings of the entire standard complex summand; this is left to the reader.   
\end{remark}

\begin{proof}[Proof of Proposition~\ref{prop:genera-bounds} (2).]
Let $\mathbb{U}$ denote the unknot in $Y$. Let $u'(K)$ be the smallest integer $\ell $ so that there exist $\F[U,V]$-complex morphisms
\[
f\colon \CFK(Y, K)\to \CFK(Y, \mathbb{U}) \qquad \mbox{ and } \qquad g\colon \CFK(Y,\mathbb{U}) \to \CFK(Y, K)
\]
such that $f\circ g$ and $g\circ f$ are both homotopic to multiplication by $U^\ell$. By \cite{AlishahiEftekharyunknotting}, $u'(K)$ is a lower bound for the unknotting number of $K$ inside $Y$. Changing our ground ring to $\X$ using \eqref{eq:fuvtox} and modding out by $\fm_V$, we obtain induced maps of $\R_U$-modules
\[
f\colon H_*(\CFKX(Y, K)/\fm_V) \to H_*(\CFKX(Y, \mathbb{U})/\fm_V)
\]
and
\[
g\colon H_*(\CFKX(Y,\mathbb{U})/\fm_V) \to H_*(\CFKX(Y, K)/\fm_V),
\]
such that $f\circ g$ and $g\circ f$ are both equal to multiplication by $U_B^\ell$. Both of these $\R_U$-modules are isomorphic to a single copy of $\R_U$ plus some torsion summands. Our goal will be to use the existence of $f$ and $g$ to place restrictions on the lengths of these torsion summands, which will force various $\varphi_{i,j}(K)$ to be zero.

We begin by understanding $H_*(\CFKX(Y, \mathbb{U})/\fm_V)$. Note that $\CFK(Y, \mathbb{U})$ is isomorphic to $\smash{\F[U, V] \otimes_{\F[U]} \HFm(Y)}$, where the action of $U$ on $\F[U, V]$ is given by $U \mapsto UV$. As usual, $\CFm(Y)$ consists of a single $U$-nontorsion generator, together with some paired generators corresponding to $U$-torsion towers. Changing our base ring to $\X$, each $U^i$-arrow in $\CFm(Y)$ turns into an arrow decorated by $(U_B + W_{T,0})^i(V_T + W_{B,0})^i = U_B^iW_{B_0}^i + V_T^iW_{T,0}^i$. Modding out by $\fm_V$ and taking homology, we see that the torsion towers in $H_*(\CFKX(Y, \mathbb{U})/\fm_V)$ are all of the form $\R_U/(U_B^iW_{B,0}^i)$.

Now fix any $(i, j)$, and suppose $\varphi_{i,j} (K) \neq 0$. Then $H_*(\CFKX(Y, K)/\fm_V)$ has at least one torsion tower of length $\smash{U_B^i W_{B,0}^j}$. Let $x$ be a generator of such a tower. If $j = 0$ and $\ell \geq i$, then we tautologically have $\ell \geq |i-j|$. Otherwise, we have $g(f(x)) = U_B^\ell x \neq 0$, so in particular $f(x) \neq 0$. Write $f(x)$ as a multiple of a primitive element $y$ in $H_*(\CFKX(Y,\mathbb{U})/\fm_V)$. It is easily checked that $f(x)$ is not in the image of $W_{B,0}$; that is, 
\[
f(x) = U_B^m y
\]
for some $m \geq 0$. Indeed, otherwise we would have that $g(f(x)) = U_B^\ell x$ was in the image of $W_{B,0}$, which it is not. A similar argument shows that $m \leq \ell$.

We now claim that the torsion order of $y$ is $\smash{U_B^j W_{B,0}^j}$, where $j$ is the exponent of $W_{B,0}$ in the torsion order of $x$. Note that this implicitly forces $j > 0$, since $y$ is nonzero. Since $y$ is primitive, the structure of $H_*(\CFKX(Y, \mathbb{U})/\fm_V)$ shows that the torsion order of $y$ is of the form $\smash{U_B^k W_{B,0}^k}$, where $k = \infty$ if $y$ is nontorsion. Suppose $j < k$. Then for $*$ sufficiently large, 
\begin{align*}
U_B^*W_{B,0}^j f(g(y)) = U_B^{*-m} W_{B,0}^j f(g(f(x))) = U_B^{*-m+\ell} W_{B,0}^j f(x),
\end{align*}
which is zero due to the torsion order of $x$. On the other hand, this is also equal to $\smash{U_B^{* + \ell} W_{B,0}^j y}$, which is nonzero. Similarly, suppose $j > k$. Then for $*$ sufficiently large,
\[
U_B^* W_{B,0}^k g(f(x)) = U_B^{* + m} W_{B,0}^k g(y)
\]
which is zero due to the torsion order of $y$. On the other hand, this is also equal to $\smash{U_B^{* + \ell} W_{B,0}^k x}$, which is nonzero.

We now finally show $\ell \geq |i - j|$. We divide into two cases. If $i \leq j$, consider 
\[
U_B^i W_{B,0}^j f(x) = U_B^{i+m} W_{B,0}^j y.
\]
The left-hand side is zero due to the torsion order of $x$; the right-hand side then shows that $i + m \geq j$. Thus in this case, $\ell \geq m \geq j - i$. If $i \geq j$, we instead consider 
\[
U_B^jW_{B,0}^j g(y) = U_B^{j- m} W_{B, 0}^j g(f(x)) = U_B^{j-m + \ell} W_{B,0}^j x.
\]
Note that $\smash{U_B^{j- m} W_{B, 0}^j}$ is well-defined since $j > 0$. The left-hand side is zero due to the torsion order of $y$; the right-hand side then shows that $j -m + \ell \geq i$. Thus in this case, $\ell \geq \ell - m \geq i - j$. 

We thus see that the unknotting number of $K$ inside $Y$ is bounded below by $|i - j|$. Since $\varphi_{i,j}$ (and thus the property of $\varphi_{i,j}$ being nonzero) is a homology concordance invariant, this bound holds replacing $(Y, K)$ with any homology concordant pair $(Y', K')$. Taking the maximum of $|i -j|$ over all such $(i,j)$ gives the claim.
\end{proof}

\section{Computations}\label{sec:computations}

In this section, we give some computations of the $\varphi_{i,j}$ for knots in various homology spheres. Our first set of examples are based on computations from \cite{Zhou}. For any $n > 1$, let $M_n$ denote $+1$-surgery on the torus knot $T_{2, 4n-1}$. Let $Y_n = M_n \# -M_n$, and let $K_n$ denote the connected sum of the core of surgery in $M_n$ and the unknot in $-M_n$. In \cite{Zhou}, Zhou computes the knot complexes associated to $(Y_n, K_n)$ working over the ring $\F[U, V]$. 

\begin{proposition}[{\cite[Proposition 7.1]{Zhou}}] \label{prop:zhou}
The knot complex $\CFK(Y_n, K_n) $ is locally equivalent (over $\F[U, V]$) to a complex generated over by $x_0$, $x_1$, and $y$, with bigradings 
\begin{align*}
\gr(x_0) &= (2, 0) 
\\
\gr(x_1) &= (0, 2)
\\
\gr(y) &= (2n-1, 2n-1)
\end{align*}
and differential 
\[
\partial x_0  = U^{n-1}V^n y \qquad \text{and} \qquad \partial x_1 = U^nV^{n-1}y.
\]
\end{proposition}

Changing our base ring to $\X$ gives:

\begin{proposition}\label{prop:changebasis}
The knot complex $\CFKX(Y_n, K_n) $ is locally equivalent (over $\X$) to a complex generated by $a_0$, $a_1$, and $b$, with bigradings 
\begin{align*}
\gr(a_0) &= (2, 0) 
\\
\gr(a_1) &= (0, 2)
\\
\gr(b) &= (2n-1, 2n-1)
\end{align*}
with differential 
\[
\partial a_0  = V_T^n W_{T,0}^{n-1} b \qquad \text{and} \qquad \partial a_1 = U_B^nW_{B,0}^{n-1}b.
\]
\end{proposition}
\begin{proof}
Recall that to translate from complexes over $\F[U, V]$ to complexes over $\X$, we map $U \mapsto U_B + W_{T,0}$ and $V \mapsto V_T + W_{B, 0}$. Performing this substitution in Proposition~\ref{prop:zhou} and using the relations in $\X$ gives
\[
\partial x_0 = (U_B^{n-1} W_{B, 0}^{n} + W_{T,0}^{n-1} V_T^{n}) y
\]
and
\[
\partial x_1 = (U_B^n W_{B, 0}^{n-1} + W_{T,0}^n V_T^{n-1}) y.
\]
We perform the change of basis
\begin{align*}
a_0 &= x_0 + W_{B, -1} x_1 ,
\\
a_1 &= x_1 + W_{T, -1} x_0 ,
\\
b & = y. 
\end{align*}
The result follows.
\end{proof}

\begin{lemma}\label{lem:phi-calc}
For each $n > 1$, we have that $\varphi_{n,n-1}(K_n) = -1$, while $\varphi_{i,j}(K_n)=0$ for $(i,j)\neq (n,n-1)$.  
\end{lemma}
\begin{proof}
The complex of Proposition~\ref{prop:changebasis} is in fact a standard complex with preferred generators (listed in order) $a_1, b,$ and $a_0$. The result immediately follows.
\end{proof}

This completes the final step in the proof of Theorem~\ref{thm:main}:

\begin{proof}[Proof of Theorem~\ref{thm:main}]
The homomorphisms $\varphi_{i,j}$ and their properties have been constructed and verified over the last several sections. Lemma~\ref{lem:phi-calc} shows that the $\varphi_{n, n-1}$ are linearly independent and that they give a surjection
\[
\bigoplus_{n > 1} \varphi_{n, n-1} \co \Chatz \rightarrow \Z^\infty.
\]
By Theorem~\ref{thm:vanishing}, we thus in fact have a surjection
\[
\bigoplus_{n > 1} \varphi_{n, n-1} \co \Chatz/\cCz \rightarrow \Z^\infty,
\]
as desired.
\end{proof}

We also look at the examples from \cite{HomLidmanLevine}. Let $J=T_{2,-3;2,3}$ denote the $(2,3)$-cable of the left-handed trefoil, and $M$ be $+1$-surgery on $J$. Let $\smash{\tilde{J}\subset M}$ denote the knot obtained as the core of the surgery. 

\begin{proposition}\label{prop:cable}
The complex $\CFK_\X(M,\tilde{J})$ is locally equivalent to 
\[
C(-(1,1),(1,0),-(1,0),(1,1)).
\]
\end{proposition}
\begin{proof}
From \cite[Figure 12]{HomLidmanLevine}, we obtain that $\CFK(M,\tilde{J})$ is locally equivalent (over $\F[U,V]$) to the complex on generators $E,F, G,J,K$ with gradings:
\[
\gr(E)=(1,-1),\qquad \gr(F)=(0,-2),\qquad  \gr(G)=(0,0),
\]
\[
\gr(J)=(-1,1),\qquad \gr(K)=(-2,0),
\]
and with differential given by:
\[
\d F=UV E,\qquad \d E=0,\qquad \d G= UE+VJ,\qquad \d J=0,\qquad \d K=UV J.
\]
Performing the usual substitution yields:
\[
\d F=(U_B W_{B,0} + W_{T,0} V_T) E, \qquad \d K=(U_B W_{B,0} + W_{T,0} V_T) J.
\]
\[
\d E=0,\qquad \d G= (U_B + W_{T,0})E+(V_T + W_{B,0})J,\qquad \d J=0.
\]
Consider the change-of-basis
\[
F' = F + W_{B,0} G + W_{B, -1} K, \qquad K' = K + W_{T,0} G + W_{T, -1} F,
\]
\[
E' = E + W_{B, -1} J, \qquad G' = G, \qquad J' = J + W_{T, -1} E.
\]
This gives the differential
\[
\partial F' = W_{T,0}V_T E', \qquad \partial K' = U_BW_{B,0} J'
\]
\[
\partial E' = 0, \qquad \partial G' = U_B E' + V_T J', \qquad \partial J' = 0.
\]
This exhibits $\CFK_\X(M,\tilde{J})$ as a standard complex, with preferred generators (listed in order) $K', J', G', F'$, and $E'$. This gives the standard complex parameters
\[
b_1 = (U_BW_{B,0})^{-1}, \quad b_2 = (V_T), \quad b_3 = (U_B)^{-1}, \quad b_4 = (W_{T,0}V_T), 
\]
as desired.
\end{proof}

\bibliographystyle{plain}
\bibliography{bib}

\end{document}